\documentclass[11pt,reqno]{amsart}

\usepackage[american]{babel}
\usepackage{comment}
\usepackage{amsmath}
\usepackage{dsfont,mathtools,amssymb}
\usepackage{color}
\usepackage{xcolor}
\newcommand{\fs}[1]{{#1}}
\newcommand{\sk}[1]{{#1}}
\usepackage[hidelinks]{hyperref}
\mathtoolsset{showonlyrefs,showmanualtags}
\usepackage{mathrsfs}
\usepackage{stmaryrd}
\usepackage{tikz}
\usepackage[shortlabels]{enumitem}
\usepackage{array}

\newtheorem{theorem}{Theorem}[section]
\newtheorem{proposition}[theorem]{Proposition}
\newtheorem{lemma}[theorem]{Lemma}
\newtheorem{corollary}[theorem]{Corollary}
\theoremstyle{remark}
\newtheorem{remark}[theorem]{Remark}
\theoremstyle{definition}
\newtheorem{definition}[theorem]{Definition}
\newtheorem{example}[theorem]{Example}

\newtheorem*{acknowledgement}{Acknowledgements}
\newtheorem*{data}{Data availability}
\newtheorem*{conflicts}{Conflicts of interest}
\numberwithin{equation}{section}
\numberwithin{figure}{section}

\newcommand{\E}{\mathds{E}}
\renewcommand{\P}{\mathds{P}}

\newcommand{\R}{\mathbb{R}}
\newcommand{\Z}{\mathbb{Z}}
\newcommand{\N}{\mathbb{N}}
\newcommand{\IND}{{\bf 1}}
\newcommand{\dd}{{\rm d}}
\newcommand{\fstop}{\; \text{.}}
\newcommand{\comma}{\; \text{,}\;\;}

\newcommand{\tonde}[1]{\left(#1\right)}

\newcommand{\ttonde}[1]{\big(#1\big)}
\newcommand{\tttonde}[1]{(#1)}
\newcommand{\scalar}[2]{\left\langle #1\, \middle \vert\, #2 \right\rangle}
\newcommand{\tscalar}[2]{\big\langle #1 \big | #2 \big\rangle}

\newcommand{\norm}[1]{\left\lVert#1\right\rVert}

\newcommand{\triplenorm}[1]{{\left\vert\kern-0.25ex\left\vert\kern-0.25ex\left\vert #1 
		\right\vert\kern-0.25ex\right\vert\kern-0.25ex\right\vert}}
\renewcommand{\complement}{c}
\newcommand{\abs}[1]{\left\lvert#1\right\rvert}

\newcommand{\emparg}{\,\cdot\,}
\newcommand{\emp}{\varnothing}
\newcommand{\eqdef}{\coloneqq}
\newcommand{\defeq}{\eqqcolon}
\newcommand{\car}{\mathds{1}}
 
\newcommand{\eps}{\varepsilon}

\newcommand{\nscalar}[2]{\langle #1\, \vert\, #2 \rangle}
\newcommand{\set}[1]{\left\{#1\right\}}
\newcommand{\ann}{\mathfrak{a}_k}
\newcommand{\cre}{\mathfrak{a}_{\alpha,k-1}^\dagger}
\newcommand{\bbr}[2]{\llbracket #1,#2 \rrbracket}

\newcommand{\fc}{\mathfrak{c}}

\newcommand{\cA}{\ensuremath{\mathcal A}} 
\newcommand{\cB}{\ensuremath{\mathcal B}} 
 
\newcommand{\cD}{\ensuremath{\mathcal D}} 
\newcommand{\cE}{\ensuremath{\mathcal E}}

\newcommand{\cL}{\ensuremath{\mathcal L}} 
\newcommand{\cM}{\ensuremath{\mathcal M}}

\newcommand{\cP}{\ensuremath{\mathcal P}} 
 
\newcommand{\cR}{\ensuremath{\mathcal R}} 
\newcommand{\cS}{\ensuremath{\mathcal S}} 
 
\newcommand{\cU}{\ensuremath{\mathcal U}} 
 
\newcommand{\cW}{\ensuremath{\mathcal W}}

\newcommand{\cZ}{\ensuremath{\mathcal Z}} 
 
\allowdisplaybreaks
\title[One- \& two-particle spectral gap identities for SIP]{One- and	two-particle spectral gap identities for the symmetric inclusion process and related models
}

\subjclass[2020]{Primary 60K35; secondary 60J27, 05C50.}
\author{Seonwoo Kim} 
\address{\sk{Department of Mathematics, Yonsei University, Seoul, Republic of Korea}}
\email{\sk{seonwookim@yonsei.ac.kr}}
\author{Federico Sau}
\address{\fs{Dipartimento di Matematica \textquotedblleft F.\ Enriques\textquotedblright, Università degli Studi di Milano, Milano, Italy}}
\email{\fs{federico.sau@unimi.it}}

\keywords{Spectral gap; interacting particle systems;   Dirichlet distribution; symmetric inclusion process; Brownian energy process}

\begin{document} 
\maketitle
\begin{abstract}
	The symmetric inclusion process (SIP) models particles diffusing on a graph with mutual attraction. We recently showed \cite{kim_sau_spectral_2023} that, in the log-concave regime (where diffusivity dominates interaction), the spectral gap of the conservative SIP matches that of a single particle. In this paper, our main result demonstrates that this identity generally fails outside this regime, but always holds  for the non-conservative SIP, regardless of the interaction strength.
When this one-particle spectral gap identity breaks down, we derive sharp bounds for the  gap in terms of diffusivity, and reveal a two-particle spectral gap identity in the vanishing diffusivity limit. Our approach leverages the rigid eigenstructure of  SIP, refined comparisons of Dirichlet forms  for arbitrary diffusivity and particle numbers, and techniques from slow--fast system analysis. These findings extend to the dual interacting diffusion known as Brownian energy process, and shed some light on the spectral gap behavior for related Dirichlet-reversible systems on general, non-mean-field, geometries.
\end{abstract}
\thispagestyle{empty}

\section{Introduction}\label{sec1}
The \textit{symmetric inclusion process} (${\rm SIP}$) is a system of interacting particles hopping on the sites $x,y,z, \ldots \in V$ of a finite graph $G$. In its conservative version, each particle jumps from a site  $x$ to another one $y$ at rate $c_{xy}\tonde{\alpha_y+\eta_y}$, where:
\begin{itemize}
	\item  $c_{xy}=c_{yx}\ge 0$ is a symmetric weight attached to the edge $xy$;
	\item $\alpha_y> 0$ is a weight attached to the arrival  site $y$;
	\item $\eta_y\in \N_0$ denotes the number of particles sitting on  $y$.
\end{itemize} While the rate $c_{xy}\,\alpha_y$ accounts for the free motion (or diffusion) of each particle,  $c_{xy}\,\eta_y$ introduces an interaction among them by favoring jumps on top of already occupied sites.  Moreover, this  dynamics is conservative (i.e., it preserves the total number of particles),  reversible with respect to a discrete Dirichlet distribution, and irreducible as soon as the underlying graph is connected. 
We refer to Section \ref{sec2} below for the detailed definition of the process and its properties.

 For now, let us just remark that ${\rm SIP}$ arises within different contexts (e.g., as a discrete model of heat conduction \cite{giardina_duality_2007}, and  as a  multi-type Moran model from population genetics with non-mean-field interactions), and comes with closely related models from statistical physics. For instance, by setting $\alpha_y=1$ and replacing the plus sign with a minus sign  in the rates, one obtains the renowned symmetric exclusion process  \cite{spitzer_interaction_1970}. Furthermore, SIP admits a  continuous-spin counterpart, referred to as the Brownian energy process \cite{giardina_duality_2009} (see also Section \ref{sec:BEP} below). This continuous-spin model, roughly speaking, describes the energies of an instance of Kac's walk \cite{kac_foundations_1956}, in which velocities  evolve  as interacting diffusions, rather than being instantaneously updated  at 	  random \textquotedblleft collision\textquotedblright\ times. 
 
In all these models, a central question is that of determining precise  convergence rates to equilibrium, and, in the context of reversible processes,  one of the most investigated quantities for this purpose  is the so-called \textit{spectral gap}.

\subsection{One-particle spectral gap identity}For ${\rm SIP}$ on a graph $G=(V,(c_{xy})_{x,y\in V})$ endowed with site weights $\alpha=(\alpha_x)_{x\in V}$,  \begin{equation}{\rm gap}_{\rm SIP}(G,\alpha)>0
	\end{equation}denotes the corresponding spectral gap, i.e., the smallest non-zero eigenvalue of (the negative of) the infinitesimal generator (see Section \ref{sec2} for the precise definition).	
In a previous work \cite{kim_sau_spectral_2023}, we established the following bounds on ${\rm gap}_{\rm SIP}(G,\alpha)$ in terms of ${\rm gap}_{\rm RW}(G,\alpha)$, the spectral gap of the random walk (${\rm RW}$) on $G$ which  jumps from $x$ to $y$ with rate $c_{xy}\,\alpha_y$: for all graphs $G$ and site weights $\alpha$,
\begin{equation}\label{eq:kim-sau-ineq}
	\tonde{1\wedge \alpha_{\rm min}}{\rm gap}_{\rm RW}(G,\alpha)\le {\rm gap}_{\rm SIP}(G,\alpha)\le {\rm gap}_{\rm RW}(G,\alpha)\fstop
\end{equation}
Here and all throughout, $a\wedge b\eqdef \min\{a,b\}$ and  $\alpha_{\rm min}\eqdef \min_{x\in V}\alpha_x$.

Notably, as soon as $\alpha_{\rm min}\ge 1$, the inequalities in \eqref{eq:kim-sau-ineq} saturate to an identity:
\begin{equation}\label{eq:gap-identity}
	{\rm gap}_{\rm SIP}(G,\alpha)={\rm gap}_{\rm RW}(G,\alpha)\comma\qquad  \alpha_{\rm min}\ge 1\fstop
\end{equation} This one-particle reduction corresponds to a ${\rm SIP}$-version of the celebrated Aldous' spectral gap conjecture, originally formulated for  interchange and symmetric exclusion processes in the early 90s, and settled two decades later in \cite{caputo_proof_2010}.  In fact, the identity in \eqref{eq:gap-identity} is extremely  powerful, as it reduces, for every underlying graph $G$, the spectral gap of ${\rm SIP}$ --- an \fs{arbitrarily large} system, as it may consist of arbitrarily many particles --- to that of ${\rm RW}$, a much simpler Markov chain with finite state space $V$, for which several techniques to bound the spectral gap in terms of simple features of the underlying graph $G$ are known (see, e.g., \cite{saloff1997lectures,montenegro_mathematical_2005,levin2017markov,hermon2023relaxation}).  For completeness, let us emphasize that an identity like \eqref{eq:gap-identity}, while trivial for a system of independent particles, is not at all expected to hold for truly interacting systems. Indeed,   apart from the processes treated in \cite{caputo_proof_2010}, only a handful of other models were discovered to satisfy, on any geometry, a spectral gap identity involving the corresponding random walk. These models are:
\begin{itemize}
	\item the symmetric exclusion process in contact with reservoirs in \cite{salez2022universality,salez_spectral_2024};
	\item the Binomial splitting process in \cite{quattropani2021mixing,bristiel_caputo_entropy_2021}.
\end{itemize}
The identity in \eqref{eq:gap-identity} adds ${\rm SIP}$ to the short list above, provided that $\alpha_{\rm min}\ge 1$  (we shall also refer to this condition as  \textquotedblleft log-concave regime\textquotedblright, see Remark \ref{rem:log-concave} below).

One of our  main results (\textbf{Theorem \ref{th:failure-gap}}) states that, {without} this condition on  $\alpha_{\rm min}$, this identity, in general, fails:  for some graphs $G$ and site weights $\alpha$,  
	\begin{equation}\label{eq:failure-gap}
		{\rm gap}_{\rm SIP}(G,\alpha)<{\rm gap}_{\rm RW}(G,\alpha) \fstop
	\end{equation}
In other words, 
  the factor $1\wedge \alpha_{\rm min}$ in \eqref{eq:kim-sau-ineq} cannot be generally neglected when 
\begin{equation}\label{eq:regime-non-log-concave}
	\alpha_{\rm min}\in (0,1)\fstop
	\end{equation}
This naturally leads us to further investigate  this regime, with the twofold goal of:
\begin{enumerate}[(a)]
	\item deriving an alternative lower bound for ${\rm gap}_{\rm SIP}(G,\alpha)$, better capturing its dependence on $\alpha_{\rm min}$;
	\item obtaining a two-particle reduction of ${\rm gap}_{\rm SIP}(G,\alpha)$, in the asymptotic regime of vanishing diffusivity $\alpha\to 0$.	
	\end{enumerate}
We discuss these two steps  in Sections \ref{sec:intro-sharp} and \ref{sec:intro-2-particle}, respectively. 

\subsection{Sharp dependence on $\alpha_{\rm min}$} \label{sec:intro-sharp}   In order to isolate the role of $\alpha_{\rm min}$, let us introduce new site weights $\hat \alpha=\frac{\alpha}{\alpha_{\rm min}}$. 
With this notation, a simple scaling argument implies	
\begin{equation}
\label{eq:gap-RW-alpha-min}	{\rm gap}_{\rm RW}(G,\alpha)=\alpha_{\rm min}\,{\rm gap}_{\rm RW}(G,\hat \alpha)\fstop
\end{equation} 
Due to the particle interaction, the same argument does not directly apply to ${\rm SIP}(G,\alpha)$.   Instead,  
if combined with the inequalities in \eqref{eq:kim-sau-ineq}, the above identity yields
\begin{equation}\label{eq:quadratic}
 \alpha_{\rm min}^2 \, {\rm gap}_{\rm RW}(G,\hat \alpha)\le 	{\rm gap}_{\rm SIP}(G,\alpha)\le \alpha_{\rm min}\, {\rm gap}_{\rm RW}(G,\hat \alpha)\comma\qquad \alpha_{\rm min}\in (0,1)\fstop
\end{equation}
Hence, if we encode the dependence on $(G,\hat \alpha)$ through  ${\rm gap}_{\rm RW}(G,\hat \alpha)$,  the lower  bound  in \eqref{eq:quadratic} above would predict  ${\rm gap}_{\rm SIP}(G,\alpha)$ to depend  on the square of 
 $\alpha_{\rm min}$.

 Our second main result (\textbf{Theorem \ref{th:sharp-alpha-min}}) proves this guess to be wrong, showing that, in general, ${\rm gap}_{\rm SIP}(G,\alpha)$ depends linearly,  rather than quadratically as in \eqref{eq:quadratic}, on $\alpha_{\rm min}\in (0,1)$: for all graphs  $G$ and site weights $\alpha$,
 \begin{equation}\label{eq:lb-alpha-min-intro}
 	C\,\alpha_{\rm min}\le {\rm gap}_{\rm SIP}(G,\alpha)
 	\comma\qquad \alpha_{\rm min}\in (0,1)\comma
 \end{equation}
 where  $C> 0$ is a constant depending only on $G$ and $\hat\alpha=\frac{\alpha}{\alpha_{\rm min}}$.
\sk{Indeed, in Theorem \ref{th:sharp-alpha-min}, since $\alpha_{\rm min} \le 1$ and $\hat\alpha_{\rm ratio} = \alpha_{\rm ratio}$, we have
\begin{equation}
	\alpha_{\rm min}\set{	\frac1{54} \, \frac{ \hat\alpha_{\rm ratio}}{ (6e)^{1/\hat\alpha_{\rm ratio}} } \, \frac{ c_{\rm min}}{|V|^2\, {\rm diam}(G) }}	\le {\rm gap}_{\rm SIP}(G,\alpha) \comma
	\end{equation}
where the constant inside the bracket in the left-hand side depends only on the graph structure and the constant $\hat\alpha$.}
Because of the second inequality in \eqref{eq:quadratic}, the order-one dependence on $\alpha_{\rm min}$ for ${\rm gap}_{\rm SIP}(G,\alpha)$ in \eqref{eq:lb-alpha-min-intro} is sharp. As we will see, this sharpened result crucially requires to express the dependence on $G$ not through ${\rm gap}_{\rm RW}(G,\hat\alpha)$, but other features of the underlying geometry. 

\subsection{Two-particle spectral gap identity}\label{sec:intro-2-particle}
  All  considerations made so far leave  the following question unanswered: 
 given that the spectral gap of ${\rm SIP}$ does not always coincide with that of a single particle (namely, ${\rm RW}$),  does it instead coincide with the spectral gap of the $k$-particle ${\rm SIP}$, for some integer $k\ge 2$,  independent of the underlying graph $G$ and site weights $\alpha$?

Asking the same question  for ${\rm SIP}$'s \textquotedblleft instantaneously thermalized\textquotedblright\ 	variant also known as Beta-Binomial splitting process (see, e.g.,  \cite{pymar2023mixing} or Section \ref{sec:extensions})  --- Pietro Caputo conjectured (personal communication)  this spectral gap reduction to hold true with $k=2$. In words, the Beta-Binomial splitting dynamics consists of, first selecting an edge $xy$ with rate $c_{xy}$ (regardless of the particle configuration), and then letting particles sitting on $x$ and $y$ redistribute themselves according to the ${\rm SIP}$-equilibrium, restricted to that edge. Besides being clearly related to ${\rm SIP}$, 
 the Beta-Binomial splitting process is the particle analogue of a model which shows up in the literature under various names, e.g.,  the Kipnis-Marchioro-Presutti model \cite{kipnis_heat_1982}, a random walk on the simplex \cite{caputo_mixing_2019}, or (the energies of) the renowned Kac's walk \cite{kac_foundations_1956,carlen_carvalho_loss_determination_2003}. Spectral gap estimates for this model are known only on two specific geometric settings:  the complete graph \cite{carlen_carvalho_loss_determination_2003,caputo_kac2008}, and the segment with $\alpha_{\rm min}\ge 1$ \cite{caputo_mixing_2019}. 	In both cases, a spectral gap identity with $k=2$ is indeed verified.

The aforementioned conjecture translates to the context of ${\rm SIP}$ as follows: for all graphs $G$ and site weights $\alpha$, 
\begin{equation}\label{eq:conj}
	{\rm gap}_{\rm SIP}(G,\alpha)={\rm gap}_2(G,\alpha)\comma
\end{equation}
where ${\rm gap}_k(G,\alpha)$, $k\ge 1$, stands for the spectral gap of the $k$-particle ${\rm SIP}$. Clearly, we have ${\rm gap}_1(G,\alpha)={\rm gap}_{\rm RW}(G,\alpha)$ \fs{and \eqref{eq:conj} equivalently reads as ${\rm gap}_k(G,\alpha)={\rm gap}_2(G,\alpha)$, for all $k\ge 2$ (cf.\ \eqref{eq:gap-sip-def} and \eqref{eq:gap-consistency}).} Moreover, in view of \eqref{eq:gap-identity} and  ${\rm gap}_{\rm SIP}(G,\alpha)\le {\rm gap}_2(G,\alpha)\le {\rm gap}_1(G,\alpha)$ in \eqref{eq:gap-consistency},  the two-particle spectral gap identity in \eqref{eq:conj} remains to be verified only off the log-concave regime, i.e., when \eqref{eq:regime-non-log-concave} holds.

  We provide a partial answer in the affirmative to this conjecture, by determining spectral gap asymptotics as $\alpha\to 0$ \fs{along directional limits (i.e., for $\eps \to 0$ and some $\hat\alpha \in (0,\infty)^V$, $\alpha=\eps\hat\alpha$; see \eqref{eq:alpha-eps-beta})}. The precise form of the asymptotic validity of \eqref{eq:conj} is the content of  \textbf{Theorem \ref{th:gap-asymptotics1}} below. We refer to Section \ref{sec:results-asymptotics} for more details, but let us briefly remark that our asymptotic analysis highlights the key role that the two-particle dynamics plays for both results in Theorems \ref{th:gap-asymptotics1} and \ref{th:failure-gap}. We believe this point of view to be  fruitful also for other related interacting systems with Gamma-like reversible measures.  For more details, see Section \ref{sec:extensions}.

\subsection{Non-conservative case}\label{sec:intro-non-consevative}

So far, we discussed only  closed (or conservative) systems in which inclusion particles  neither get created nor annihilated. Open (or non-conservative) systems allow for this possibility,  usually modeled as systems in contact with reservoirs, see, e.g., \cite{spohn_long_1983,derrida_exact_1993-1,carinci_duality_2013-1}. 
For ${\rm SIP}$, the standard choice consists of creating a particle in $x$ with rate $\omega_x\,\theta_x\tonde{\alpha_x+\eta_x}$, and
 annihilating each particle therein (if any) with rate $\omega_x\tonde{1+\theta_x}$, where  $\omega_x$ and $\theta_x$  form a set of non-negative site parameters added to the model. For precise definitions,  see Section \ref{sec:open-SIP}.

In essence, the open ${\rm SIP}$ presents two main features: first, due to the lack of particle conservation, the configuration space now consists of a unique, countably infinite, irreducible component; second, the unique ergodic measure is reversible if the $\theta$-parameters do not depend on $x$, and non-reversible if they do depend. 
This drastic change from reversibility  to non-reversibility goes together with the emergence of a number of remarkable  phenomena (e.g., currents, long-range correlations), which distinguish equilibrium from non-equilibrium statistical mechanics, see, e.g., \cite{spohn_large_1991,de_masi_mathematical_1991,kipnis_scaling_1999,schutz_exactly2001}.

Nevertheless, as already observed in the physics literature around thirty years ago  for a large class of exactly solvable models \cite{alcaraz_droz_henkel_rittenberg_reaction_1994} (see also \cite{frassek_duality2020}), the \textquotedblleft spectrum\textquotedblright\ does not  depend on the $\theta$-parameters, and, thus, is the same  for both equilibrium and non-equilibrium systems. The quotation marks above are required when dealing with infinite configuration spaces and corresponding generators, because the full spectrum, \textit{a priori}, may include also non-eigenvalues and may depend on the underlying functional space on which the matrix/operators act. That said, this $\theta$-independence of the spectrum becomes  rigorous for systems having a finite configuration space as, for instance,  the symmetric exclusion process in contact with reservoirs (see Remark \ref{rem:SEP} below for more details).

In this article, we establish the following  results for   the non-conservative ${\rm SIP}$, for	 any underlying graph $G$ and site weights $\alpha$, $\omega$, and $\theta$:
\begin{enumerate}[(i)]
	\item \label{it:eigen1} Eigenfunctions of the purely absorbing system (corresponding to $\theta\equiv 0$) \textquotedblleft lift\textquotedblright\ to (generalized) eigenfunctions of ${\rm SIP}$ \fs{(with general $\theta$)} associated to the same eigenvalues.
	\item \label{it:eigen2} Remarkably, the spectral gap of the particle system with $\theta\equiv 0$  coincides with \begin{equation}{\rm gap}_{\rm RW}(G,\alpha,\omega)>0\comma
		\end{equation}
		the spectral gap of the one-particle system, namely, of the random walk evolving on $G$,  and killed with rate $\omega_x$ when sitting on $x\in V$ (\textbf{Theorem \ref{th:non-conservative}}).
		
	\item \label{it:eigen3} In the reversible case (i.e., when $\theta \equiv {\rm const.}$), the eigenfunctions obtained in   \ref{it:eigen1} provide an orthonormal basis in a natural $L^2$-space.
\end{enumerate}Let us emphasize that step \ref{it:eigen1} holds \fs{for general $\theta$ (i.e., in both reversible and non-reversible settings)}, while step \ref{it:eigen2} is a statement concerning the case $\theta\equiv 0$ only; establishing the claim in \ref{it:eigen3} is where we need to restrict to the reversible ${\rm SIP}$.

Putting these steps  together, our main result in this non-conservative setting (\textbf{Corollary \ref{cor:non-conservative}}) may be summarized as follows: for all graphs $G$ and site weights $\alpha$ and $\omega$, 	
\begin{equation}
	{\rm gap}_{\rm SIP}(G,\alpha,\omega,\varrho)={\rm gap}_{\rm RW}(G,\alpha,\omega)\comma\qquad \text{whenever}\ \theta \equiv \varrho>0\comma
\end{equation}
where ${\rm gap}_{\rm SIP}(G,\alpha,\omega,\varrho)$ denotes the first gap in the spectrum of the reversible ${\rm SIP}$ with $\theta\equiv \varrho>0$, compatibly with the functional setting of step \ref{it:eigen3}. 

When comparing the conservative and non-conservative settings, we remark a different behavior of ${\rm SIP}$'s spectral gap:  while a system of \textit{two particles} sharply captures  the spectral gap of the many-particle \textit{closed} ${\rm SIP}$,  the spectral gap of the \textit{open} ${\rm SIP}$ is always governed (at least when $\theta\equiv {\rm const.}$) by 	 \textit{one random walk with killing} --- with no constraint on the value of $\alpha_{\rm min}>0$.	We believe this dichotomy to hold true not just for ${\rm SIP}$, but also for a larger class of  interacting systems in statistical mechanics having Dirichlet-like  steady states \fs{(see Section \ref{sec:extensions} for further details)}.

\subsection{Organization of the paper} The rest of the paper is organized as follows. Sections \ref{sec2}--\ref{sec5} concern the closed ${\rm SIP}$.  More specifically, Section \ref{sec2} contains the model definition, the main results (Theorems \ref{th:sharp-alpha-min}--\ref{th:failure-gap}), and an outline of their proofs. In Sections \ref{sec3}--\ref{sec4}, we present the proof of Theorem \ref{th:sharp-alpha-min}; in Section \ref{sec5}, we prove the remaining two theorems from Section \ref{sec2}. Sections \ref{sec:open-SIP}--\ref{sec:proof-res} focus on ${\rm SIP}$ in contact with reservoirs, the first one of these sections detailing the model and main results, the second one presenting the proofs of the main results therein.
Finally, in Section \ref{sec:extension}, we briefly discuss some  extensions of our results  to  discrete and continuous-spin models related to ${\rm SIP}$.  In particular, in view of the isospectrality showed in \cite{kim_sau_spectral_2023} between ${\rm SIP}$ and Brownian energy process, we derive spectral gap estimates for the latter. We conclude the paper with two appendices, in which we provide the full proofs of two technical ingredients employed in Section \ref{sec5}.
\fs{For the reader's convenience, we add, at the end of the paper, a notation guide collecting the main recurrent symbols.}

\section{Conservative SIP. Setting and main results}\label{sec2}
Unless stated otherwise, all throughout the article, our underlying geometry consists of a weighted finite graph $G=(V,(c_{xy})_{x,y\in V})$, with symmetric conductances $c_{xy}=c_{yx}\ge 0$ (conventionally, $c_{xx}=0$).	We always assume the undirected graph $G$ to be connected, i.e., for all $x, y \in V$, there exists a sequence $x_0=x, x_1,\ldots, x_{\ell-1}, x_\ell=y$ in $V$ such that $\prod_{j=1}^\ell c_{x_{j-1}x_j}>0$. The length $\ell\in \N$ of the shortest sequence connecting two distinct sites $x,y\in V$ corresponds to their graph distance, ${\rm dist}_G(x,y)$. As usual, we define the diameter, ${\rm diam}(G)\in \{1,\ldots, 
|V|\}$, as the largest of all pairwise distances.

\subsection{Model} For a graph $G=(V,(c_{xy})_{x,y\in V})$,  positive site weights  $\alpha=(\alpha_x)_{x\in V}$, and $k\in \N$, ${\rm SIP}_k(G,\alpha)$  denotes the Markov chain evolving on the configuration space
\begin{equation}
	\Xi_k\eqdef \left\{\eta \in \N_0^V: |\eta|=k\right\}\comma\qquad \text{with}\ |\eta|\eqdef \sum_{x\in V} \eta_x\comma
\end{equation}
 with  infinitesimal generator given, for all $f\in \R^{\Xi_k}$, as
\begin{equation}\label{eq:gen-conservative}
	L_{G,\alpha,k}f(\eta)=\sum_{x, y \in V} c_{xy}\, \eta_x\left(\alpha_y+\eta_y\right)\tonde{f(\eta-\delta_x+\delta_y)-f(\eta)}\comma\qquad  \eta \in \Xi_k\fstop 
\end{equation}
In this formula, $ \eta-\delta_x+\delta_y\in \Xi_k$ denotes the  configuration  obtained from $\eta$ by removing a particle from $x$ (if any) and placing it on $y$.  
 Since all particle configurations in $\Xi_k$ are accessible \fs{(as the graph $G$ is connected and the weights $\alpha=(\alpha_x)_{x\in V}$ are all positive)}, ${\rm SIP}_k(G,\alpha)$ is irreducible. As a simple detailed balance computation shows, ${\rm SIP}_k(G,\alpha)$ admits a unique reversible measure  $\mu_{\alpha,k}$, given, for all $\eta\in \Xi_k$, by
\begin{equation}\label{mu-def}
	\mu_{\alpha,k}(\eta)= \frac{1}{Z_{\alpha,k}}\prod_{x\in V}\frac{\Gamma(\alpha_x+\eta_x)}{\Gamma(\alpha_x)\,\eta_x!}\comma\qquad \text{with}\ Z_{\alpha,k}\eqdef \frac{\Gamma(|\alpha|+k)}{\Gamma(|\alpha|)\,k!}\fstop
\end{equation}
Here, $\Gamma$ denotes the usual gamma function satisfying, for all $a>0$, $\Gamma(a+1)=a\, \Gamma(a)$.

\begin{remark}[Log-concavity]\label{rem:log-concave} The  measure $\mu_{\alpha,k}$ in \eqref{mu-def} is known as Dirichlet-Multinomial distribution of parameters $k\in \N$ and $\alpha=(\alpha_x)_{x\in V}$, and is the discrete analogue of the  Dirichlet distribution on the simplex $\varSigma\subset \R^V$ of probability measures, see  \eqref{eq:simplex} and \eqref{eq:dirichlet-distribution} below. The Dirichlet distribution is well known to be log-concave (in the sense that its probability density is of the form $\exp(-U)$, for some convex $U:\R^V\to \R$)	 if and only if $\alpha_{\rm min}\ge1$.	Hence, $\mu_{\alpha,k}$ is the discrete counterpart of a  log-concave measure if and only if $\alpha_{\rm min}\ge 1$. 
	More directly, $\mu_{\alpha,k}$ is a \textquotedblleft discrete log-concave measure\textquotedblright\ also in the following sense: $\mu_{\alpha,k}$ is the canonical measure (i.e., conditional on $|\eta|=k$) of the product measure $\nu_{\alpha,\varrho}$ in \eqref{eq:nu-sigma} below, whose Negative-Binomial marginal measures on $\N_0$ are log-concave (i.e., $\nu_{\alpha,\varrho}(\eta_x=\ell)^2\ge \nu_{\alpha,\varrho}(\eta_x=\ell-1)\,\nu_{\alpha,\varrho}(\eta_x=\ell+1)$, for all $\ell \ge 1$ and $x\in V$) if and only if $\alpha_{\rm min}\ge 1$.	
	
\end{remark}

\subsection{A non-asymptotic lower bound for the spectral gap}\label{sec:spectral-gap}

Because of irreducibility and reversibility, all $|\Xi_k|$ eigenvalues
 of the (negative) generator $-L_{G,\alpha,k}$ in \eqref{eq:gen-conservative} are real and non-negative, the smallest	 one being equal to zero, whereas the second  one --- referred to as spectral gap of ${\rm SIP}_k(G,\alpha)$ and shortened as ${\rm gap}_k(G,\alpha)$ --- being strictly positive. 	
 Confronting with the notation from Section \ref{sec1}, we have 
 \begin{equation}\label{eq:gap-RW}
 	{\rm gap}_{\rm RW}(G,\alpha)={\rm gap}_1(G,\alpha)\comma
 	\end{equation} because ${\rm SIP}_k(G,\alpha)$ with $k=1$ corresponds to a single, thus, non-interacting particle, abbreviated as ${\rm RW}(G,\alpha)$. All cases $k\ge 2$ describe a truly interacting system; hence,  we 	define \begin{equation}\label{eq:gap-sip-def}
	{\rm gap}_{\rm SIP}(G,\alpha)= \inf_{k\ge 2} {\rm gap}_k(G,\alpha)\fstop
\end{equation}

We are now ready to state the precise statement of the result in \eqref{eq:lb-alpha-min-intro}, concerned with the sharp dependence of ${\rm gap}_{\rm SIP}(G,\alpha)$ on $\alpha_{\rm min}=\min_{x\in V}\alpha_x$, with a special focus on the regime $\alpha_{\rm min}\in (0,1)$.

\begin{theorem}\label{th:sharp-alpha-min}
	For all graphs $G=(V,(c_{xy})_{x,y\in V})$ and site weights $\alpha=(\alpha_x)_{x\in V}$, we have
	\begin{equation}\label{eq:lb-alpha-min}
	\alpha_{\rm min}\set{	\frac1{\sk{54}} \, \frac{ \alpha_{\rm ratio}}{{\sk{(6e)}}^{\alpha_{\rm min}/\alpha_{\rm ratio}} } \, \frac{ c_{\rm min}}{|V|^2\, {\rm diam}(G) }}	\le {\rm gap}_{\rm SIP}(G,\alpha)\comma
	\end{equation} where
\begin{equation}\label{alpha-c-def}
	\alpha_{\rm ratio}\eqdef  \frac{\alpha_{\rm min}}{\alpha_{\rm max}}= \frac{\min_{x\in V}\alpha_x}{\max_{y\in V}\alpha_y}\le 1\comma\qquad c_{\rm min}\eqdef \min_{\substack{x,y\in V\\ c_{xy}> 0}}c_{xy}>0\fstop
\end{equation}
\end{theorem}
As already discussed in Section \ref{sec:intro-sharp}, this lower bound captures the correct dependence of
${\rm gap}_{\rm SIP}(G,\alpha)$ on $\alpha_{\rm min}$ for every fixed $G$, provided that
$\alpha_{\rm ratio}$ remains bounded away from zero. Moreover, the estimate is non-asymptotic, with a rather explicit constant depending only on simple geometric parameters of the graph, such as its size and diameter. We do not claim, however, that this geometric dependence is sharp.
\sk{This is fundamentally due to the fact that the $L^2$ inequalities (Lemmas \ref{lem:4.1}--\ref{lem:4.4}) used in the proof of Theorem \ref{th:sharp-alpha-min}
are designed to maintain the sharp linear dependence on $\alpha_{\rm min}$, at the expense of suboptimality caused by exploiting the geometric structure.}

\fs{To illustrate this point, take $\alpha\equiv \varepsilon\in(0,1)$, so that
$\alpha_{\rm min}=\varepsilon$ and $\alpha_{\rm ratio}=1$. For $d$-dimensional boxes or tori of side length $N$, Theorem \ref{th:sharp-alpha-min} gives
\begin{equation}
{\rm gap}_{\rm SIP}(G,\varepsilon{\bf 1})
\ge C\,\varepsilon\,N^{-(2d+1)}\comma
\end{equation}
for some $C=C(d,c_{\rm min})>0$. This is not expected to be the correct dependence on $N$. Indeed, in view of Theorem \ref{th:gap-asymptotics1} and of the proof of Theorem \ref{th:failure-gap}, the relevant small-$\varepsilon$ scale should rather be
\begin{equation}\label{eq:cox}
{\rm gap}_{\rm SIP}(G,\eps{\bf 1})
\asymp \eps/s_d(N)\comma
\qquad
 s_d(N)\eqdef
\begin{cases}
	N^{2} &\text{if}\ d=1\\
	N^{2}\log N &\text{if}\  d=2\\
	N^{d} &\text{if}\  d\ge 3\fstop
\end{cases}
\end{equation}
Thus, while Theorem \ref{th:sharp-alpha-min} is sharp in its dependence on
$\alpha_{\rm min}$, its dependence on the geometry is generally far from optimal.}

Rather than optimizing these geometric constants on specific families of graphs, we keep the underlying geometry arbitrary and pass to the regime of vanishing weights $\alpha\to0$. In this setting, we identify the relevant asymptotic object not in terms of the spectral gap of ${\rm RW}(G,\alpha)$, but of the ``second simplest system'' on the same geometry: ${\rm SIP}_2(G,\alpha)$, namely the symmetric inclusion process with just two particles.

\subsection{Spectral gap's asymptotics}\label{sec:results-asymptotics}

As already proved in \cite{kim_sau_spectral_2023} (see also Section \ref{sec3.1} below), we have,  for all graphs $G$ and site weights $\alpha$,
\begin{equation}\label{eq:gap-consistency}
\fs{	{\rm gap}_{k+1}(G,\alpha)\le {\rm gap}_{k}(G,\alpha)\comma\qquad k \ge 1\fstop}
\end{equation}
 The next result shows that,  \fs{along directional limits $\alpha \to 0$, i.e.,}
 \begin{equation}\label{eq:alpha-eps-beta}
 	\alpha =\eps \hat\alpha \comma\qquad \text{with}\  \eps\to 0\  \text{and}\  \hat\alpha = (\hat\alpha_x)_{x\in V}\in (0,\infty)^V\  \text{fixed}\comma
 \end{equation} \fs{equality is achieved in \eqref{eq:gap-consistency} for $k\ge 2$}.
 \fs{This result is formulated along fixed positive directions; see Remark
 	\ref{rem:general-alpha-limit} below for the relation with general limits $\alpha\to0$.} 
\begin{theorem}\label{th:gap-asymptotics1}
For all graphs $G=(V,(c_{xy})_{x,y\in V})$ and   site weights $\hat\alpha = (\hat\alpha_x)_{x\in V}$,  we have
\begin{equation}\label{eq:asymp-identity}
	\lim_{\eps\to 0}	\frac{{\rm gap}_k(G,\eps\hat\alpha)}{{\rm gap}_2(G,\eps\hat\alpha)} =1\comma\qquad \fs{k\ge 2}\fstop
\end{equation}
\end{theorem}In view of the definition of ${\rm gap}_{\rm SIP}(G,\alpha)$ in \eqref{eq:gap-sip-def} and of the inequalities in \eqref{eq:gap-consistency}, 
 the above limit is an  asymptotic version  of the two-particle spectral gap identity in \eqref{eq:conj}.
Let us stress that this result  holds true for any graph $G$, \fs{positive site weights $\hat\alpha$,} and integer $k\ge 2$ \fs{(trivially, for $k=2$)}. However,
 the case \fs{$k=1$} is excluded for a good reason, as shown in the following result.
\begin{theorem}\label{th:failure-gap}
For some graphs $G=(V,(c_{xy})_{x,y\in V})$ and site weights $\hat\alpha=(\hat\alpha_x)_{x\in V}$, we have
\begin{equation}\label{eq:failure-identity2}
	\lim_{\eps \to 0}\frac{{\rm gap}_2(G,\eps\hat\alpha)}{{\rm gap}_1(G,\eps\hat \alpha)}<1\fstop
\end{equation}
\end{theorem}
This inequality  provides, if combined with  \eqref{eq:gap-consistency} and the definitions in \eqref{eq:gap-RW}--\eqref{eq:gap-sip-def}, an instance of \eqref{eq:failure-gap}, namely, of the failure of the one-particle spectral gap identity in \eqref{eq:gap-identity}.
The graphs $G$ and site weights $\hat \alpha$ which we adopt to prove this result are far from being intricate. In fact, we take $G$ to be the  standard $d$-dimensional discrete torus, with $d\ge 2$ and sufficiently large size, and $\hat \alpha \equiv 1$. These are just some of the possible examples that one could exhibit for proving  \eqref{eq:failure-identity2}. Indeed, as it will become apparent from the proof, essentially any other sparse geometry for which the random walk's relaxation time (i.e., ${\rm gap}_{\rm RW}(G,\alpha)^{-1}$)   is much smaller than the  expected meeting time of two independent random walks initialized at equilibrium, will equally work.
We refer to the subsequent section  for  further explanations.

\subsection{Proofs outline}\label{sec:metastability}
Recall that an inclusion particle jumps from $x$ to a nearest neighbor $y$ at rate proportional to $\alpha_y + \eta_y$. On the one hand,  when non-zero,	the term $\eta_y\ge 1$ stands for the interaction between particles on adjacent sites, which mutually attract each other, tending to stick together. On the other hand, $\alpha_y> 0$ represents an independent mechanism of particle diffusion. 	In this sense, the regime $\alpha_{\rm min}\in (0,1)$  depicts	the situation in which, at least in some portions of the graph, particle stickiness dominates over  diffusion.

This regime becomes particularly significant when considering the limit $\alpha\to 0$. In this setting,  ${\rm SIP}(G,\alpha)$ exhibits a metastable behavior, as thoroughly studied in the last decade \cite{grosskinsky_redig_vafayi_dynamics_2013,bianchi_metastability_2017,kim_seo_condensation_2021,kim_second_2021,kim_hierarchical_2023}. These works describe the following qualitative picture, when $\alpha=\eps\hat\alpha$ as in \eqref{eq:alpha-eps-beta}. Roughly speaking, as long	 as particles are far from each other (i.e., at graph distance $\ge 2$), they perform jumps on the timescale $\eps^{-1}$. As soon as they are at distance one, they stick together in a   much shorter time, roughly of order one. These two mechanisms bring particles to meet and  pile up together in large stacks (or, condensates) in a time of order $\eps^{-1}$. Concurrently, some particles sitting on a stack would still attempt to jump to empty nearest neighboring vertices at rate $\asymp \eps$. While some of the attempts  fail, that is, the \textquotedblleft courageous\textquotedblright\ particle is instantaneously sucked back into a neighboring stack, some of these will succeed to attract sufficiently many particles, managing to move the whole stack. Due to the rates' symmetry encoded in the condition $c_{xy}=c_{yx}$, it is then part of the results in \cite{grosskinsky_redig_vafayi_dynamics_2013,bianchi_metastability_2017} to show that relevant jumps of entire stacks occur at times of order $\eps^{-1}$, and that, on this timescale, stacks are effectively approximated by independent random walks on $G$, each evolving at rate proportional to $\eps$, and capable of coalescing (possibly in a very complicated way if three or more stacks are involved) when getting at distance one from each other. Hence, the metastable relevant part of the dynamics takes place, according to this qualitative picture,  on the timescale $\eps^{-1}$.

Although this metastable picture provides a first correct intuition that ${\rm gap}_{\rm SIP}(G,\eps\hat\alpha) \asymp \eps$ as $\eps\to0$, our quantitative analysis must take into account also features not captured by usual  metastability limit theorems. 
For instance, while metastability describes macroscopic features of the system when initialized from macroscopically relevant configurations,   global functional inequalities as those determining spectral gaps require  bounds which must be uniform over the initial conditions. 
Moreover, we seek for lower bounds for ${\rm gap}_{\rm SIP}(G,\eps\hat\alpha)$ \fs{that are} independent of the system size.
However, a crucial ingredient in the metastability picture of ${\rm SIP}(G,\eps\hat\alpha)$ is that $k\in \N$, the number of particles in the system is not too large; more precisely, 
\begin{equation}\label{eq:not-huge}	\log k\ll \eps^{-1}\comma\qquad \text{as}\ \eps\ll 1\fstop
\end{equation} This condition, always assumed in previous works (see, e.g., \cite{grosskinsky_redig_vafayi_dynamics_2013,bianchi_metastability_2017,kim_second_2021,kim_hierarchical_2023}) is not just technical, but strictly required for the condensation mechanism to take place: \eqref{eq:not-huge} is a sufficient and necessary condition for the steady state of the system to charge  configurations consisting of a single stack of particles only.

\subsubsection{Outline of the  proof of Theorem \ref{th:sharp-alpha-min}} \label{sec:outline-1}
Our proof for determining the order-one dependence on $\eps \hat\alpha_{\rm min}$ of ${\rm gap}_{\rm SIP}(G,\eps\hat\alpha)$  combines three main ingredients: 
\begin{itemize}
	\item the nested eigenstructure of ${\rm SIP}$ (valid for any underlying graph $G$) as  already  exploited in \cite{kim_sau_spectral_2023};
	\item the full knowledge of the eigendecomposition in mean-field geometries and, more specifically, the fact that eigenvalues grow quadratically (neglecting multiplicities) with the particle total number;
	\item comparison inequalities of Dirichlet forms associated to the particle system on a graph $G$, against that on the complete graph. 
\end{itemize}
All three steps  are presented and combined to yield the proof of Theorem \ref{th:sharp-alpha-min} in Section \ref{sec3}, but the actual proof of the comparison inequality, longer and more technical,  is presented in Section \ref{sec4}. 

In the metastable regime \eqref{eq:not-huge}, our comparisons build on the classical distinguishing paths method as applied, e.g., in the seminal work \cite{diaconis_saloff-coste-comparison_1993} to the symmetric exclusion process. Most of the care in our context lies in devising paths in which the total \textquotedblleft cost\textquotedblright\ becomes not larger than $\eps^{-1} \, k^2$, up to constants depending only on $G$ and $\hat\alpha$. Here, the cost of each move  heavily depends on whether relocating a single particle augments, keeps constant, or lowers the total number of stacks, and whether stacks do have or not neighboring stacks.  Hence, following the metastable behavior of the system offers us a guideline to construct efficient distinguishing paths, e.g., taking care of never creating more than one extra occupied site along the path. We refer to Sections \ref{sec4.2} and \ref{sec4.3} for the details.

When \eqref{eq:not-huge} does not hold, as already mentioned, the metastability picture breaks down, and so does this approach via distinguishing paths.  In order to overcome this, we replace paths by \sk{a} more sophisticated two-dimensional \sk{\emph{surface}, which is obtained by collecting the trace of all relevant paths.} In presence of stacks consisting of a huge number of particles, \sk{this} two-dimensional \sk{surface allows} us to rigorously implement the idea of moving only a smaller portion of the stack at the time according to some suitably chosen probability, and leaving the rest untouched.  While this strategy of moving only some (and not all) particles of a stack produces an extra factor $\eps^{-1}$, which we could not afford in the regime \eqref{eq:not-huge}, this is now not problematic, as we can control this factor with a suitable function of the number of particles of the stack. For more details, see Section \ref{sec4.4}.	

As we just sketched, 	these  estimates involving Dirichlet forms yield  comparison constants which unavoidably degenerate like $1/k^2$ as the total number of particles $k\in \N$ grows. Here is where we crucially exploit the first two ingredients of the proof, in particular, the quadratic growth of eigenvalues for the mean-field system, which removes this degeneracy.

\subsubsection{Outline of the proof  of Theorems \ref{th:gap-asymptotics1} and \ref{th:failure-gap}}\label{sec:outline-2}
In order to establish the strict inequality between ${\rm gap}_{\rm SIP}(G,\eps\hat\alpha)$ and ${\rm gap}_{\rm RW}(G,\eps\hat\alpha)$ as in \eqref{eq:failure-gap}, comparison arguments as those we just described turn out to be too loose  (and, actually, too elaborated) to capture the precise pre-factors. Instead,  we follow a more basic approach, namely,  turning the aforementioned metastability picture into some quantitative spectral information. 

More specifically, for a fixed graph $G$ and a fixed number of particles $k\in \N$, we employ and sharpen classical limit theorems for slow--fast systems from \cite{kurtz_limit_theorem_1973} (see also \cite{grosskinsky_redig_vafayi_dynamics_2013} for another application in the context of ${\rm SIP}$). These theorems rigorously describe the metastable  dynamics for the slow macroscopic variables of ${\rm SIP}(G,\eps\hat\alpha)$ in the limit $\eps \to 0$. The limit of the slow variables after a thermalization of the fast ones  identifies the low-lying spectrum. This allows us to capture the behavior of ${\rm gap}_{\rm SIP}(G,\eps\hat\alpha)$, at least for $\eps$ small enough. These intuitive ideas are stated and proved rigorously in Section \ref{sec5.1}.

We are then left with identifying the spectral gap of the limiting metastable dynamics. In contrast to ${\rm SIP}$, which is irreducible, this metastable chain may consist of both recurrent and transient states, as we detail in Section \ref{sec5.2}. Recurrent regions correspond to those configurations in which all particles sit together, and the resulting single stack moves like ${\rm RW}(G,\hat\alpha)$. Transient states are  those configurations in which there are  stacks  at graph distance at least two from each other. Hence, the spectral gap of the metastable process is just the smallest among the random walk's gap and the eigenvalues associated to any part of the (sub-stochastic) transient dynamics. 

Most of our work is devoted to analyzing the hierarchy of transient states. Indeed, even the decomposition into irreducible components and the  jump rates of  the transient dynamics are, in general, highly complicated and  non-trivially dependent on the underlying geometry and the number of particles. Nevertheless, by exploiting a form of consistency for the metastable limiting dynamics, we are able to deduce that the lowest-lying eigenvalue of the transient dynamics is attained by the system with just \textit{two} particles, as adding more particles --- thus, potentially,  stacks --- does not cause any slowdown in the system. More in detail, in Section \ref{sec5.3}, we verify that one only needs to focus on the number of stacks (and not on the precise allocation of particles in each stack): for each $m \ge 2$, the smallest eigenvalue of all transient sub-systems with $m$ piles comes exactly from the sub-system with $m$ isolated particles. As a next step, in Section \ref{sec5.4}, we prove that the smallest eigenvalue of all sub-systems with $m \ge 2$ isolated particles is attained at $m=2$, i.e., the sub-system with two isolated particles. Collecting these observations, we prove  Theorem \ref{th:gap-asymptotics1} in Section \ref{sec5.5}.
Theorem \ref{th:failure-gap} is also proved in this section: by exploiting this rigorous metastable description of eigenvalues' asymptotics, our proof   boils down to estimates of  relaxation and mean meeting times  of independent particles on a well studied geometry, corresponding, respectively, to  (the inverse of) the spectral gap of the recurrent and transient dynamics of    ${\rm SIP}$'s  metastable process.

\section{Proof of Theorem \ref{th:sharp-alpha-min}}\label{sec3}
We follow the outline in Section \ref{sec:outline-1}.
\subsection{General eigenstructure of SIP}\label{sec3.1}
We start by recalling from \cite{kim_sau_spectral_2023} some general facts on the eigenstructure of ${\rm SIP}$ functional to our analysis. 
First of all, ${\rm SIP}$ is {consistent}, in the sense that the system with $k-1$ particles is recovered (in the sense of finite-dimensional distributions) from the system with $k$ particles, provided that one particle is removed \textit{uniformly at random}. In formula, this means that
\begin{equation}\label{eq:consistency-SIP}
	L_{G,\alpha,k}\,\mathfrak a_{k} = \mathfrak a_{k}\,L_{G,\alpha,k-1}
\end{equation}
holds true for all graphs $G$, site weights $\alpha$, and $k\ge 2$, where $\mathfrak a_{k}:\R^{\Xi_{k-1}}\to \R^{\Xi_k}$ is the \textit{annihilation operator} defined as
\begin{equation}\label{ann-op-def}
	\mathfrak a_{k} g(\eta)\eqdef \sum_{x\in V}\eta_x\, g(\eta-\delta_x)\comma\qquad g\in \R^{\Xi_{k-1}}\comma \eta\in \Xi_k\fstop
\end{equation}
As proved in \cite[Appendix A]{kim_sau_spectral_2023}, $\mathfrak a_k$ is one-to-one. This readily implies that the (real) spectrum of $L_{G,\alpha,k-1}$ is contained (with multiplicities) in that of $L_{G,\alpha,k}$ (and, in particular, that \eqref{eq:gap-consistency} holds). More precisely, if $(\lambda,f)$ is an eigenvalue-eigenfunction pair for $-L_{G,\alpha,k-1}$, then $(\lambda,\mathfrak a_k f)$ is one for $-L_{G,\alpha,k}$. As a consequence of this fact and reversibility, the rest of the  spectrum of $-L_{G,\alpha,k}$ that does not come from $-L_{G,\alpha,k-1}$  must be found in the orthogonal complement of those eigenfunctions \textquotedblleft lifted\textquotedblright\ from those of $-L_{G,\alpha,k-1}$.
This observation leads us to introduce the  \textit{creation operator} $\cre :\R^{\Xi_k}\to \R^{\Xi_{k-1}}$:
\begin{equation}\label{eq:cre-def}
\cre f(\xi) \eqdef \sum_{x\in V} (\xi_x +\alpha_x ) \, f(\xi + \delta_x ) \comma \qquad f \in \R^{\Xi_k} \comma \xi \in \Xi_{k-1} \fstop
\end{equation}
As demonstrated in \cite[Proposition 3.1]{kim_sau_spectral_2023}, the two operators $\ann$ and $\cre$ are adjoint one to each other in the following sense:
\begin{equation}\label{eq:orth}
\nscalar{\ann g}{f}_{\alpha,k} = \frac{k}{|\alpha|+k-1} \, \nscalar{g}{\cre f}_{\alpha,k-1} \comma \qquad f\in \R^{\Xi_k} \comma g \in \R^{\Xi_{k-1}} \comma
\end{equation}
where $\nscalar{\emparg}{\emparg}_{\alpha,k}$ denotes the inner product in $\Xi_k$ with respect to $\mu_{\alpha,k}$ given in \eqref{mu-def}. In turn,  $\cre$ is \fs{surjective} and  the following orthogonal decomposition holds:
\begin{equation}\label{ortho-decomp}
L^2(\mu_{\alpha,k}) = {\rm Im}\,\ann \oplus_{\perp} {\rm Ker}\,\cre \fstop
\end{equation}

Next, let us recall a well-known variational characterization of ${\rm gap}_k(G,\alpha)$: for each $k\ge 2$,
\begin{equation}\label{eq:var-formulation-gap}
	{\rm gap}_k(G,\alpha) = \inf_{\substack{f \in \R^{\Xi_k}\\ f \ne \text{const.}}} \frac{\cE_{G,\alpha,k}(f)}{{\rm Var}_{\alpha,k}(f)} \fstop
\end{equation}
Here,  $\cE_{G,\alpha,k}(f)=\scalar{f}{-L_{G,\alpha,k}f}_{\alpha,k}$ denotes the Dirichlet form evaluated at $f\in\R^{\Xi_k}$, namely,
\begin{equation}\label{eq:dir-form-SIP}
\cE_{G,\alpha,k}(f) 
= \frac12 \sum_{\eta \in \Xi_k} \sum_{x,y\in V} \mu_{\alpha,k}(\eta)\,c_{xy}\, \eta_x \tonde{\alpha_y + \eta_y} \tonde{f(\eta-\delta_x+\delta_y)-f(\eta)}^2 \comma
\end{equation}
whereas ${\rm Var}_{\alpha,k}(f)$ stands for the variance of $f \in \R^{\Xi_k}$ with respect to $\mu_{\alpha,k}$. If $f\in {\rm Ker}\,\cre$, then $\langle f\rangle_{\alpha,k}=\nscalar{f}{1}_{\alpha,k}=0$ \fs{(remark that $1\in {\rm Im}\,\mathfrak a_k$ and apply \eqref{eq:orth})} and,  thus,	
\begin{equation}\label{eq:ker-var}
{\rm Var}_{\alpha,k}(f) = \norm{f}_{\alpha,k}^2\eqdef \sum_{\eta \in \Xi_k} \mu_{\alpha,k}(\eta)\, f(\eta)^2 \fstop
\end{equation}
Taking advantage of the lifting property of $\ann$, the self-adjointness of $L_{G,\alpha,k}$ on $L^2(\mu_{\alpha,k})$, and the orthogonal decomposition in \eqref{ortho-decomp}, \eqref{eq:var-formulation-gap} simplifies as follows:
\begin{equation}\label{eq:gap-induction}
{\rm gap}_k(G,\alpha) = {\rm gap}_{k-1}(G,\alpha) \wedge \bigg( \inf_{\substack{f\in {\rm Ker}\,\cre \\ f\ne 0}} \frac{\cE_{G,\alpha,k}(f)}{\norm{f}_{\alpha,k}^2} \bigg) \fstop
\end{equation}
Hence, by an iterative argument on $k\ge 2$, we may focus on comparing  $\cE_{G,\alpha,k}(f)$ and $\norm{f}_{\alpha,k}^2$, only for functions  $f \in {\rm Ker}\,\cre$.

\subsection{Dirichlet forms and comparisons}\label{sec3.2}
The Dirichlet form $\cE_{G,\alpha,k}(f)$ depends on the underlying graph $G$, whereas \fs{the reversible measure $\mu_{\alpha,k}$, and, thus,} the orthogonal decomposition in \eqref{ortho-decomp} and $\norm{f}_{\alpha,k}^2$,  do not. This simple observation motivates us to compare $\cE_{G,\alpha,k}(f)$ with its complete graph analogue, i.e., $\cE_{K,\alpha,k}(f)$, where $K=K_V$ denotes the complete graph on the sites of $V$ with unitary conductances $c_{xy}\equiv 1$.

	 It turns out that, in this complete graph case, we can obtain the full eigendecomposition of ${\rm SIP}$. Remark that, although the spectrum and a set of eigenfunctions of ${\rm SIP}_k(K,\alpha)$ are  known \cite{shimakura_equations1977} (see also \cite[Theorem 1.4]{corujo2020spectrum} or \cite{wang_zhang_nash_2019} and references therein), the next proposition, together with the identity in \eqref{eq:gap-induction},  provides a simple and complete description of eigenvalues and eigenspaces when $G=K$, which we shall exploit later.

\begin{proposition}\label{prop:complete-spectrum}
For every $k \ge 1$ and $f\in {\rm Ker}\,\cre$, we have
\begin{equation}\label{eq:complete-spectrum}
\cE_{K,\alpha,k}(f) = k \tonde{|\alpha|+k-1} \norm{f}_{\alpha,k}^2 \fstop
\end{equation}
\end{proposition}
\begin{proof}
The result is well known for $k=1$,  the random walk case. Indeed, ${\rm Ker}\,\mathfrak a_{\alpha,0}^\dagger$ coincides with the subspace of  functions $f\in \R^{\Xi_1}$ having mean zero with respect to $\mu_{\alpha,1}$, which equals $\mu_{\alpha,1}(\delta_x)=\frac{\alpha_x}{|\alpha|}$, $x\in V$.

Now, fix $k\ge 2$ and $f\in {\rm Ker}\,\cre$. We have
\begin{equation}\label{eq:dec-dir}
	\cE_{K,\alpha,k}(f)= k\sum_{\xi\in \Xi_{\fs{k-1}}}\mu_{\alpha,k-1}(\xi)\, \cE_{K,\alpha+\xi,1}(f_\xi)\comma
\end{equation}
with $f_\xi\in \R^{\Xi_1}$ being defined as $f_\xi(\delta_x)\eqdef f(\xi+\delta_x)$. Remark that \fs{\eqref{eq:dec-dir} corresponds to   \cite[Eq.\ (3.10)]{kim_sau_spectral_2023}, which is deduced}  from the following  identity \cite[Eq.\ (3.9)]{kim_sau_spectral_2023}: for all $x\in V$ and $\xi\in \Xi_{k-1}$, 
\begin{equation}\label{eq:basic-id}
	\mu_{\alpha,k}(\xi+\delta_x)\tonde{\xi_x+1}= \frac{Z_{\alpha,k-1}}{Z_{\alpha,k}}\,\mu_{\alpha,k-1}(\xi)\tonde{\alpha_x+\xi_x} = k\,\mu_{\alpha,k-1}(\xi)\,\frac{\alpha_x+\xi_x}{|\alpha|+k-1}\comma
\end{equation}
where the second step used (cf.\ \eqref{mu-def})
\begin{equation}\label{eq:Zk-Zk-1}
	\frac{Z_{\alpha,k-1}}{Z_{\alpha,k}} =\frac{k}{|\alpha|+k-1}\fstop
\end{equation} Analogously \fs{to \eqref{eq:dec-dir}}, we get
\begin{equation}\label{eq:dec-norm}
	\norm{f}_{\alpha,k}^2 = \sum_{\xi\in \Xi_{k-1}}\mu_{\alpha,k-1}(\xi)\norm{f_\xi}_{\alpha+\xi,1}^2\fstop
\end{equation}By  \eqref{eq:dec-dir} and the identity in \eqref{eq:complete-spectrum} for $k=1$ (note that $f\in {\rm Ker}\,\cre$ ensures that $f_\xi\in {\rm Ker}\,\mathfrak a_{\alpha,0}^\dagger$) \fs{with $\alpha+\xi$ in place of $\alpha$ therein (note that $|\alpha+\xi|=|\alpha|+k-1$)}, we obtain
\begin{equation}
	\cE_{K,\alpha,k}(f) = k\tonde{|\alpha|+k-1}\sum_{\xi\in \Xi_{k-1}}\mu_{\alpha,k-1}(\xi) \norm{f_\xi}_{\alpha+\xi,1}^2 = k\tonde{|\alpha|+k-1}\norm{f}_{\alpha,k}^2\comma
\end{equation}
\fs{where the last step used \eqref{eq:dec-norm}. This proves the desired result for $k\ge 2$.}
\end{proof}

In view of \eqref{eq:gap-induction} and Proposition \ref{prop:complete-spectrum}, it remains to compare the two Dirichlet forms $\cE_{G,\alpha,k}(f)$ and $\cE_{K,\alpha,k}(f)$, for any $f\in \R^{\Xi_k}$.  This is the content of the following theorem, and certainly represents the hardest step of the proof of Theorem \ref{th:sharp-alpha-min}. 

\begin{theorem}\label{th:key-ing}
For all $k \ge 2$ and $f\in \R^{\Xi_k}$, we have
\begin{equation}\label{eq:key-ing}
\cE_{K,\alpha,k}(f)\le \frac{\sk{54} \, k \tonde{\alpha_{\rm max}+k-1} |V|^2 \,  {\rm diam}(G) \, \sk{(6e)}^{\alpha_{\rm max} }}{\alpha_{\rm min} \, \alpha_{\rm ratio} \, c_{\rm min}}   \, \cE_{G,\alpha,k}(f) \fstop
\end{equation}
\end{theorem}
We postpone the proof of this estimate to Section \ref{sec4} and, in Remark \ref{rem:new-eps-to-zero}, we discuss possible improvements.  

\begin{remark}While all previous steps  relied on the fact that we chose the test function $f$ in ${\rm Ker}\,\cre\subset \R^{\Xi_k}$, \fs{the last theorem is true for all	 $f\in \R^{\Xi_k}$}.
	\end{remark}

\subsection{Proof of Theorem \ref{th:sharp-alpha-min}}\label{sec3.3}
Assuming the validity of Theorem \ref{th:key-ing}, 
we now conclude the proof of Theorem \ref{th:sharp-alpha-min}.

\begin{proof}[Proof of Theorem \ref{th:sharp-alpha-min}]
By \eqref{eq:gap-sip-def}, \eqref{eq:gap-induction} and an iterative argument, it suffices to prove
\begin{equation}
\inf_{\substack{f\in {\rm Ker}\,\cre \\ f\ne 0 } } \frac{\cE_{G,\alpha,k}(f)}{\norm{f}_{\alpha,k}^2} \ge \alpha_{\rm min}\, \frac1{21} \frac{ \alpha_{\rm ratio}\, c_{\rm min}}{ 6^{\alpha_{\rm max}} \,  |V|^2\, {\rm diam}(G)} \comma\qquad \text{\fs{for all}}\ k\ge 2\fstop
\end{equation}
By Theorem \ref{th:key-ing} and Proposition \ref{prop:complete-spectrum}, we have, for all $f\in{\rm Ker}\,\cre$ and $k \ge 2$, 
\begin{align}
\cE_{G,\alpha,k}(f) & \ge \frac{ \alpha_{\rm min} \, \alpha_{\rm ratio} \, c_{\rm min} }{ \sk{54} \, k \tonde{\alpha_{\rm max}+k-1} |V|^2 \,  {\rm diam}(G) \, \sk{(6e)}^{\alpha_{\rm max} }  }  \, \cE_{K,\alpha,k}(f) \\
& = \frac{\alpha_{\rm min}\, c_{\rm min}\, \alpha_{\rm ratio}\tonde{|\alpha|+k-1}}{\sk{54}  \tonde{\alpha_{\rm max}+k-1} |V|^2\, {\rm diam} (G) \, \sk{(6e)}^{\alpha_{\rm max}}  }\norm{f}_{\alpha,k}^2  \\
& \ge \frac1{\sk{54}}\frac{\alpha_{\rm min}\, c_{\rm min}  \, \alpha_{\rm ratio} }{|V|^2 \, {\rm diam}(G) \, \sk{(6e)}^{\alpha_{\rm max}}} \norm{f}_{\alpha,k}^2 \fstop
\end{align}
This concludes the proof of the theorem.
\end{proof}

\section{Proof of Theorem \ref{th:key-ing}}\label{sec4}
 Since $G, \alpha, k$ and $f\in \R^{\Xi_k}$ are fixed all throughout this section,  we abbreviate
\begin{equation}\label{eq:abbrv-dir-forms}
\cE = \cE_{G,\alpha,k}(f) \qquad \text{and} \qquad \cE_K = \cE_{K,\alpha,k}(f) \fstop
\end{equation}
Let us assume that an ordering is given on $V$, so that
\begin{equation}\label{eq:Diri-K-dec}
\cE_K= \sum_{x<y} \sum_{\eta\in\Xi_{k}} \mu_{\alpha,k}(\eta)\,\eta_x\tonde{\alpha_y+\eta_y}\tonde{f(\eta-\delta_x+\delta_y)-f(\eta)}^{2} \eqqcolon \sum_{x<y} \cE_K^{x,y} \comma
\end{equation}
where, for any pair $x<y$, $\cE_K^{x,y}$ is defined as the summation in $\eta\in\Xi_k$ above.
\sk{The following lemma gives an upper bound for the last term $\mathcal E_K^{x,y}$.}

\sk{\begin{lemma}\label{lem:WTS-xy}
		For every pair $x<y$,
		\begin{equation}\label{eq:WTS-xy}
			\cE_K^{x,y} \le \frac{\sk{108 \, k \tonde{\alpha_{\rm max}+k-1}  {\rm diam}(G) \, (6e)^{\alpha_{\rm max} }}}{\alpha_{\rm min} \, \alpha_{\rm ratio} \, c_{\rm min}}  \, \cE \fstop
		\end{equation}
\end{lemma}}

\sk{We first prove Theorem \ref{th:key-ing} provided that Lemma \ref{lem:WTS-xy} holds true.

\begin{proof}[Proof of Theorem \ref{th:key-ing}] Recall the notation from \eqref{eq:abbrv-dir-forms}.
By \eqref{eq:Diri-K-dec} and \eqref{eq:WTS-xy}, we obtain
\begin{align}
\cE_K & \le {|V| \choose 2} \, \frac{108 \,k \tonde{\alpha_{\rm max}+k-1} {\rm diam}(G) \, (6e)^{\alpha_{\rm max}}}{\alpha_{\rm min} \, c_{\rm min} \, \alpha_{\rm ratio}} \, \cE \\
& \le \frac{54 \, k \tonde{\alpha_{\rm max}+k-1} |V|^2 \, {\rm diam}(G) \, (6e)^{\alpha_{\rm max}} }{\alpha_{\rm min} \, c_{\rm min} \, \alpha_{\rm ratio}} \, \cE \fstop
\end{align}
This is exactly the claim in Theorem \ref{th:key-ing}.
\end{proof}
}

In the remainder of this section, we fix $x<y$ and prove \sk{Lemma \ref{lem:WTS-xy}}.

For $A\subseteq V$ and $\ell\in\N_0$, define
\begin{equation}
\Xi_\ell^A \eqdef \left\{ \sigma \in \Xi_\ell : \sum_{\sk{z} \in A}\sigma_\sk{z} = \ell \right\} \fstop
\end{equation}
Note that $\Xi_{k}^{V}=\Xi_{k}$. Moreover, for $\sigma \in \Xi_{\ell}^{A}$, $y_{1},\ldots,y_{s}\in V$, and $m_{1},\ldots,m_{s}\in\N_0$, define
\begin{equation}\label{eq:eta-xi-mi}
\sigma_{m_1,\ldots,m_{s}}^{y_{1}\cdots y_s} \eqdef \sigma+\sum_{i=1}^{s} m_{i}\, \delta_{y_{i}}  \in \Xi_{\ell+m_1 +\cdots +m_s}\comma
\end{equation}
where the right-hand side should be understood as a summation of
functions. Thus, the number of particles in the new configuration $\sigma_{m_{1},\ldots,m_{s}}^{y_{1}\cdots y_{s}}$
equals $\ell+m_{1}+\cdots+m_s$.

According to \eqref{eq:eta-xi-mi}, we may decompose the Dirichlet summation $\cE_K^{x,y}$ in \eqref{eq:Diri-K-dec} as
\begin{equation}
\sum_{\ell=1}^{k} \sum_{m=1}^{\ell} \sum_{\sigma\in\Xi_{k-\ell}^{V\setminus\{x,y\}}} \mu_{\alpha,k}(\sigma_{m,\ell-m}^{xy})\,m\tonde{\alpha_{y}+\ell-m}\tttonde{f(\sigma_{m-1,\ell-m+1}^{xy})-f(\sigma_{m,\ell-m}^{xy})}^{2} \comma
\end{equation}
where $\ell\in\bbr{1}{k}\eqdef [1,k]\cap\Z$ denotes the number of particles in $\{x,y\}$, and $m\in\bbr{1}{\ell}$ denotes the number of particles at $x$. 
\sk{For each $\eta \in \Xi_k$ and $x,y \in V$, define
\begin{equation}\label{eq:cK-nablaK-def}
{\fc_K(\eta;x,y)} \eqdef \mu_{\alpha,k}(\eta) \, \eta_x \tonde{\alpha_y + \eta_y} \comma\quad
\nabla_K^2 f(\eta;x,y) \eqdef \fc_K (\eta;x,y)\, \tttonde{f(\eta-\delta_x+\delta_y)-f(\eta)}^2 \fstop
\end{equation}
According to this notation,
\begin{equation}\label{eq:Diri-dec-sigma}
\cE_K^{x,y} = \sum_{\ell=1}^k \sum_{m=1}^\ell \sum_{\sigma \in \Xi_{k-\ell}^{V\setminus \{x,y\}}}
\nabla_K^2 f ( \sigma_{m,\ell-m}^{xy} ; x,y) \fstop
\end{equation}
Indeed, the two summations in $\ell \in \bbr1k$ and $m \in \bbr1\ell$ count all possible types of jumps from $x$ to $y$, and the summation in $\sigma \in \Xi_{k-\ell}^{V \setminus \{x,y\}}$ counts the possible particle configuration on the remaining vertices in $V \setminus \{x,y\}$.
}

Fix a shortest sequence
\begin{equation}\label{eq:xy-path}
x=x_{0},x_{1},\ldots,x_{t}=y
\end{equation}
in $G$ such that $c_{x_{s-1}x_s}>0$ for all $s\in\bbr{1}{t}$, where $t\in\bbr{1}{{\rm diam}(G)}$. 
For $\eta,\zeta\in\Xi_{k}$, write
\begin{equation}\label{eq:c-nabla-def}
	\fc(\eta,\zeta)\eqdef \mu_{\alpha,k}(\eta)\, r_{\alpha,k}(\eta,\zeta)\qquad \text{and} \qquad \nabla^{2}f(\eta,\zeta) \eqdef \fc(\eta,\zeta) \tonde{f(\zeta)-f(\eta)}^{2}\comma
\end{equation}
where $r_{\alpha,k}(\cdot,\cdot)$ is the transition rate function of ${\rm SIP}_{k}(G,\alpha)$.
\sk{Note that by the reversibility of $\mu_{\alpha,k}$, $\nabla^2f(\eta,\zeta) = \nabla^2f(\zeta,\eta)$.}
Then, we have
\begin{equation}\label{eq:Diri-nabla-dec}
	\cE = \sum_{ \eta \in \Xi_k } \sum_{\substack{z<w \\ c_{zw}>0} } \nabla^2f(\eta,\eta-\delta_z+\delta_w ) \fstop
\end{equation}
 A sequence of configurations $\eta_0,\eta_1,\ldots, \eta_M$ in $\Xi_k$ is  a \emph{path} if the  transition rates $r_{\alpha,k}(\cdot,\cdot)$ are positive along the sequence, i.e., $r_{\alpha,k}(\eta_{s-1},\eta_{s})>0$ for all $s\in \bbr{1}{M}$.

The idea to prove \sk{Lemma \ref{lem:WTS-xy}} is as follows. First, we decompose $\cE_K^{x,y}$ according to \eqref{eq:Diri-dec-sigma}. Then, for each triple $(\ell,m,\sigma)$, we upper bound  $\sk{\nabla_K^2 f ( \sigma_{m,\ell-m}^{xy} ; x,y) }$ with a certain collection of terms $\nabla^2f(\eta,\eta-\delta_z+\delta_w)$ that appear in the right-hand side of \eqref{eq:Diri-nabla-dec}. Finally, we count the number of overlaps for each term $\nabla^2f(\eta,\eta-\delta_z+\delta_w)$ for all such triple $(\ell,m,\sigma)$,  which gives an upper bound of $\cE_K^{x.y}$ in terms of $\cE$.

The procedure of upper bounding each $\sk{\nabla_K^2 f ( \sigma_{m,\ell-m}^{xy} ; x,y) }$ depends on the detailed distribution of particles along the sequence \eqref{eq:xy-path}. In the following four subsections, Sections \ref{sec4.1}, \ref{sec4.2}, \ref{sec4.3} and \ref{sec4.4}, we demonstrate each simple cases, and finally in Section \ref{sec4.5} we deal with the general case and conclude the proof of \sk{Lemma \ref{lem:WTS-xy}}.

\subsection{Connected case}\label{sec4.1}
First, suppose that (cf.\ \eqref{eq:xy-path})
\begin{equation}\label{eq:cond-connected}
c_{xy}>0 \comma \quad \sk{\text{which implies that}}\ t=1  \fstop
\end{equation}
Define
\begin{equation}\label{eq:Omega-def-1}
\sk{\Omega_{x,y}^{\ell,m,\sigma}} \eqdef \sk{\, \Big\{ } \{ \sigma_{m,\ell-m}^{xy} , \sigma_{m-1,\ell-m+1}^{xy} \} \sk{\Big\} \, } \comma
\end{equation}
\sk{which is the singleton set that contains the edge $\{\sigma_{m,\ell-m}^{xy} , \sigma_{m-1,\ell-m+1}^{xy} \}$. When there is no room for confusion, we simply write $\Omega = \Omega_{x,y}^{\ell,m,\sigma}$.}

\begin{lemma}\label{lem:4.1}
Suppose that \eqref{eq:cond-connected} holds. Then, we have
\begin{equation}
\sk{\nabla_K^2 f ( \sigma_{m,\ell-m}^{xy} ; x,y) } \le \sk{\frac{m \tonde{\alpha_y+\ell-m}}{c_{\rm min} \, \alpha_{\rm min}}}
\sum_{\{\zeta,\zeta'\} \sk{\in } \Omega } \nabla^2f(\zeta,\zeta') \fstop
\end{equation}
\end{lemma}

\sk{
\begin{proof}
It is clear that $\sigma_{m,\ell-m}^{xy},\sigma_{m-1,\ell-m+1}^{xy}$ is itself a path, and
\begin{equation}
\sk{\nabla_K^2 f ( \sigma_{m,\ell-m}^{xy} ; x,y) } = \frac1{c_{xy}} \, \nabla^2f(\sigma_{m,\ell-m}^{xy},\sigma_{m-1,\ell-m+1}^{xy}) = \frac1{c_{xy}} \, \sum_{\{\zeta,\zeta'\} \in \Omega} \nabla^2 f(\zeta,\zeta') \fstop
\end{equation}
Thus, the lemma easily follows since $c_{xy} \ge c_{\rm min}$ and $m\,(\alpha_y+\ell-m) \ge \alpha_y \ge \alpha_{\rm min}$.
\end{proof}
}

\subsection{Decomposition of gradients: occupied sites}\label{sec4.2}
Now, assume that (cf.\ \eqref{eq:xy-path} and \eqref{eq:cond-connected})
\begin{equation}\label{eq:cond-far}
c_{xy}=0 \comma \quad \sk{\text{which implies that}} \ t \ge 2  \fstop
\end{equation}
Additionally, in this subsection, assume that all sites along the sequence \eqref{eq:xy-path} are occupied by particles of $\sigma$, i.e., 
\begin{equation}\label{eq:cond-occupied}
\sigma_{x_s} \ge 1\quad \text{for all} \ s\in\bbr{1}{t-1} \fstop
\end{equation}
Then, starting from $\sigma_{m,\ell-m}^{xy}$, we send a single particle from $x_{s-1}$ to $x_s$ for each $s \in \bbr{1}{t}$, to arrive at $\sigma_{m-1,\ell-m+1}^{xy}$. The corresponding path becomes
\begin{equation}\label{eq:4.1-path}
\sigma_{m,\ell-m}^{xy},\sigma_{m-1,1,\ell-m}^{xx_1y},\ldots,\sigma_{m-1,1,\ell-m}^{xx_{t-1}y},\sigma_{m-1,\ell-m+1}^{xy} \fstop
\end{equation}
See Figure \ref{fig1} for a visual representation.

\begin{figure}
\begin{tikzpicture}[scale=0.8]
\foreach \j in {-2,...,2} {\draw (0.5*\j,0)--(0.5*\j,-0.1); }
\draw (-1,-0.1) node[below]{\scriptsize $x$};
\draw (-0.5,-0.1) node[below]{\scriptsize $x_1$};
\draw (0.5,-0.1) node[below]{\scriptsize $x_{t-1}$};
\draw (1,-0.1) node[below]{\scriptsize $y$};
\foreach \i in {0,...,8} {\fill[black!30!white] (-1,0.1+0.2*\i) circle (0.1); }
\foreach \i in {9} {\fill[red] (-1,0.1+0.2*\i) circle (0.1); }
\foreach \i in {0,...,2} {\fill[black!30!white] (-0.5,0.1+0.2*\i) circle (0.1); }
\foreach \i in {0,...,1} {\fill[black!30!white] (0.5,0.1+0.2*\i) circle (0.1); }
\foreach \i in {0,...,5} {\fill[black!30!white] (1,0.1+0.2*\i) circle (0.1); }
\draw (-1.5,0)--(1.5,0);
\foreach \i in {0,...,9} {\draw (-1,0.1+0.2*\i) circle (0.1); }
\foreach \i in {0,...,2} {\draw (-0.5,0.1+0.2*\i) circle (0.1); }
\foreach \i in {0,...,1} {\draw (0.5,0.1+0.2*\i) circle (0.1); }
\foreach \i in {0,...,5} {\draw (1,0.1+0.2*\i) circle (0.1); }
\draw[very thick,-latex] (1.7,1)--(2.3,1);
\begin{scope}[shift={(4,0)}]
\foreach \j in {-2,...,2} {\draw (0.5*\j,0)--(0.5*\j,-0.1); }
\draw (-1,-0.1) node[below]{\scriptsize $x$};
\draw (-0.5,-0.1) node[below]{\scriptsize $x_1$};
\draw (0.5,-0.1) node[below]{\scriptsize $x_{t-1}$};
\draw (1,-0.1) node[below]{\scriptsize $y$};
\foreach \i in {0,...,8} {\fill[black!30!white] (-1,0.1+0.2*\i) circle (0.1); }
\foreach \i in {0,...,2} {\fill[black!30!white] (-0.5,0.1+0.2*\i) circle (0.1); }
\foreach \i in {3} {\fill[red] (-0.5,0.1+0.2*\i) circle (0.1); }
\foreach \i in {0,...,1} {\fill[black!30!white] (0.5,0.1+0.2*\i) circle (0.1); }
\foreach \i in {0,...,5} {\fill[black!30!white] (1,0.1+0.2*\i) circle (0.1); }
\draw (-1.5,0)--(1.5,0);
\foreach \i in {0,...,8} {\draw (-1,0.1+0.2*\i) circle (0.1); }
\foreach \i in {0,...,3} {\draw (-0.5,0.1+0.2*\i) circle (0.1); }
\foreach \i in {0,...,1} {\draw (0.5,0.1+0.2*\i) circle (0.1); }
\foreach \i in {0,...,5} {\draw (1,0.1+0.2*\i) circle (0.1); }
\end{scope}
\draw[very thick,-latex] (5.7,1)--(6.3,1);
\draw (7,1) node{$\boldsymbol{\cdots}$};
\draw[very thick,-latex] (7.7,1)--(8.3,1);
\begin{scope}[shift={(10,0)}]
\foreach \j in {-2,...,2} {\draw (0.5*\j,0)--(0.5*\j,-0.1); }
\draw (-1,-0.1) node[below]{\scriptsize $x$};
\draw (-0.5,-0.1) node[below]{\scriptsize $x_1$};
\draw (0.5,-0.1) node[below]{\scriptsize $x_{t-1}$};
\draw (1,-0.1) node[below]{\scriptsize $y$};
\foreach \i in {0,...,8} {\fill[black!30!white] (-1,0.1+0.2*\i) circle (0.1); }
\foreach \i in {0,...,2} {\fill[black!30!white] (-0.5,0.1+0.2*\i) circle (0.1); }
\foreach \i in {0,...,1} {\fill[black!30!white] (0.5,0.1+0.2*\i) circle (0.1); }
\foreach \i in {2} {\fill[red] (0.5,0.1+0.2*\i) circle (0.1); }
\foreach \i in {0,...,5} {\fill[black!30!white] (1,0.1+0.2*\i) circle (0.1); }
\draw (-1.5,0)--(1.5,0);
\foreach \i in {0,...,8} {\draw (-1,0.1+0.2*\i) circle (0.1); }
\foreach \i in {0,...,2} {\draw (-0.5,0.1+0.2*\i) circle (0.1); }
\foreach \i in {0,...,2} {\draw (0.5,0.1+0.2*\i) circle (0.1); }
\foreach \i in {0,...,5} {\draw (1,0.1+0.2*\i) circle (0.1); }
\end{scope}
\draw[very thick,-latex] (11.7,1)--(12.3,1);
\begin{scope}[shift={(14,0)}]
\foreach \j in {-2,...,2} {\draw (0.5*\j,0)--(0.5*\j,-0.1); }
\draw (-1,-0.1) node[below]{\scriptsize $x$};
\draw (-0.5,-0.1) node[below]{\scriptsize $x_1$};
\draw (0.5,-0.1) node[below]{\scriptsize $x_{t-1}$};
\draw (1,-0.1) node[below]{\scriptsize $y$};
\foreach \i in {0,...,8} {\fill[black!30!white] (-1,0.1+0.2*\i) circle (0.1); }
\foreach \i in {0,...,2} {\fill[black!30!white] (-0.5,0.1+0.2*\i) circle (0.1); }
\foreach \i in {0,...,1} {\fill[black!30!white] (0.5,0.1+0.2*\i) circle (0.1); }
\foreach \i in {0,...,5} {\fill[black!30!white] (1,0.1+0.2*\i) circle (0.1); }
\foreach \i in {6} {\fill[red] (1,0.1+0.2*\i) circle (0.1); }
\draw (-1.5,0)--(1.5,0);
\foreach \i in {0,...,8} {\draw (-1,0.1+0.2*\i) circle (0.1); }
\foreach \i in {0,...,2} {\draw (-0.5,0.1+0.2*\i) circle (0.1); }
\foreach \i in {0,...,1} {\draw (0.5,0.1+0.2*\i) circle (0.1); }
\foreach \i in {0,...,6} {\draw (1,0.1+0.2*\i) circle (0.1); }
\end{scope}
\end{tikzpicture}\caption{\label{fig1}Path of configurations from $\sigma_{m,\ell-m}^{xy}$ to $\sigma_{m-1,\ell-m+1}^{xy}$ in Section \ref{sec4.2}. In this occupied case, the red particle simply moves from $x$ to $y$ along the sequence $x=x_0,x_1,\dots,x_t=y$ consecutively.}
\end{figure}
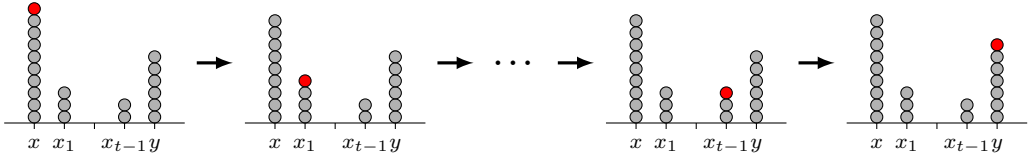

Let $\Omega=\sk{\Omega_{x,y}^{\ell,m,\sigma}}$ be the collection of \sk{all edges $\{\sigma_{m-1,1,\ell-m}^{xx_{s-1}y},\sigma_{m-1,1,\ell-m}^{xx_sy}\}$ for all $s \in \llbracket 1,t \rrbracket$, i.e.,
\begin{equation}
\Omega_{x,y}^{\ell,m,\sigma} := \left\{ \{\sigma_{m-1,1,\ell-m}^{xx_{s-1}y},\sigma_{m-1,1,\ell-m}^{xx_sy}\} : s \in \llbracket 1,t \rrbracket \right\} \fstop
\end{equation}
}
With this notation, we \sk{prove in this subsection} the following lemma.

\begin{lemma}\label{lem:4.2}
Assume that the \sk{conditions in \eqref{eq:cond-far} and \eqref{eq:cond-occupied} hold.} Then, we have
\begin{equation}
\sk{\nabla_K^2 f ( \sigma_{m,\ell-m}^{xy} ; x,y) } \le
\sk{\frac{ t \, m \tonde{\alpha_y + \ell-m} (1+\alpha_{\rm max})}{c_{\rm min} \, \alpha_{\rm min}} }
\sum_{\{\zeta,\zeta'\} \sk{\in } \Omega } \nabla^2f(\zeta,\zeta') \fstop
\end{equation}
\end{lemma}

\begin{proof}

By  Cauchy--Schwarz inequality, for any path $\eta=\eta_0,\eta_1,\ldots,\eta_M=\zeta$,  we have (cf.\ \eqref{eq:c-nabla-def})
\begin{equation}\label{eq:C-S}
(f(\zeta)-f(\eta))^2 \le \left( \sum_{i=1}^M \nabla^2f(\eta_{i-1},\eta_i) \right) \left(\sum_{i=1}^M \frac1{\fc(\eta_{i-1},\eta_i)} \right) \fstop
\end{equation}
Thus, applying \eqref{eq:C-S} along the path \eqref{eq:4.1-path}, we get (cf.\ \eqref{eq:cK-nablaK-def})
\begin{equation}\label{eq:Diri-C-S-dec}
\begin{aligned}
\sk{\nabla_K^2 f ( \sigma_{m,\ell-m}^{xy} ; x,y) } \le \, \Bigg( \sum_{s=1}^t \nabla^2f( & \sigma_{m-1,1,\ell-m}^{xx_{s-1}y},\sigma_{m-1,1,\ell-m}^{xx_sy}) \Bigg) \, \\
&\times \left( \sum_{s=1}^t \frac  { \mu_{\alpha,k}(\sigma_{m,\ell-m}^{xy})\, m\tonde{\alpha_y + \ell-m } } { \fc(\sigma_{m-1,1,\ell-m}^{xx_{s-1}y},\sigma_{m-1,1,\ell-m}^{xx_sy}) } \right) \fstop
\end{aligned}
\end{equation}
For $x\in V$ and $m \in \N_0$, define
\begin{equation}\label{eq:omega-xn-not}
\pi_{x}(m) \eqdef \frac{\Gamma(\alpha_{x}+m)}{\Gamma(\alpha_{x})\,m!}\comma \quad \text{such that} \ \mu_{\alpha,k}(\eta)=\frac{1}{Z_{\alpha,k}}\prod_{x\in V}\pi_{x}(\eta_{x})\fstop
\end{equation}
\sk{Here, $\Gamma(\cdot)$ denotes the usual gamma function. We now collect some basic properties.
\begin{lemma}\label{lem:omega-prop}
We have, for each $m \in \N_0$,
\begin{equation}\label{eq:omega-prop}
\alpha_x \le \pi_x(m)\tonde{\alpha_x+m} = \pi_x(m+1) \tonde{m+1} \le \alpha_x\,e^{\alpha_x\sum_{j=1}^m \frac1j} \fstop
\end{equation}
Moreover, for $m \ge m' \ge 1$,
\begin{equation}\label{eq:omega-prop-2}
1 \le \frac{ \pi_x(m)\,m } { \pi_x(m')\,m' }  \le e^{\alpha_x \sum_{j=m'}^{m-1}\frac1j}  \fstop
\end{equation}
\end{lemma}
\begin{proof}
Since $\Gamma(a+1)=a\,\Gamma(a)$ for all $a> 0$,
\begin{equation}
\pi_x(m) \tonde{\alpha_x+m} = \frac{ \Gamma(\alpha_x+m+1) }{\Gamma(\alpha_x) \, m!} = \pi_x(m+1) \tonde{m+1} \fstop
\end{equation}
Clearly, since $\alpha_x >0$,
\begin{equation}
\frac{ \Gamma(\alpha_x+m+1) }{\Gamma(\alpha_x) \, m!} = \alpha_x \, \prod_{j=1}^m \frac{j+\alpha_x}j \ge \alpha_x \fstop
\end{equation}
On the other hand, since $1+a \le e^a$ for any $a \in \R$,
\begin{equation}
\frac{ \Gamma(\alpha_x+m+1) }{\Gamma(\alpha_x) \, m!} = \alpha_x \, \prod_{j=1}^m \left( 1+\frac{\alpha_x}j \right) \le \alpha_x \, \exp \left( \sum_{j=1}^m \frac{\alpha_x}j \right) \fstop
\end{equation}
Combining the three displayed identities proves \eqref{eq:omega-prop}. In addition, for $m \ge m' \ge 1$,
\begin{equation}
\frac{\pi_x(m) \, m}{\pi_x(m') \, m'} = \frac{ \Gamma(\alpha_x+m) / (m-1)!}{ \Gamma(\alpha_x+m') / (m'-1)!} = \prod_{j=m'}^{m-1} \frac{j+\alpha_x}{j} \comma
\end{equation}
which lies inside the interval $[1,\exp \, (\sum_{j=m'}^{m-1}\frac{\alpha_x}j)]$ due to the same reasoning as above. This proves \eqref{eq:omega-prop-2}.
\end{proof}
}

\sk{
Returning to the proof of Lemma \ref{lem:4.2}, let us look at the term inside the second parenthesis in \eqref{eq:Diri-C-S-dec}. The term for $s=1$ equals
\begin{equation}
\frac  { \mu_{\alpha,k}(\sigma_{m,\ell-m}^{xy})\, m\tonde{\alpha_y + \ell-m } } { \fc(\sigma_{m,\ell-m}^{xy},\sigma_{m-1,1,\ell-m}^{xx_1y}) } =
\frac  { m\tonde{\alpha_y + \ell-m } } { r_{\alpha,k} (\sigma_{m,\ell-m}^{xy},\sigma_{m-1,1,\ell-m}^{xx_1y}) } = \frac{ \alpha_{y}+\ell-m } { \tonde{\alpha_{x_1}+\sigma_{x_1}} c_{xx_1} } \fstop
\end{equation}
Similarly, by \eqref{eq:omega-xn-not}, the term for $s=t$ becomes
\begin{equation}
\frac{ \pi_{x}(m)\,\pi_{x_{t-1}}(\sigma_{x_{t-1}})\, m } { \pi_{x}(m-1)\, \pi_{x_{t-1}} (\sigma_{x_{t-1}}+1) \, (\sigma_{x_{t-1}}+1) \, c_{x_{t-1}y} } \comma
\end{equation}
and the term for $s \in \llbracket 2,t-1 \rrbracket$ becomes
\begin{equation}
\frac{ \pi_{x}(m)\,\pi_{x_{s-1}}(\sigma_{x_{s-1}})\, m\tonde{\alpha_{y}+\ell-m}} { \pi_{x}(m-1)\, \pi_{x_{s-1}}(\sigma_{x_{s-1}}+1) \, (\sigma_{x_{s-1}}+1) \tonde{\alpha_{x_{s}}+\sigma_{x_{s}}} c_{x_{s-1}x_{s}} } \fstop
\end{equation}
Combining the last three displayed identities,
}
the summation inside the second parenthesis in \eqref{eq:Diri-C-S-dec} equals
\begin{equation}\label{eq:sum-dec}
\begin{aligned}
& \frac{ \alpha_{y}+\ell-m } { \tonde{\alpha_{x_1}+\sigma_{x_1}} c_{xx_1} } + \frac{ \pi_{x}(m)\,\pi_{x_{t-1}}(\sigma_{x_{t-1}})\, m } { \pi_{x}(m-1)\, \pi_{x_{t-1}} (\sigma_{x_{t-1}}+1) \, (\sigma_{x_{t-1}}+1) \, c_{x_{t-1}y} } \\
& + \sum_{s=2}^{t-1} \frac{ \pi_{x}(m)\,\pi_{x_{s-1}}(\sigma_{x_{s-1}})\, m\tonde{\alpha_{y}+\ell-m}} { \pi_{x}(m-1)\, \pi_{x_{s-1}}(\sigma_{x_{s-1}}+1) \, (\sigma_{x_{s-1}}+1) \tonde{\alpha_{x_{s}}+\sigma_{x_{s}}} c_{x_{s-1}x_{s}} } \fstop
\end{aligned}
\end{equation}
The first term in \eqref{eq:sum-dec} is easily bounded by (cf.\ \eqref{alpha-c-def})
\begin{equation}
\frac{\alpha_y + \ell-m }{ \tonde{\alpha_{x_1}+\sigma_{x_1}}c_{xx_1} } \le \frac{m \tonde{\alpha_y+\ell-m} } {c_{\rm min}} \comma 
\end{equation}
where we used that $\sigma_{x_1} \ge 1$. By \eqref{eq:omega-xn-not}, the second term in \eqref{eq:sum-dec} reads as
\begin{equation}
\frac { \alpha_x + m-1 } { (\alpha_{x_{t-1}} + \sigma_{x_{t-1}}) \, c_{x_{t-1} \sk{y}} } \le \sk{\frac{m \tonde{\alpha_y + \ell-m}} {c_{\rm min} \, \alpha_{\rm min}} \, \frac{\alpha_x + m-1}{m} \le \frac{m \tonde{\alpha_y + \ell-m}} {c_{\rm min} \, \alpha_{\rm min}} \, (1+\alpha_{\rm max}) } \comma 
\end{equation}
where we used $\sigma_{x_{t-1}} \ge 1$ \sk{and $\alpha_y + \ell-m \ge \alpha_{\rm min}$ at the first inequality and $\alpha_x + m-1 \le m(1+\alpha_x)$ at the second inequality.} Finally, the third term in \eqref{eq:sum-dec} is similarly dealt with:
\begin{equation}
\sum_{s=2}^{t-1} \frac{ \tonde{\alpha_x + m-1}\tonde{\alpha_{y}+\ell-m} } { (\alpha_{x_{s-1}} + \sigma_{x_{s-1}})\tonde{\alpha_{x_{s}}+\sigma_{x_{s}}} c_{x_{s-1}x_{s}} } 
 \le \sk{\frac{(t-2) \, m \tonde{\alpha_y + \ell-m} (1 + \alpha_{\rm max})}{c_{\rm min} \, \alpha_{\rm min} } } \comma
\end{equation}
where we used that $\sigma_{x_{s-1}}\ge1$ and $\sk{\alpha_{x_s} + \sigma_{x_s} \ge \alpha_{\rm min}}$ for all $s\in \bbr{2}{t-1}$. Therefore, inserting these three bounds \sk{and \eqref{eq:sum-dec}} into \eqref{eq:Diri-C-S-dec}, we conclude that
\begin{equation}
\sk{\nabla_K^2 f ( \sigma_{m,\ell-m}^{xy} ; x,y) } \le \sk{\frac{ t \, m \tonde{\alpha_y + \ell-m} (1+\alpha_{\rm max})}{c_{\rm min} \, \alpha_{\rm min}} } \left( \sum_{s=1}^t \nabla^2f(\sigma_{m-1,1,\ell-m}^{xx_{s-1}y},\sigma_{m-1,1,\ell-m}^{xx_sy}) \right) \fstop
\end{equation}
\sk{This proves Lemma \ref{lem:4.2}.}
\end{proof}

 \subsection{Decomposition of gradients: empty sites \& few particles}\label{sec4.3}
 Next, assume \eqref{eq:cond-far} and that sites $x_1,\dots,x_{t-1}$ are empty with respect to $\sigma$, i.e.,
 \begin{equation}\label{eq:cond-empty}
 \sigma_{x_s}=0 \quad \text{for all} \ s \in \bbr{1}{t-1} \fstop
 \end{equation}
 In this case, we divide the analysis into two parts. In Section \ref{sec4.3}, assume that
 \begin{equation}\label{eq:cond-few}
 \alpha_x \sum_{j=1}^{m-1}\frac1j  \le 1 \fstop
 \end{equation}
 We construct a path from $\sigma_{m,\ell-m}^{xy}$ to $\sigma_{m-1,\ell-m+1}^{xy}$
as follows (see also Figure \ref{fig2}):

\begin{enumerate}
\item[{\bf (F)}] Move $m$ particles from $x_{s-1}$ to $x_{s}$ consecutively, for each $s\in\bbr{1}{t-1}$. This subpath starts from $\sigma_{m,\ell-m}^{xy}$, visits $\sigma_{m-i,i,\ell-m}^{x_{s-1}x_sy}$ for each $i\in\bbr{0}{m}$ and $s\in\bbr{1}{t-1}$, and finally arrives at $\sigma_{m,\ell-m}^{x_{t-1}y}$.
\item[{\bf (S)}] Move a particle from $x_{t-1}$ to $y$. The resulting configuration
is $\sigma_{m-1,\ell-m+1}^{x_{t-1}y}$. If $m=1$, the path is complete.
\item[{\bf (B)}] If $m\ge2$, move the remaining $m-1$ particles from $x_s$ to $x_{s-1}$
consecutively, for each $s\in\bbr{1}{t-1}$ (backwards). This subpath starts from $\sigma_{m-1,\ell-m+1}^{x_{t-1}y}$, visits $\sigma_{m-i,i-1,\ell-m+1}^{x_{s-1}x_sy}$ for each $i\in\bbr{1}{m}$ and $s\in\bbr{1}{t-1}$, thereby arrives at $\sigma_{m-1,\ell-m+1}^{xy}$ as desired.
\end{enumerate}

\begin{figure}
\begin{tikzpicture}
\foreach \j in {-2,...,2} {\draw (0.5*\j,0)--(0.5*\j,-0.1); }
\draw (-1,-0.1) node[below]{\footnotesize $x$};
\draw (-0.5,-0.1) node[below]{\footnotesize $x_1$};
\draw (0.5,-0.1) node[below]{\footnotesize $x_{t-1}$};
\draw (1,-0.1) node[below]{\footnotesize $y$};
\foreach \i in {0,...,8} {\fill[black!30!white] (-1,0.1+0.2*\i) circle (0.1); }
\foreach \i in {9} {\fill[red] (-1,0.1+0.2*\i) circle (0.1); }
\foreach \i in {0,...,5} {\fill[black!30!white] (1,0.1+0.2*\i) circle (0.1); }
\draw (-1.5,0)--(1.5,0);
\foreach \i in {0,...,9} {\draw (-1,0.1+0.2*\i) circle (0.1); }
\foreach \i in {0,...,5} {\draw (1,0.1+0.2*\i) circle (0.1); }
\draw[very thick,-latex,double] (1.7,1)--(2.3,1);
\draw (2,1.1) node[above]{\footnotesize \bf (F)};
\begin{scope}[shift={(4,0)}]
\foreach \j in {-2,...,2} {\draw (0.5*\j,0)--(0.5*\j,-0.1); }
\draw (-1,-0.1) node[below]{\footnotesize $x$};
\draw (-0.5,-0.1) node[below]{\footnotesize $x_1$};
\draw (0.5,-0.1) node[below]{\footnotesize $x_{t-1}$};
\draw (1,-0.1) node[below]{\footnotesize $y$};
\foreach \i in {0,...,8} {\fill[black!30!white] (-0.5,0.1+0.2*\i) circle (0.1); }
\foreach \i in {9} {\fill[red] (-0.5,0.1+0.2*\i) circle (0.1); }
\foreach \i in {0,...,5} {\fill[black!30!white] (1,0.1+0.2*\i) circle (0.1); }
\draw (-1.5,0)--(1.5,0);
\foreach \i in {0,...,9} {\draw (-0.5,0.1+0.2*\i) circle (0.1); }
\foreach \i in {0,...,5} {\draw (1,0.1+0.2*\i) circle (0.1); }
\end{scope}
\draw[very thick,-latex,double] (5.7,1)--(6.3,1);
\draw (6,1.1) node[above]{\footnotesize \bf (F)};
\draw (7,1) node{$\boldsymbol{\cdots}$};
\draw[very thick,-latex,double] (7.7,1)--(8.3,1);
\draw (8,1.1) node[above]{\footnotesize \bf (F)};
\begin{scope}[shift={(10,0)}]
\foreach \j in {-2,...,2} {\draw (0.5*\j,0)--(0.5*\j,-0.1); }
\draw (-1,-0.1) node[below]{\footnotesize $x$};
\draw (-0.5,-0.1) node[below]{\footnotesize $x_1$};
\draw (0.5,-0.1) node[below]{\footnotesize $x_{t-1}$};
\draw (1,-0.1) node[below]{\footnotesize $y$};
\foreach \i in {0,...,8} {\fill[black!30!white] (0.5,0.1+0.2*\i) circle (0.1); }
\foreach \i in {9} {\fill[red] (0.5,0.1+0.2*\i) circle (0.1); }
\foreach \i in {0,...,5} {\fill[black!30!white] (1,0.1+0.2*\i) circle (0.1); }
\draw (-1.5,0)--(1.5,0);
\foreach \i in {0,...,9} {\draw (0.5,0.1+0.2*\i) circle (0.1); }
\foreach \i in {0,...,5} {\draw (1,0.1+0.2*\i) circle (0.1); }
\end{scope}
\draw[very thick,-latex] (10,-0.7)--(10,-1.3);
\draw (10.1,-1) node[right]{\footnotesize \bf (S)};
\begin{scope}[shift={(0,-3.5)}]
\foreach \j in {-2,...,2} {\draw (0.5*\j,0)--(0.5*\j,-0.1); }
\draw (-1,-0.1) node[below]{\footnotesize $x$};
\draw (-0.5,-0.1) node[below]{\footnotesize $x_1$};
\draw (0.5,-0.1) node[below]{\footnotesize $x_{t-1}$};
\draw (1,-0.1) node[below]{\footnotesize $y$};
\foreach \i in {0,...,8} {\fill[black!30!white] (-1,0.1+0.2*\i) circle (0.1); }
\foreach \i in {0,...,5} {\fill[black!30!white] (1,0.1+0.2*\i) circle (0.1); }
\foreach \i in {6} {\fill[red] (1,0.1+0.2*\i) circle (0.1); }
\draw (-1.5,0)--(1.5,0);
\foreach \i in {0,...,8} {\draw (-1,0.1+0.2*\i) circle (0.1); }
\foreach \i in {0,...,6} {\draw (1,0.1+0.2*\i) circle (0.1); }
\draw[very thick,latex-,double] (1.7,1)--(2.3,1);
\draw (2,1.1) node[above]{\footnotesize \bf (B)};
\begin{scope}[shift={(4,0)}]
\foreach \j in {-2,...,2} {\draw (0.5*\j,0)--(0.5*\j,-0.1); }
\draw (-1,-0.1) node[below]{\footnotesize $x$};
\draw (-0.5,-0.1) node[below]{\footnotesize $x_1$};
\draw (0.5,-0.1) node[below]{\footnotesize $x_{t-1}$};
\draw (1,-0.1) node[below]{\footnotesize $y$};
\foreach \i in {0,...,8} {\fill[black!30!white] (-0.5,0.1+0.2*\i) circle (0.1); }
\foreach \i in {0,...,5} {\fill[black!30!white] (1,0.1+0.2*\i) circle (0.1); }
\foreach \i in {6} {\fill[red] (1,0.1+0.2*\i) circle (0.1); }
\draw (-1.5,0)--(1.5,0);
\foreach \i in {0,...,8} {\draw (-0.5,0.1+0.2*\i) circle (0.1); }
\foreach \i in {0,...,6} {\draw (1,0.1+0.2*\i) circle (0.1); }
\end{scope}
\draw[very thick,latex-,double] (5.7,1)--(6.3,1);
\draw (6,1.1) node[above]{\footnotesize \bf (B)};
\draw (7,1) node{$\boldsymbol{\cdots}$};
\draw[very thick,latex-,double] (7.7,1)--(8.3,1);
\draw (8,1.1) node[above]{\footnotesize \bf (B)};
\begin{scope}[shift={(10,0)}]
\foreach \j in {-2,...,2} {\draw (0.5*\j,0)--(0.5*\j,-0.1); }
\draw (-1,-0.1) node[below]{\footnotesize $x$};
\draw (-0.5,-0.1) node[below]{\footnotesize $x_1$};
\draw (0.5,-0.1) node[below]{\footnotesize $x_{t-1}$};
\draw (1,-0.1) node[below]{\footnotesize $y$};
\foreach \i in {0,...,8} {\fill[black!30!white] (0.5,0.1+0.2*\i) circle (0.1); }
\foreach \i in {0,...,5} {\fill[black!30!white] (1,0.1+0.2*\i) circle (0.1); }
\foreach \i in {6} {\fill[red] (1,0.1+0.2*\i) circle (0.1); }
\draw (-1.5,0)--(1.5,0);
\foreach \i in {0,...,8} {\draw (0.5,0.1+0.2*\i) circle (0.1); }
\foreach \i in {0,...,6} {\draw (1,0.1+0.2*\i) circle (0.1); }
\end{scope}
\end{scope}
\end{tikzpicture}\caption{\label{fig2}Path of configurations from $\sigma_{m,\ell-m}^{xy}$ to $\sigma_{m-1,\ell-m+1}^{xy}$ in Section \ref{sec4.3}. Double arrows indicate series of consecutive jumps. Here, first the stack of $m$ particles at site $x$ moves to $x_{t-1}$ (\sk{type} {\bf (F)}), then the red particle jumps from site $x_{t-1}$ to $y$ (\sk{type} {\bf (S)}), and then the remaining stack of $m-1$ particles at site $x_{t-1}$ moves back to $x$ (\sk{type} {\bf (B)}).}
\end{figure}
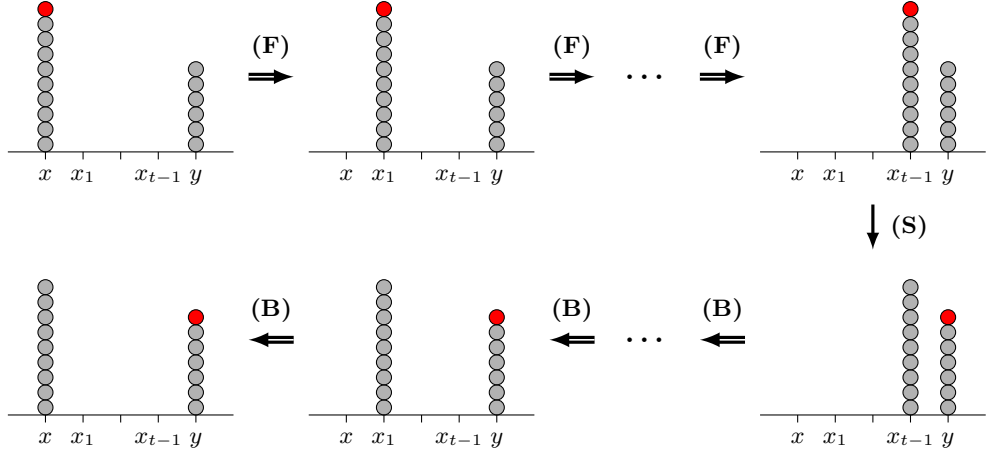

Above, letters {\bf (F)}, {(\bf S)} and {\bf (B)} stand for \textit{forward}, \textit{single} and \textit{backward}, respectively. Let $\Omega=\sk{\Omega_{x,y}^{\ell,m,\sigma}}$ denote the collection of \sk{all edges} that appear in the aforementioned path from $\sigma_{m,\ell-m}^{xy}$ to $\sigma_{m-1,\ell-m+1}^{xy}$.
\sk{Recall \eqref{eq:cK-nablaK-def}. In this subsection, we prove the following lemma:}
\begin{lemma}\label{lem:4.3}
\sk{Assume that \eqref{eq:cond-far} holds, and} suppose that $\ell\in\bbr{1}{k}$, $m\in\bbr{1}{\ell}$ and $\sigma\in\Xi_{k-\ell}^{V\setminus\{x,y\}}$ satisfy \eqref{eq:cond-empty} and \eqref{eq:cond-few}. Then,
\begin{equation}
\sk{\nabla_K^2 f ( \sigma_{m,\ell-m}^{xy} ; x,y) } \le \frac{\sk{6t \, m \tonde{ \alpha_y + \ell-m} }}{ c_{\rm min}\, \alpha_{\rm min}\, \alpha_{\rm ratio} }  \left( \sum_{\{\zeta,\zeta'\} \sk{\in } \Omega} \nabla^2f(\zeta,\zeta') \right) \fstop
\end{equation}
\end{lemma}

\begin{proof}
Applying \eqref{eq:C-S} along the path constructed above, we have the following upper bound for $\sk{\nabla_K^2 f ( \sigma_{m,\ell-m}^{xy} ; x,y) }$:
\begin{equation}\label{eq:UB-v1}
\tonde{\sum_{\{\eta,\zeta\} \sk{\in } \Omega} \nabla^2f(\eta,\zeta)} \tonde{\sum_{\{\eta,\zeta\} \sk{\in } \Omega} \frac  { \mu_{\alpha,k}(\sigma_{m,\ell-m}^{xy})\, m\tonde{\alpha_y + \ell-m} } { \fc(\eta,\zeta) } } \fstop
\end{equation}
\sk{Here, we expand the term in the second parenthesis as done exactly in \eqref{eq:sum-dec}. Namely,}
referring to the definition of $\Omega$ and \eqref{eq:omega-xn-not}, the term inside the second parenthesis in \eqref{eq:UB-v1} expands as
\begin{equation}\label{eq:2.1-1}
\begin{aligned}
& \sum_{s=1}^{t-1} \sum_{i=0}^{m-1} \frac{ \pi_x(m)\, m\tonde{\alpha_{y}+\ell-m}  } { \pi_{x_{s-1}}(m-i) \, \pi_{x_s} (i ) \, (m-i ) \tonde{\alpha_{x_s}+i  } c_{x_{s-1} x_s}} + \frac{ \pi_x(m)  } { \pi_{x_{t-1}}(m )\, c_{x_{t-1}y}  } \\
& + \sum_{s=1}^{t-1} \sum_{i=2}^{m} \frac{ \pi_x(m) \, m \tonde{\alpha_y + \ell-m}} { \pi_{x_{s-1}}(m-i ) \, \pi_{x_{s}}(i-1  ) \tonde{i-1 } (\alpha_{x_{s-1}} + m-i  )  \, c_{x_s x_{s-1}}   } \fstop
\end{aligned}
\end{equation}
\sk{Here, the first double summation indicates the forward jumps in \textbf{(F)}, the second single term indicates the jump in \textbf{(S)}, and the last double summation indicates the backward jumps in \textbf{(B)}.}
First, consider the first (double) summation in \eqref{eq:2.1-1}.
\sk{
According to \eqref{eq:omega-prop} in Lemma \ref{lem:omega-prop},
\begin{equation}
\pi_x(m) \, m \le \alpha_x \, e^{\alpha_x \sum_{j=1}^{m-1} \frac1j} \comma
\quad \pi_{x_{s-1}}(m-i) \, (m-i) \ge \alpha_{x_{s-1}} \comma
\quad \pi_{x_s}(i) \, (\alpha_{x_s}+i) \ge \alpha_{x_s} \fstop
\end{equation}
Since $c_{x_{s-1}x_s} \ge c_{\rm min}$ from \eqref{alpha-c-def}, we may bound the first double summation in \eqref{eq:2.1-1} from above by
}
\begin{equation}\label{eq:2.1-1-1}
\frac1{c_{\rm min}} \sum_{s=1}^{t-1} \sum_{i=0}^{m-1} \frac {  \alpha_x \, e^{\alpha_x \sum_{j=1}^{m-1} \frac1j } \tonde{\alpha_y + \ell-m } } { \alpha_{x_{s-1}}  \, \alpha_{x_s}   } \le \frac {e \, \sk{m \tonde{\alpha_y + \ell-m} (t-1) } } { c_{\rm min}\, \alpha_{\rm min}\, \alpha_{\rm ratio} }   \comma
\end{equation}
where we used that $\alpha_x  \sum_{j=1}^{m-1} \frac1j   \le 1$ from \eqref{eq:cond-few},
\sk{$\alpha_x / \alpha_{x_s} \le \alpha_{\rm ratio}^{-1}$ from \eqref{alpha-c-def},
and $\alpha_{x_{s-1}} \ge \alpha_{\rm min}$.}
Next, the second term in \eqref{eq:2.1-1} can be similarly bounded as
\begin{equation}\label{eq:2.1-1-2}
\frac{ \pi_x(m)  } { \pi_{x_{t-1}}(m )\, c_{x_{t-1}y}  }
\le \frac{ \alpha_x \, e^{\alpha_x \sum_{j=1}^{m-1} \frac1j } } {  \alpha_{x_{t-1}}  \, c_{\rm min}} \le \frac { e } { c_{\rm min} \,\alpha_{\rm ratio} }
\sk{\le \frac{e \, m \tonde{\alpha_y + \ell-m}}{c_{\rm min} \, \alpha_{\rm ratio} \, \alpha_{\rm min}} } \fstop
\end{equation}
Similarly, the third (double) summation in \eqref{eq:2.1-1} is bounded by 
\begin{equation}\label{eq:2.1-1-3}
\frac1{c_{\rm min}} \sum_{s=1}^{t-1} \sum_{i=2}^{m} \frac {   \alpha_x \, e^{\alpha_x \sum_{j=1}^{m-1} \frac1j } \tonde{\alpha_y + \ell-m } } {  \alpha_{x_{s-1}}  \,  \alpha_{x_s}   } \le \frac {e \, \sk{m \tonde{\alpha_y + \ell-m} (t-1) } } { c_{\rm min}\, \alpha_{\rm min}\, \alpha_{\rm ratio} }   \fstop
\end{equation}
Collecting \eqref{eq:2.1-1}, \eqref{eq:2.1-1-1}, \eqref{eq:2.1-1-2} and \eqref{eq:2.1-1-3},
\sk{and applying that $e \le 3$,}
the term inside the second parenthesis in \eqref{eq:UB-v1} is bounded by
\begin{equation}
\frac{\sk{9t \, m \tonde{ \alpha_y + \ell-m} }}{ c_{\rm min} \,\alpha_{\rm min}\, \alpha_{\rm ratio} } \fstop
\end{equation}
This concludes the proof of Lemma \ref{lem:4.3}.
\end{proof}

\subsection{Decomposition of gradients: empty sites \& many particles}\label{sec4.4}
In this subsection, we assume that \eqref{eq:cond-far}, \eqref{eq:cond-empty} hold and that (cf.\ \eqref{eq:cond-few})
\begin{equation}\label{eq:cond-many}
\alpha_x \sum_{j=1}^{m-1} \frac1j  >1  \fstop
\end{equation}
In this case, constructing a single one-dimensional path is insufficient to bound \sk{the term $\sk{\nabla_K^2 f ( \sigma_{m,\ell-m}^{xy} ; x,y) }$} properly. The main reason is that in the inequality in \eqref{eq:2.1-1-1}, we are not able to bound the term $e^{\alpha_x \sum_{j=1}^{m-1}\frac1j }$ with a constant in this new regime \eqref{eq:cond-many}. Having in mind that this term essentially comes from the mechanism that all $m$ particles at $x$ move together,  we now avoid this obstacle by sending only a limited amount of particles along the path. This replacement costs us an additional $\alpha_x^{-1}$   term in the upper bound, which is now manageable by the new condition $\alpha_x \sum_{j=1}^{m-1}\frac1j >1$; we refer to \eqref{eq6} for the exact place where this point is effectively exploited.

Define a new integer $\widetilde{m}$ as 
\begin{equation}\label{eq:m-tilde-def}
\widetilde{m}\eqdef \left\lfloor \frac{m}2  \right\rfloor \comma
\end{equation}
where $\lfloor a \rfloor$ is the greatest integer less than or equal to $a \in \R$. Then, we consider a two-dimensional system of paths from $\sigma_{m,\ell-m}^{xy}$ to $\sigma_{m-1,\ell-m+1}^{xy}$ as depicted in Figure \ref{fig3}.

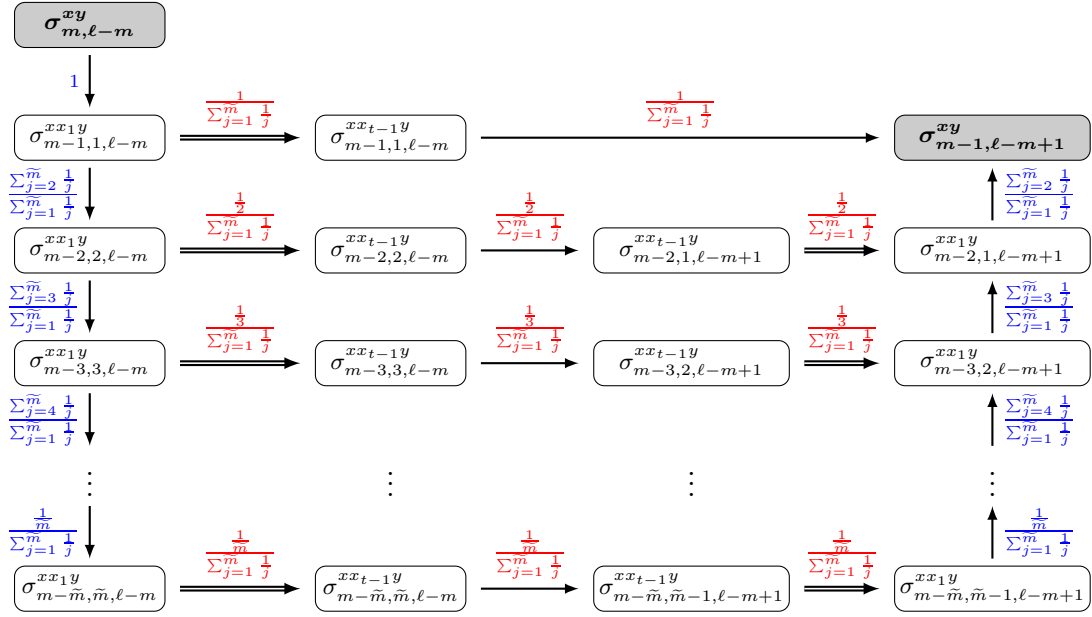
\begin{figure}
\begin{tikzpicture}
\fill[black!20!white,rounded corners] (-1,-0.3) rectangle (1,0.3); \draw[rounded corners] (-1,-0.3) rectangle (1,0.3); \draw (0,0) node[]{\footnotesize $\boldsymbol{\sigma_{m,\ell-m}^{xy}}$};
\fill[black!0!white,rounded corners] (-1,-1.8) rectangle (1,-1.2); \draw[rounded corners] (-1,-1.8) rectangle (1,-1.2); \draw (0,-1.5) node[]{\footnotesize $\sigma_{m-1,1,\ell-m}^{xx_1y}$};
\fill[black!0!white,rounded corners] (-1,-3.3) rectangle (1,-2.7); \draw[rounded corners] (-1,-3.3) rectangle (1,-2.7); \draw (0,-3) node[]{\footnotesize $\sigma_{m-2,2,\ell-m}^{xx_1y}$};
\fill[black!0!white,rounded corners] (-1,-4.8) rectangle (1,-4.2); \draw[rounded corners] (-1,-4.8) rectangle (1,-4.2); \draw (0,-4.5) node[]{\footnotesize $\sigma_{m-3,3,\ell-m}^{xx_1y}$};
\draw (0,-6) node[]{$\vdots$};
\fill[black!0!white,rounded corners] (-1,-7.8) rectangle (1,-7.2); \draw[rounded corners] (-1,-7.8) rectangle (1,-7.2); \draw (0,-7.5) node[]{\footnotesize $\sigma_{m-\widetilde{m},\widetilde{m},\ell-m}^{xx_1y}$};

\fill[black!0!white,rounded corners] (3,-1.8) rectangle (5,-1.2); \draw[rounded corners] (3,-1.8) rectangle (5,-1.2); \draw (4,-1.5) node[]{\footnotesize $\sigma_{m-1,1,\ell-m}^{xx_{t-1}y}$};
\fill[black!0!white,rounded corners] (3,-3.3) rectangle (5,-2.7); \draw[rounded corners] (3,-3.3) rectangle (5,-2.7); \draw (4,-3) node[]{\footnotesize $\sigma_{m-2,2,\ell-m}^{xx_{t-1}y}$};
\fill[black!0!white,rounded corners] (3,-4.8) rectangle (5,-4.2); \draw[rounded corners] (3,-4.8) rectangle (5,-4.2); \draw (4,-4.5) node[]{\footnotesize $\sigma_{m-3,3,\ell-m}^{xx_{t-1}y}$};
\draw (4,-6) node[]{$\vdots$};
\fill[black!0!white,rounded corners] (3,-7.8) rectangle (5,-7.2); \draw[rounded corners] (3,-7.8) rectangle (5,-7.2); \draw (4,-7.5) node[]{\footnotesize $\sigma_{m-\widetilde{m},\widetilde{m},\ell-m}^{xx_{t-1}y}$};

\fill[black!0!white,rounded corners] (6.7,-3.3) rectangle (9.3,-2.7); \draw[rounded corners] (6.7,-3.3) rectangle (9.3,-2.7); \draw (8,-3) node[]{\footnotesize $\sigma_{m-2,1,\ell-m+1}^{xx_{t-1}y}$};
\fill[black!0!white,rounded corners] (6.7,-4.8) rectangle (9.3,-4.2); \draw[rounded corners] (6.7,-4.8) rectangle (9.3,-4.2); \draw (8,-4.5) node[]{\footnotesize $\sigma_{m-3,2,\ell-m+1}^{xx_{t-1}y}$};
\draw (8,-6) node[]{$\vdots$};
\fill[black!0!white,rounded corners] (6.7,-7.8) rectangle (9.3,-7.2); \draw[rounded corners] (6.7,-7.8) rectangle (9.3,-7.2); \draw (8,-7.5) node[]{\footnotesize $\sigma_{m-\widetilde{m},\widetilde{m}-1,\ell-m+1}^{xx_{t-1}y}$};

\fill[black!20!white,rounded corners] (10.7,-1.8) rectangle (13.3,-1.2); \draw[rounded corners] (10.7,-1.8) rectangle (13.3,-1.2); \draw (12,-1.5) node[]{\footnotesize $\boldsymbol{\sigma_{m-1,\ell-m+1}^{xy}}$};
\fill[black!0!white,rounded corners] (10.7,-3.3) rectangle (13.3,-2.7); \draw[rounded corners] (10.7,-3.3) rectangle (13.3,-2.7); \draw (12,-3) node[]{\footnotesize $\sigma_{m-2,1,\ell-m+1}^{xx_1y}$};
\fill[black!0!white,rounded corners] (10.7,-4.8) rectangle (13.3,-4.2); \draw[rounded corners] (10.7,-4.8) rectangle (13.3,-4.2); \draw (12,-4.5) node[]{\footnotesize $\sigma_{m-3,2,\ell-m+1}^{xx_1y}$};
\draw (12,-6) node[]{$\vdots$};
\fill[black!0!white,rounded corners] (10.7,-7.8) rectangle (13.3,-7.2); \draw[rounded corners] (10.7,-7.8) rectangle (13.3,-7.2); \draw (12,-7.5) node[]{\footnotesize $\sigma_{m-\widetilde{m},\widetilde{m}-1,\ell-m+1}^{xx_1y}$};

\draw[thick,-latex] (0,-0.4)--(0,-1.1); \draw (0,-0.75) node[left]{\color{blue} \tiny $1$};
\draw[thick,-latex] (0,-1.9)--(0,-2.6); \draw (0,-2.25) node[left]{\color{blue} \tiny $\frac { \sum_{j=2}^{\widetilde m} \frac1{j} } { \sum_{j=1}^{\widetilde m} \frac1{j} }$};
\draw[thick,-latex] (0,-3.4)--(0,-4.1); \draw (0,-3.75) node[left]{\color{blue} \tiny $\frac { \sum_{j=3}^{\widetilde m} \frac1{j} } { \sum_{j=1}^{\widetilde m} \frac1{j} }$};
\draw[thick,-latex] (0,-4.9)--(0,-5.6); \draw (0,-5.25) node[left]{\color{blue} \tiny $\frac { \sum_{j=4}^{\widetilde m} \frac1{j} } { \sum_{j=1}^{\widetilde m} \frac1{j} }$};
\draw[thick,-latex] (0,-6.4)--(0,-7.1); \draw (0,-6.75) node[left]{\color{blue} \tiny $\frac { \frac1{\widetilde m} } { \sum_{j=1}^{\widetilde m} \frac1{j} }$};

\draw[thick,-latex,double] (1.2,-1.5)--(2.8,-1.5); \draw (2,-1.5) node[above]{\color{red} \tiny $\frac { 1 } { \sum_{j=1}^{\widetilde m} \frac1{j} }$};
\draw[thick,-latex,double] (1.2,-3)--(2.8,-3); \draw (2,-3) node[above]{\color{red} \tiny $\frac { \frac12 } { \sum_{j=1}^{\widetilde m} \frac1{j} }$};
\draw[thick,-latex,double] (1.2,-4.5)--(2.8,-4.5); \draw (2,-4.5) node[above]{\color{red} \tiny $\frac { \frac13 } { \sum_{j=1}^{\widetilde m} \frac1{j} }$};
\draw[thick,-latex,double] (1.2,-7.5)--(2.8,-7.5); \draw (2,-7.5) node[above]{\color{red} \tiny $\frac { \frac1{\widetilde{m}} } { \sum_{j=1}^{\widetilde m} \frac1{j} }$};

\draw[thick,-latex] (5.2,-1.5)--(10.5,-1.5); \draw (7.85,-1.5) node[above]{\color{red} \tiny $\frac { 1 } { \sum_{j=1}^{\widetilde m} \frac1{j} }$};
\draw[thick,-latex] (5.2,-3)--(6.5,-3); \draw (5.85,-3) node[above]{\color{red} \tiny $\frac { \frac12 } { \sum_{j=1}^{\widetilde m} \frac1{j} }$};
\draw[thick,-latex] (5.2,-4.5)--(6.5,-4.5); \draw (5.85,-4.5) node[above]{\color{red} \tiny $\frac { \frac13 } { \sum_{j=1}^{\widetilde m} \frac1{j} }$};
\draw[thick,-latex] (5.2,-7.5)--(6.5,-7.5); \draw (5.85,-7.5) node[above]{\color{red} \tiny $\frac { \frac1{\widetilde{m}} } { \sum_{j=1}^{\widetilde m} \frac1{j} }$};

\draw[thick,-latex,double] (9.5,-3)--(10.5,-3); \draw (10,-3) node[above]{\color{red} \tiny $\frac { \frac12 } { \sum_{j=1}^{\widetilde m} \frac1{j} }$};
\draw[thick,-latex,double] (9.5,-4.5)--(10.5,-4.5); \draw (10,-4.5) node[above]{\color{red} \tiny $\frac { \frac13 } { \sum_{j=1}^{\widetilde m} \frac1{j} }$};
\draw[thick,-latex,double] (9.5,-7.5)--(10.5,-7.5); \draw (10,-7.5) node[above]{\color{red} \tiny $\frac { \frac1{\widetilde{m}} } { \sum_{j=1}^{\widetilde m} \frac1{j} }$};

\draw[thick,latex-] (12,-1.9)--(12,-2.6); \draw (12,-2.25) node[right]{\color{blue} \tiny $\frac { \sum_{j=2}^{\widetilde m} \frac1{j} } { \sum_{j=1}^{\widetilde m} \frac1{j} }$};
\draw[thick,latex-] (12,-3.4)--(12,-4.1); \draw (12,-3.75) node[right]{\color{blue} \tiny $\frac { \sum_{j=3}^{\widetilde m} \frac1{j} } { \sum_{j=1}^{\widetilde m} \frac1{j} }$};
\draw[thick,latex-] (12,-4.9)--(12,-5.6); \draw (12,-5.25) node[right]{\color{blue} \tiny $\frac { \sum_{j=4}^{\widetilde m} \frac1{j} } { \sum_{j=1}^{\widetilde m} \frac1{j} }$};
\draw[thick,latex-] (12,-6.4)--(12,-7.1); \draw (12,-6.75) node[right]{\color{blue} \tiny $\frac { \frac1{\widetilde{m}} } { \sum_{j=1}^{\widetilde m} \frac1{j} }$};
\end{tikzpicture}\caption{\label{fig3}Two-dimensional system of paths from $\sigma_{m,\ell-m}^{xy}$ (top left)
to $\sigma_{m-1,\ell-m+1}^{xy}$ (top	 right) explained in Section \ref{sec4.4}. Double arrows indicate
series of consecutive paths. Moreover, the blue (resp.\ red) numbers indicate the corresponding value of the flow $\varphi$ along the path in the vertical (resp.\ horizontal) direction, as defined in \eqref{eq:flow-def-vertical} and \eqref{eq:flow-def-horizontal}. The values are chosen so as to ensure that $\varphi$ is  a unit flow from $\sigma_{m,\ell-m}^{xy}$ to $\sigma_{m-1,\ell-m+1}^{xy}$.}
\end{figure}

In  Figure \ref{fig3}, the downward vertical arrows starting from $\sigma_{m,\ell-m}^{xy}$ in the top left corner denote the $\widetilde{m}$ consecutive particle jumps from $x$ to $x_1$. Then, for $i \in \bbr{1}{\widetilde{m}}$, the $i$-th bold arrow on the left, from $\sigma_{m-i,i,\ell-m}^{xx_1y}$  to $\sigma_{m-i,i,\ell-m}^{xx_{t-1}y}$, denote the movements of $i$ particles from $x_{s-1}$ to $x_s$ consecutively for each $s\in\bbr{2}{t-1}$, as in mechanism {\bf (F)} in Section \ref{sec4.3}. Next, the middle horizontal arrows denote the single particle jump from $x_{t-1}$ to $y$, corresponding to mechanism {\bf (S)} in Section \ref{sec4.3}. The bold arrows in the right part denote the movements of $i-1$ particles from $x_{t-1}$ back to $x_1$ for each $i \in \bbr{2}{\widetilde{m}}$ (backwards), as in mechanism {\bf (B)} in Section \ref{sec4.3}. Finally, the upward vertical arrows below $\sigma_{m-1,\ell-m+1}^{xy}$ in the top right corner denote the $\widetilde{m}-1$ consecutive particle jumps from $x_1$ to $x$.
The values attached to the arrows, roughly speaking, play the role of weights (encoded as the value of the flow, see below) for each move.

\sk{As before, let $\Omega = \Omega_{x,y}^{\ell,m,\sigma}$ be the collection of all edges that appear in Figure \ref{fig3}. The objective of this subsection is to prove the following lemma:}

\begin{lemma}\label{lem:4.4}
\sk{Assume that \eqref{eq:cond-far} holds, and} suppose that $\ell\in\bbr{1}{k}$, $m\in\bbr{1}{\ell}$
and $\sigma\in\Xi_{k-\ell}^{V\setminus\{x,y\}}$ satisfy \eqref{eq:cond-empty} and \eqref{eq:cond-many}.
Then, we have
\begin{equation}
\sk{\nabla_K^2 f ( \sigma_{m,\ell-m}^{xy} ; x,y) } \le  \frac { \sk{4t \, m \tonde{\alpha_y + \ell-m}} \, 6^{\alpha_{\rm max} }} {\alpha_{\rm min} \, \alpha_{\rm ratio} \, c_{\rm min}}  \left( \sum_{\{\zeta,\zeta'\} \sk{\in} \Omega} \nabla^2f(\zeta,\zeta') \right) \fstop
\end{equation}
\end{lemma}

To derive an estimate similar to \eqref{eq:UB-v1} that works also in this two-dimensional setting, we make use of some potential-theoretic arguments. We present the statements as concisely as possible and refer the readers to, e.g., \cite[Chapter 7]{bovier-den_hollander-metastability} or \cite[Appendix A]{kim_hierarchical_2023}   for more details about the objects used in this subsection.

Only in this subsection, we restrict the original system ${\rm SIP}_k(G,\alpha)$ to the subset \sk{of the configurations that appear} in Figure \ref{fig3}\sk{, which we denote by $\hat\Xi$}. Moreover, let $\cD_{\Omega}(g)$ denote the (normalized) Dirichlet form in $\sk{\hat\Xi_k}$
given by (cf.\ \eqref{eq:c-nabla-def})
\begin{equation}\label{eq:restricted-Diri-def}
\cD_\Omega(g)=
\frac12 \sum_{\eta,\zeta\in\sk{\hat\Xi_k}} \mu_{\alpha,k}(\eta)\, r_{\alpha,k}(\eta,\zeta) \tonde{g(\zeta)-g(\eta)}^2 =
\sum_{\{\eta,\zeta\} \sk{\in } \Omega} \nabla^2g(\eta,\zeta)\fstop
\end{equation}
For non-empty disjoint subsets $A,B$ of $\Omega$, introduce    (e.g., \cite[Theorem 7.33]{bovier-den_hollander-metastability} or \cite[Proposition A.1]{kim_hierarchical_2023})
\begin{equation}
	{\rm cap}(A,B)\eqdef \inf\left\{\cD_\Omega(g): g\in \R^{\sk{\hat\Xi_k}}\ \text{such that}\ g = 1\ \text{on}\ A\ \text{and}\ g = 0\ \text{on}\ B\right\} \comma
\end{equation}
the {capacity} between $A$ and $B$. In particular, for $A=\{\sigma_{m,\ell-m}^{xy}\}$ and $B=\{\sigma_{m-1,\ell-m+1}^{xy}\}$,  we have, after shifting and renormalizing,
\begin{equation}\label{eq:Diri-prin}
{\rm cap}(\sigma_{m,\ell-m}^{xy},\sigma_{m-1,\ell-m+1}^{xy})\, \tttonde{f(\sigma_{m-1,\ell-m+1}^{xy})-f(\sigma_{m,\ell-m}^{xy})}^{2} \le \cD_\Omega(f) \comma
\end{equation}
where $f\in \R^{\Xi_k}$ is the function which we fixed at the beginning of Section \ref{sec4}.

\begin{proof}[\sk{Proof of Lemma \ref{lem:4.4}}]

Let us describe the flow $\varphi$ on $\sk{\hat\Xi_k}$ represented in Figure \ref{fig3}, where we recall that  \textquotedblleft flow on $\Omega$\textquotedblright\ refers to any antisymmetric real-valued function on $\sk{\hat\Xi_k} \times \sk{\hat\Xi_k}$. Due to this antisymmetry, we shall omit to specify the value of $\varphi$ when exchanging the arguments. Hence, vertically,  define
\begin{equation}\label{eq:flow-def-vertical}
	\begin{aligned}
\varphi(\sigma_{m-i,i,\ell-m}^{xx_1y},\sigma_{m-i-1,i+1,\ell-m}^{xx_1y}) &\eqdef \frac { \sum_{j=i+1}^{\widetilde m} \frac1j } { \sum_{j=1}^{\widetilde m} \frac1j }\comma &&\quad \text{for}\ i\in\bbr{0}{\widetilde{m}-1} \comma \\
\varphi(\sigma_{m-i-1,i,\ell-m+1}^{x\sk{x_1}y},\sigma_{m-i,i-1,\ell-m+1}^{x\sk{x_1}y}) &\eqdef \frac { \sum_{j=i+1}^{\widetilde m} \frac1j } { \sum_{j=1}^{\widetilde m} \frac1j }\comma &&\quad \text{for}\ i\in\bbr{1}{\widetilde{m}-1} \fstop
\end{aligned}
\end{equation}
Horizontally, define $\varphi$ as a constant flow at each $i$-th vertical level from $\sigma_{m-i,i,\ell-m}^{xx_1y}$ to $\sigma_{m-i,i-1,\ell-m+1}^{xx_1y}$ with value
\begin{equation}\label{eq:flow-def-horizontal}
\frac{\frac{1}{i}}{ \sum_{j=1}^{\widetilde m} \frac1j }\comma\quad \text{for each}\ i \in \bbr{1}{\widetilde{m}}\fstop
\end{equation}
The function $\varphi$ is a unit flow from $\sigma_{m,\ell-m}^{xy}$ to $\sigma_{m-1,\ell-m+1}^{xy}$. Indeed, letting $({\rm div}\,\varphi)(\eta)$ denote the net divergence at $\eta$, i.e., $({\rm div}\,\varphi)(\eta) \eqdef \sum_{\zeta \in \Omega} \varphi(\eta,\zeta)$, we have
\begin{equation}\label{eq:unit-flow}
({\rm div}\,\varphi)(\sigma_{m,\ell-m}^{xy}) = 1 \comma \quad ({\rm div}\,\varphi)(\sigma_{m-1,\ell-m+1}^{xy}) = -1 \quad \text{and} \quad {\rm div}\,\varphi = 0 \quad \text{otherwise}  \fstop
\end{equation}
Thus, the standard Thomson principle (e.g.,  \cite[Theorem 7.37]{bovier-den_hollander-metastability} or \cite[Proposition A.2]{kim_hierarchical_2023}) indicates that
\begin{equation}\label{eq:Thom-prin}
{\rm cap}(\sigma_{m,\ell-m}^{xy},\sigma_{m-1,\ell-m+1}^{xy})\ge\frac{1}{\|\varphi\|^{2}} \comma
\end{equation}
where $\|\varphi\|^2$ is the (square) flow norm given by	
\begin{equation}
\|\varphi\|^2 \eqdef \sum_{ \sk{\{ \eta,\zeta \} } \in \Omega } \frac{\varphi(\eta,\zeta)^2}{\fc(\eta,\zeta)} \fstop
\end{equation}
Hence, by \eqref{eq:cK-nablaK-def}, \eqref{eq:Diri-prin} and \eqref{eq:Thom-prin},  we obtain 	
\begin{equation}\label{eq1}
\sk{\nabla_K^2 f ( \sigma_{m,\ell-m}^{xy} ; x,y) } \le \cD_\Omega(f)\, \mu_{\alpha,k}(\sigma_{m,\ell-m}^{xy})\,m\tonde{\alpha_y+\ell-m} \|\varphi\|^2 \fstop
\end{equation}
Next, we upper bound the right-hand side of \eqref{eq1}, except $\cD_\Omega(f)$. For this purpose, we decompose $\varphi = \varphi_1 + \varphi_2$, where $\varphi_1$ (resp.\ $\varphi_2$) indicates the vertical (resp.\ horizontal) part of $\varphi$. First, the (square) flow norm of the vertical part $\varphi_1$ reads as
\begin{equation}\label{eq:flow-norm-v}
\|\varphi_1\|^2 = \sum_{i=0}^{\widetilde{m}-1} \frac{( \sum_{j=i+1}^{\widetilde{m}}\frac1j  )^{2}/( \sum_{j=1}^{\widetilde{m}}\frac1j )^{2}}{ \fc(\sigma_{m-i,i,\ell-m}^{xx_{1}y},\sigma_{m-i-1,i+1,\ell-m}^{xx_{1}y}) }
+\sum_{i=1}^{\widetilde{m}-1} \frac{( \sum_{j=i+1}^{\widetilde{m}} \frac1{j} )^{2}/( \sum_{j=1}^{\widetilde{m}}\frac1{j} )^{2}}{ \fc(\sigma_{m-i-1,i,\ell-m+1}^{x\sk{x_1}y},\sigma_{m-i,i-1,\ell-m+1}^{x\sk{x_1}y}) }\fstop
\end{equation}
Since the two  numerators are clearly bounded by $1$, we further get
\begin{equation}\label{eq2}
\begin{aligned}
& \mu_{\alpha,k}(\sigma_{m,\ell-m}^{xy})\,m\tonde{\alpha_y+\ell-m} \|\varphi_1\|^2 \\ 
& \le \sum_{i=0}^{\widetilde{m}-1} \frac{ \mu_{\alpha,k}(\sigma_{m,\ell-m}^{xy})\,m\tonde{\alpha_y+\ell-m} } { \fc(\sigma_{m-i,i,\ell-m}^{xx_{1}y},\sigma_{m-i-1,i+1,\ell-m}^{xx_{1}y}) }
+\sum_{i=1}^{\widetilde{m}-1} \frac{ \mu_{\alpha,k}(\sigma_{m,\ell-m}^{xy})\,m\tonde{\alpha_y+\ell-m} } { \fc(\sigma_{m-i-1,i,\ell-m+1}^{x\sk{x_1}y},\sigma_{m-i,i-1,\ell-m+1}^{x\sk{x_1}y}) }\fstop
\end{aligned}
\end{equation}
By the definitions in \eqref{eq:c-nabla-def} and \eqref{eq:omega-xn-not},
\sk{
the summand in the first summation in \eqref{eq2} becomes
\begin{equation}
\frac{ \mu_{\alpha,k}(\sigma_{m,\ell-m}^{xy})\,m\tonde{\alpha_y+\ell-m} } { \mu_{\alpha,k} (\sigma_{m-i,i,\ell-m}^{xx_{1}y}) \, r_{\alpha,k} (\sigma_{m-i,i,\ell-m}^{xx_{1}y},\sigma_{m-i-1,i+1,\ell-m}^{xx_{1}y}) } = \frac{ \pi_x(m)\,  m \tonde{\alpha_y + \ell - m} } { \pi_x(m-i)\, \pi_{x_1}(i) \tonde{m-i} \tonde{\alpha_{x_1} + i  } c_{xx_1} } \fstop
\end{equation}
Thus, by applying \eqref{eq:omega-prop} and \eqref{eq:omega-prop-2},
}
the first summation in \eqref{eq2} can be bounded from above by
\begin{align}
& \sum_{i=0}^{\widetilde{m}-1} \frac{ \pi_x(m)\,  m \tonde{\alpha_y + \ell - m} } { \pi_x(m-i)\, \pi_{x_1}(i) \tonde{m-i} \tonde{\alpha_{x_1} + i  } c_{xx_1} } \\
& \le \sum_{i=0}^{\widetilde{m}-1}  e^{\alpha_x \sum_{j=m-i}^{m-1}\frac1j } \, \frac{\alpha_y + \ell-m}{ \alpha_{x_1} \, c_{\rm min}  }  \le \frac{\sk{\tilde m \tonde{\alpha_y + \ell-m}}}{\alpha_{\rm min} \, c_{\rm min} } \, e^{\alpha_x \sum_{j=m- \widetilde{m}+1}^{m-1}\frac1j }  \fstop
\end{align}
Noting that $\sum_{j=m-\widetilde m+1}^{m-1}\frac1j \le \log \frac{m-1}{m-\widetilde{m}} \le \log 2$ (cf.\ \eqref{eq:m-tilde-def}), we deduce that
\begin{equation}\label{eq3}
\sum_{i=0}^{\widetilde{m}-1} \frac{ \mu_{\alpha,k}(\sigma_{m,\ell-m}^{xy})\,m\tonde{\alpha_y+\ell-m} } { \fc(\sigma_{m-i,i,\ell-m}^{xx_{1}y},\sigma_{m-i-1,i+1,\ell-m}^{xx_{1}y}) } \le \frac{ \sk{\tilde m \tonde{\alpha_y + \ell-m} } \, 2^{\alpha_{\rm max}}}{\alpha_{\rm min} \, c_{\rm min} } \fstop
\end{equation}
Similarly, the second summation in \eqref{eq2} can be bounded by the same value in the right-hand side of \eqref{eq3}. Thus, collecting \eqref{eq2} and \eqref{eq3} \sk{along with $2 \tilde m \le m$}, we deduce
\begin{equation}\label{eq5}
\mu_{\alpha,k}(\sigma_{m,\ell-m}^{xy}) \, m \tonde{\alpha_y+\ell-m}  \|\varphi_1 \|^2 \le \frac {\sk{m \tonde{\alpha_y + \ell-m} 2^{\alpha_{\rm max}}}}{\alpha_{\rm min} \, c_{\rm min} } \fstop
\end{equation}

Next, the horizontal part $\mu_{\alpha,k}(\sigma_{m,\ell-m}^{xy}) \, m \tonde{\alpha_y+\ell-m}  \|\varphi_2\|^2$ becomes
\begin{equation}\label{eq:flow-norm-h}
\begin{aligned}
& \sum_{i=1}^{\widetilde{m}} \sum_{s=2}^{t-1} \sum_{i'=0}^{i-1} \frac { \mu_{\alpha,k}(\sigma_{m,\ell-m}^{xy}) \,m \tonde{\alpha_y+\ell-m}  \frac{1}{i^{2}}/( \sum_{j=1}^{\widetilde m}\frac1j )^{2} } { \fc(\sigma_{m-i,i-i',i',\ell-m}^{xx_{s-1}x_{s}y},\sigma_{m-i,i-i'-1,i'+1,\ell-m}^{xx_{s-1}x_{s}y}) } \\
& + \sum_{i=1}^{\widetilde{m}}\frac{ \mu_{\alpha,k}(\sigma_{m,\ell-m}^{xy}) \,m \tonde{\alpha_y+\ell-m} \frac{1}{i^{2}}/( \sum_{j=1}^{\widetilde m}\frac1j )^{2}}{\fc(\sigma_{m-i,i,\ell-m}^{xx_{t-1}y},\sigma_{m-i,i-1,\ell-m+1}^{xx_{t-1}y})} \\
 & + \sum_{i=2}^{\widetilde{m}} \sum_{s=2}^{t-1} \sum_{i'=1}^{i-1}\frac{ \mu_{\alpha,k}(\sigma_{m,\ell-m}^{xy}) \,m \tonde{\alpha_y+\ell-m} \frac{1}{i^{2}}/( \sum_{j=1}^{\widetilde m}\frac1j )^{2}}{\fc(\sigma_{m-i,i-i'-1,i',\ell-m+1}^{xx_{s-1}x_sy},\sigma_{m-i,i-i',i'-1,\ell-m+1}^{xx_{s-1}x_sy})} \fstop
\end{aligned}
\end{equation}
\sk{According to \eqref{eq:c-nabla-def} and \eqref{eq:omega-xn-not}, the first (triple) summation in \eqref{eq:flow-norm-h} equals
\begin{equation}
\sum_{i=1}^{\widetilde{m}} \sum_{s=2}^{t-1} \sum_{i'=0}^{i-1} \frac{ \pi_{x}(m)  \, m \tonde{\alpha_{y}+\ell-m} \frac{1}{i^{2}}/( \sum_{j=1}^{\widetilde m}\frac1j )^{2}  } { \pi_{x}(m-i) \, \pi_{x_{s-1}}(i-i' ) \, \pi_{x_{s}}(i'  )  \tonde{i-i' } \tonde{\alpha_{x_{s}}+i' } \, c_{x_{s-1}x_{s}}  } \fstop
\end{equation}
By Lemma \ref{lem:omega-prop}, $\pi_{x_{s-1}}(i-i') \, (i-i') \ge \alpha_{x_{s-1}} \ge \alpha_{\rm min}$, $\pi_{x_s} (i') \, (\alpha_{x_s} + i') \ge \alpha_{\rm min}$, and
\begin{equation}
\frac{\pi_x(m)\,m}{\pi_x(m-i)} \le (m-i) \, e^{\alpha_x \sum_{j=m-i}^{m-1} \frac1j} \le m \, e^{\alpha_x \sum_{j=m-i}^{m-1} \frac1j} \fstop
\end{equation}
Substituting these to the penultimate display, along with $c_{x_{s-1}x_s} \ge c_{\rm min}$,
}
the first (triple) summation in \eqref{eq:flow-norm-h} \sk{is upper bounded by}
\begin{equation}
\frac {\sk{m \tonde{\alpha_y + \ell-m}}}{\alpha_{\rm min}^2 \, c_{\rm min} }  \sum_{i=1}^{\widetilde{m}} \sum_{s=2}^{t-1} \sum_{i'=0}^{i-1}  \frac{e^{\alpha_x \sum_{j=m-i}^{m-1}\frac1j }}{i^{2} \, ( \sum_{j=1}^{\widetilde m}\frac1j )^{2}} \le
\frac { \sk{m \tonde{\alpha_y + \ell-m} (t-2)} \, e^{\alpha_x \sum_{j=m-\widetilde{m}}^{m-1}\frac1j }}{ {\alpha_{\rm min}^2 \, c_{\rm min} } \sum_{j=1}^{\widetilde m}\frac1j  } \fstop
\end{equation}
Since $\sum_{j=m-\widetilde{m}}^{m-1}\frac1j \le 1+\log 2 < \log 6$ and $\sum_{j=1}^{\widetilde m}\frac1j \ge \frac12 \sum_{j=1}^{m-1}\frac1j$ (cf.\ \eqref{eq:m-tilde-def}), we may further bound the right-hand side with
\begin{equation}\label{eq6}
\frac{2 \, \sk{m \tonde{\alpha_y + \ell-m} (t-2) } \, e^{\alpha_x  \log 6 }}{\alpha_{\rm min} \, \alpha_{\rm ratio} \, c_{\rm min} \, \alpha_x \sum_{j=1}^{m-1}\frac1j  } \le
\frac{2 \, \sk{m \tonde{\alpha_y + \ell-m} (t-2) } \, 6^{\alpha_{\rm max} }}{\alpha_{\rm min} \, \alpha_{\rm ratio} \, c_{\rm min}} \comma
\end{equation}
where in the inequality we used  condition \eqref{eq:cond-many}. Similarly, the third (triple) summation in \eqref{eq:flow-norm-h} is bounded by the same value in the right-hand side of \eqref{eq6}. Finally, \sk{by the same ideas,} the second summation in \eqref{eq:flow-norm-h} is bounded from above by
\begin{equation}\label{eq8}\sk{
\frac{m}{\alpha_{\rm min} \, c_{\rm min}} \sum_{i=1}^{\tilde m} \frac{6^{\alpha_{\rm max}}}{i^2 \tonde{\sum_{j=1}^{\tilde m}\frac1j}^2}
\le \frac{ m \tonde{\alpha_y + \ell-m} 6^{\alpha_{\rm max}}}{\alpha_{\rm  min}^2 \, c_{\rm min} \sum_{j=1}^{\tilde m}\frac1j }
\le \frac{ 2 \, m \tonde{\alpha_y + \ell-m} \, 6^{\alpha_{\rm max}}}{\alpha_{\rm  min} \, c_{\rm min} \, \alpha_{\rm ratio}  } \fstop}
\end{equation}
\sk{At the first inequality we used $i^2 \ge i$ and $\alpha_y + \ell-m \ge \alpha_{\rm min}$, and at the second inequality we used $\alpha_{\rm min} \ge \alpha_{\rm ratio} \, \alpha_x$ and \eqref{eq:cond-many}.}
Collecting \eqref{eq:flow-norm-h}, \eqref{eq6} and \eqref{eq8}, we obtain 
\begin{equation}\label{eq9}
\mu_{\alpha,k}(\sigma_{m,\ell-m}^{xy}) \,m \tonde{\alpha_y+\ell-m }   \|\varphi_2\|^2 \le \frac{\sk{(4t-6) \, m \tonde{\alpha_y + \ell-m} 6^{\alpha_{\rm max} }}}{\alpha_{\rm min} \, \alpha_{\rm ratio} \, c_{\rm min}} \fstop
\end{equation}
\sk{The constant $4t$ appears because the term in \eqref{eq6} should be counted twice to deal with the first and third lines in \eqref{eq:flow-norm-h}, and the term in \eqref{eq8} corresponds to the second line in \eqref{eq:flow-norm-h}, thus $ 2 \times 2(t-2) + 2 = 4t-6$.}
Since $\varphi_1$ and $\varphi_2$ \sk{have disjoint supports}, we have $\|\varphi\|^2 = \|\varphi_1\|^2 + \|\varphi_2\|^2$. Therefore, by \eqref{eq5} and \eqref{eq9}, since $\alpha_{\rm ratio} \le 1$, we conclude that
\begin{equation}\label{eq10}
\mu_{\alpha,k}(\sigma_{m,\ell-m}^{xy}) \, m \tonde{\alpha_y+\ell-m }  \|\varphi\|^2 \le
\frac {\sk{4t \, m \tonde{\alpha_y + \ell-m} 6^{\alpha_{\rm max}} }} {\alpha_{\rm min} \, \alpha_{\rm ratio} \, c_{\rm min}} \fstop
\end{equation}
Collecting \eqref{eq:restricted-Diri-def}, \eqref{eq1} and \eqref{eq10}, \sk{we arrive at conclusion of Lemma \ref{lem:4.4}.}
\end{proof}

\subsection{General case and proof of Theorem \ref{th:key-ing}}\label{sec4.5}
Now, we handle the general case without the restrictions imposed in the previous four subsections.
The idea is to send a particle from $x$ to $y$ along the sequence $x=x_{0},x_{1},\ldots,x_{t}=y$
by obeying the following two criteria:

\begin{itemize}
\item If consecutive sites are occupied by particles in $\sigma$, send a single particle along these occupied sites as explained in Section \ref{sec4.2}.
\item If consecutive sites are empty for the configuration $\sigma$, follow the
mechanism explained in Sections \ref{sec4.3} or \ref{sec4.4} depending
on the size of the stack of particles just before these empty sites.
\end{itemize}

To state this procedure in a rigorous manner, we decompose $\sk{\bbr{0}{t}}$ as
\begin{equation}\label{eq:seg-dec}
\bbr{\sk{0}}{b_{1}-1} \cup \bbr{ b_{1}}{a_{2}-1} \cup \cdots \cup \sk{\bbr{ b_{r-1}}{a_{r}-1} \cup \bbr{ a_{r}}{t}} \eqqcolon O_{1} \cup  E_{1} \cup \cdots \cup \sk{E_{r-1} \cup O_{r}} \comma
\end{equation}
where $1\le b_{1}<a_{2}<b_{2}<\cdots < b_{r-1}<a_{r} \le t$ and
\begin{equation}
\begin{cases}
\sigma_{x_{s}} \ge 1 & \text{if}\ s\in O_{j} \ \text{for some} \ j\in\bbr{1}{r} { \sk{, \ \text{except for} \ s=t}} \\
\sigma_{x_{s}}=0 & \text{if}\ s\in E_{j} \ \text{for some} \ j\in\bbr{1}{\sk{r-1}} \fstop
\end{cases}
\end{equation}
Here, letter $O$ (resp.\ $E$) stands for \textit{occupied} (resp.\ \textit{empty}).
Also, let $a_{1} \eqdef 0$ and $\sk{b_r} \eqdef t+1$. Now, the path from $\sigma_{m,\ell-m}^{xy}$
to $\sigma_{m-1,\ell-m+1}^{xy}$ is constructed as follows. Refer to Figure \ref{fig4} for a visual presentation of this recursive procedure.

\begin{enumerate}
\item[{\bf [I]}] Along $x_{a_j}, \ldots, x_{b_j-1}$ for $j\in\bbr{1}{r}$:
send a single particle from $x_{a_{j}}$ to $x_{b_{j}-1}$ consecutively, as explained in Section \ref{sec4.2}.
\sk{This would be a path from $\sigma_{m-1,1,\ell-m}^{xx_{a_j}y}$ to $\sigma_{m-1,1,\ell-m}^{xx_{b_j-1}y}$, where the first configuration for $j=1$ is $\sigma_{m,\ell-m}^{xy}$ and the second configuration for $j=r$ is $\sigma_{m-1,\ell-m+1}^{xy}$. Notice that this step is omitted if $b_j = a_j+1$.}
\item[{\bf [II]}] Along $x_{b_j-1}, \ldots, x_{a_{j+1}}$ for $j\in\bbr{1}{\sk{r-1}}$
if $$\alpha_{x_{b_j-1}} \sum_{j=1}^{\sigma_{x_{b_j-1}}}\frac1j  \le 1 \; :$$ we proceed as in Section \ref{sec4.3}; move all $\sigma_{x_{b_{j}-1}}+1$ particles from $x_{b_{j}-1}$
to $x_{a_{j+1}-1}$ consecutively, move a single particle from $x_{a_{j+1}-1}$
to $x_{a_{j+1}}$, and then move back the remaining $\sigma_{x_{b_{j}-1}}$
particles at $x_{a_{j+1}-1}$ to $x_{b_{j}-1}$ backwards.
\sk{This is a path from $\sigma_{m-1,1,\ell-m}^{xx_{b_j-1}y}$ to $\sigma_{m-1,1,\ell-m}^{xx_{a_{j+1}}y}$.}
\item[{\bf [III]}] Along $x_{b_j-1}, \ldots, x_{a_{j+1}}$ for $j\in\bbr{1}{\sk{r-1}}$
if $$\alpha_{x_{b_j-1}} \sum_{j=1}^{\sigma_{x_{b_j-1}}}\frac1j > 1 \; :$$ we proceed as in Section
\ref{sec4.4}; move each $i \in \bbr{1} { \lfloor(\sigma_{x_{b_{j}-1}}+1)/2\rfloor }$ particles
from $x_{b_{j}-1}$ to $x_{a_{j+1}-1}$ consecutively, move a single
particle from $x_{a_{j+1}-1}$ to $x_{a_{j+1}}$, and then move back
the remaining $i-1$ particles at $x_{a_{j+1}-1}$ to $x_{b_{j}-1}$ backwards.
\sk{As in {\bf [II]}, this gives a path from $\sigma_{m-1,1,\ell-m}^{xx_{b_j-1}y}$ to $\sigma_{m-1,1,\ell-m}^{xx_{a_{j+1}}y}$.}

\end{enumerate}

\begin{figure}
\begin{tikzpicture}
\draw[thick,red] (0,1.9) sin (0.25,2.1); \draw[thick,red] (0.25,2.1) cos (0.5,0.7);
\draw[thick,red] (0.5,0.7) sin (0.75,0.9); \draw[thick,red] (0.75,0.9) cos (1,0.3);
\draw[thick,red] (1,0.3) sin (1.25,1.1); \draw[thick,red] (1.25,1.1) cos (1.5,0.9);
\draw[thick,red] (1.5,0.9) sin (1.75,1.1); \draw[thick,red,-latex] (1.75,1.1) cos (2,0.6);
\draw (1.25,1.3) node[above]{Sec. \ref{sec4.2}};

\draw[thick,red] (2,0.5) sin (3.25,0.9); \draw[thick,red,-latex] (3.25,0.9) cos (4.5,0.4);
\draw (3.25,1) node[above]{Sec. \ref{sec4.3}};

\draw[thick,red] (4.5,0.3) sin (4.75,1.1); \draw[thick,red] (4.75,1.1) cos (5,0.9);
\draw[thick,red] (5,0.9) sin (5.25,1.3); \draw[thick,red] (5.25,1.3) cos (5.5,1.1);
\draw[thick,red] (5.5,1.1) sin (5.75,1.9); \draw[thick,red,-latex] (5.75,1.9) cos (6,1.8);
\draw (5.25,1.9) node[above]{Sec. \ref{sec4.2}};

\draw[thick,red] (6,1.7) sin (7.25,1.9); \draw[thick,red,-latex] (7.25,1.9) cos (8.5,0.8);
\draw (7.25,2) node[above]{Sec. \ref{sec4.4}};

\draw (10,1) node[]{\color{red} $\boldsymbol{\cdots\cdots}$};

\foreach \j in {0,...,23} {\draw (0.5*\j,0)--(0.5*\j,-0.1); }
\draw (0,-0.1) node[below]{$x$};
\draw (11.5,-0.1) node[below]{$y$};

\draw[latex-latex,very thick] (0,-0.8)--(2,-0.8); \draw (1,-0.8) node[above]{$O_1$};
\draw[latex-latex,very thick] (2.5,-0.8)--(4,-0.8); \draw (3.25,-0.8) node[above]{$E_1$};
\draw[latex-latex,very thick] (4.5,-0.8)--(6,-0.8); \draw (5.25,-0.8) node[above]{$O_2$};
\draw[latex-latex,very thick] (6.5,-0.8)--(8,-0.8); \draw (7.25,-0.8) node[above]{$E_2$};
\draw (10,-0.3) node[below]{$\cdots\cdots$};

\foreach \i in {0,...,8} {\fill[black!30!white] (0,0.1+0.2*\i) circle (0.1); } \foreach \i in {9} {\fill[red] (0,0.1+0.2*\i) circle (0.1); } \foreach \i in {0,...,9} {\draw (0,0.1+0.2*\i) circle (0.1); }

\foreach \i in {0,...,2} {\fill[black!30!white] (0.5,0.1+0.2*\i) circle (0.1); } \foreach \i in {0,...,2} {\draw (0.5,0.1+0.2*\i) circle (0.1); }
\foreach \i in {0} {\fill[black!30!white] (1,0.1+0.2*\i) circle (0.1); } \foreach \i in {0} {\draw (1,0.1+0.2*\i) circle (0.1); }
\foreach \i in {0,...,3} {\fill[black!30!white] (1.5,0.1+0.2*\i) circle (0.1); } \foreach \i in {0,...,3} {\draw (1.5,0.1+0.2*\i) circle (0.1); }
\foreach \i in {0,...,1} {\fill[black!30!white] (2,0.1+0.2*\i) circle (0.1); }\foreach \i in {2} {\fill[red] (2,0.1+0.2*\i) circle (0.1); }  \foreach \i in {0,...,2} {\draw (2,0.1+0.2*\i) circle (0.1); }

\foreach \i in {0} {\fill[black!30!white] (4.5,0.1+0.2*\i) circle (0.1); } \foreach \i in {1} {\fill[red] (4.5,0.1+0.2*\i) circle (0.1); }  \foreach \i in {0,...,1} {\draw (4.5,0.1+0.2*\i) circle (0.1); }
\foreach \i in {0,...,3} {\fill[black!30!white] (5,0.1+0.2*\i) circle (0.1); } \foreach \i in {0,...,3} {\draw (5,0.1+0.2*\i) circle (0.1); }
\foreach \i in {0,...,4} {\fill[black!30!white] (5.5,0.1+0.2*\i) circle (0.1); } \foreach \i in {0,...,4} {\draw (5.5,0.1+0.2*\i) circle (0.1); }
\foreach \i in {0,...,7} {\fill[black!30!white] (6,0.1+0.2*\i) circle (0.1); } \foreach \i in {8} {\fill[red] (6,0.1+0.2*\i) circle (0.1); } \foreach \i in {0,...,8} {\draw (6,0.1+0.2*\i) circle (0.1); }

\foreach \i in {0,...,2} {\fill[black!30!white] (8.5,0.1+0.2*\i) circle (0.1); } \foreach \i in {3} {\fill[red] (8.5,0.1+0.2*\i) circle (0.1); } \foreach \i in {0,...,3} {\draw (8.5,0.1+0.2*\i) circle (0.1); }

\foreach \i in {0,...,5} {\fill[black!30!white] (11.5,0.1+0.2*\i) circle (0.1); } \foreach \i in {6} {\fill[red] (11.5,0.1+0.2*\i) circle (0.1); } \foreach \i in {0,...,6} {\draw (11.5,0.1+0.2*\i) circle (0.1); }
\draw (-0.5,0)--(12,0);
\end{tikzpicture}\caption{\label{fig4} Path from $\sigma_{m,\ell-m}^{xy}$ to $\sigma_{m-1,\ell-m+1}^{xy}$ in the general case explained in Section \ref{sec4.5}. The figure explains how the red particle moves from $x_0=x$ to $x_t=y$ along the path. First, in $O_1$ where the sites are occupied, the red particle simply jumps consecutively to the right as explained in Section \ref{sec4.2}. Next, in $E_1$ where the sites are empty and the initial stack is small, we follow the mechanism presented in Section \ref{sec4.3}.
In $O_2$ where the sites are again occupied, we proceed as in Section \ref{sec4.2}. In $E_2$ where the sites are empty and the initial stack is big, we follow the two-dimensional collection of paths explained in Section \ref{sec4.4}. We iterate these procedures until the red particle arrives at $y$.}
\end{figure}
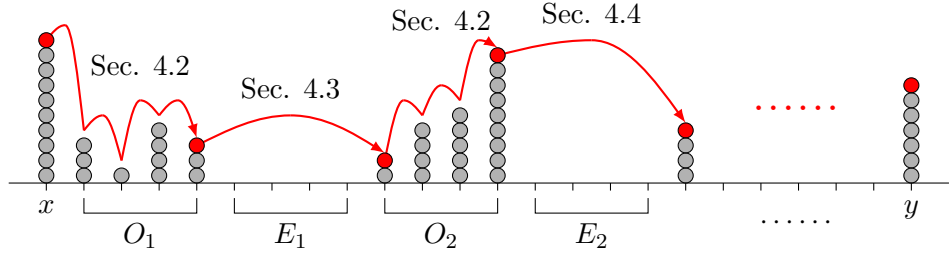

Again, denote by $\sk{\Omega_{x,y}^{\ell,m,\sigma}}$ the collection of all \sk{edges} that appear along the paths.

\begin{proof}[Proof of \eqref{eq:WTS-xy}]
\sk{
First, we bound $\sk{\nabla_K^2 f ( \sigma_{m,\ell-m}^{xy} ; x,y) }$ from above by intermediate quadratic forms, on which we can apply the special results in Sections \ref{sec4.1}--\ref{sec4.4}.
Recall the decomposition \eqref{eq:seg-dec}. The path $x \to y$ can be decomposed into (at most) $2r-1$ steps
\begin{equation}
x = x_0 \to x_{b_1-1} \to x_{a_2} \to x_{b_2-1} \to \cdots \to x_{a_r} \to x_t = y \comma
\end{equation}
where step $x_{a_j} \to x_{b_j-1}$ must be omitted if $b_j = a_j +1$. In this regard, we may decompose as
\begin{equation}\label{eq:Omega-dec}
\Omega_{x,y}^{\ell,m,\sigma} = \bigcup_{j=1}^r \Omega_{x_{a_j},x_{b_j-1}}^{\ell,m,\sigma} \cup \bigcup_{j=1}^{r-1} \Omega_{x_{b_j-1},x_{a_{j+1}}}^{\ell,m,\sigma}
\end{equation}
where, with a slight abuse of notation, each $\Omega_{x_{a_j},x_{b_j-1}}^{\ell,m,\sigma}$ (resp. $\Omega_{x_{b_j-1},x_{a_{j+1}}}^{\ell,m,\sigma}$) collects the edges that are obtained within the step $x_{a_j} \to x_{b_j-1}$ (resp. $x_{b_j-1} \to x_{a_{j+1}}$).
We will upper bound the term $\sk{\nabla_K^2 f ( \sigma_{m,\ell-m}^{xy} ; x,y) }$ with
\begin{equation}
\sk{\nabla_K^2 f ( \sigma_{m-1,1,\ell-m}^{xx_{a_1}y} ; x_{a_1},x_{b_1-1}) },
\sk{\nabla_K^2 f ( \sigma_{m-1,1,\ell-m}^{xx_{b_1-1}y} ; x_{b_1-1},x_{a_2}) }, \dots,
\sk{\nabla_K^2 f ( \sigma_{m-1,1,\ell-m}^{xx_{a_r}y} ; x_{a_r},x_{b_r-1}) }
\end{equation}
where $x_{a_1}=x_0=x$ and $x_{b_r-1}=x_t=y$. The Cauchy--Schwarz inequality applied to the difference of $f$ along these steps implies that
\begin{equation}
\begin{aligned}
& \tttonde{f(\sigma_{m-1,\ell-m+1}^{xy})-f(\sigma_{m,\ell-m}^{xy})}^{2} \le \\
& t \left( \sum_{j=1}^r \frac{\tttonde{f(\sigma_{m-1,1,\ell-m}^{xx_{b_j-1}y})-f(\sigma_{m-1,1,\ell-m}^{xx_{a_j}y})}^{2}}{b_j - a_j - 1}
+ \sum_{j=1}^{r-1} \frac{\tttonde{f(\sigma_{m-1,1,\ell-m}^{xx_{a_{j+1}}y})-f(\sigma_{m-1,1,\ell-m}^{xx_{b_j-1}y})}^{2}}{a_{j+1} - b_j +1} \right) \comma
\end{aligned}
\end{equation}
where it was used that $\sum_{j=1}^r (b_j-a_j-1) + \sum_{j=1}^{r-1} (a_{j+1}-b_j+1) = b_r - a_1 - 1 = t$. Here, the first summation in $j \in \bbr{1}{r}$ should be run over those $j$ such that $b_j > a_j+1$. Substituting \eqref{eq:cK-nablaK-def} to this inequality, we obtain that
\begin{equation}\label{eq:CS-inter}
\begin{aligned}
\nabla_K^2 f ( \sigma_{m,\ell-m}^{xy} ; x,y) \le t \, \Bigg(  \sum_{j=1}^r & \frac{ \fc_K (\sigma_{m,\ell-m}^{xy} ; x,y) \, \nabla_K^2 f ( \sigma_{m-1,1,\ell-m}^{xx_{a_j}y} ; x_{a_j},x_{b_j-1}) } {\fc_K(\sigma_{m-1,1,\ell-m}^{xx_{a_j}y};x_{a_j},x_{b_j-1}) (b_j-a_j-1)} \\
& + \sum_{j=1}^{r-1} \frac{ \fc_K (\sigma_{m,\ell-m}^{xy} ; x,y) \, \nabla_K^2 f ( \sigma_{m-1,1,\ell-m}^{xx_{b_j-1}y} ; x_{b_j-1},x_{a_{j+1}}) } {\fc_K(\sigma_{m-1,1,\ell-m}^{xx_{b_j-1}y};x_{b_j-1},x_{a_{j+1}}) (a_{j+1} - b_j + 1)} \Bigg) \, \fstop
\end{aligned}
\end{equation}
According to the results in Lemmas \ref{lem:4.1}, \ref{lem:4.2}, \ref{lem:4.3}, and \ref{lem:4.4}, for each $j \in \bbr1r$ we have
\begin{equation}
\begin{aligned}
& \frac{\fc_K (\sigma_{m,\ell-m}^{xy} ; x,y) \, \nabla_K^2 f ( \sigma_{m-1,1,\ell-m}^{xx_{a_j}y} ; x_{a_j},x_{b_j-1}) } {\fc_K(\sigma_{m-1,1,\ell-m}^{xx_{a_j}y};x_{a_j},x_{b_j-1}) (b_j-a_j-1)} \le \\
& \frac { \mu_{\alpha,k}(\sigma_{m,\ell-m}^{xy}) \, 9 \, m \tonde{\alpha_y + \ell-m} 6^{\alpha_{\rm max} }} { \mu_{\alpha,k} (\sigma_{m-1,1,\ell-m}^{xx_{a_j}y}) \, \alpha_{\rm min} \, \alpha_{\rm ratio} \, c_{\rm min}} \left( \sum_{\{\zeta,\zeta'\} \sk{\in} \Omega_{x_{a_j},x_{b_j-1}}^{\ell,m,\sigma}} \nabla^2f(\zeta,\zeta') \right) \comma
\end{aligned}
\end{equation}
and for each $j \in \bbr1{j-1}$ we have
\begin{equation}
\begin{aligned}
& \frac{\fc_K (\sigma_{m,\ell-m}^{xy} ; x,y) \, \nabla_K^2 f ( \sigma_{m-1,1,\ell-m}^{xx_{b_j-1}y} ; x_{b_j-1},x_{a_{j+1}}) } {\fc_K(\sigma_{m-1,1,\ell-m}^{xx_{b_j-1}y};x_{b_j-1},x_{a_{j+1}}) (a_{j+1}-b_j+1)} \le \\
& \frac { \mu_{\alpha,k}(\sigma_{m,\ell-m}^{xy}) \, 9 \, m \tonde{\alpha_y + \ell-m} 6^{\alpha_{\rm max} }} { \mu_{\alpha,k} (\sigma_{m-1,1,\ell-m}^{xx_{b_j-1}y}) \, \alpha_{\rm min} \, \alpha_{\rm ratio} \, c_{\rm min}} \left( \sum_{\{\zeta,\zeta'\} \sk{\in} \Omega_{x_{b_j-1},x_{a_{j+1}}}^{\ell,m,\sigma}} \nabla^2f(\zeta,\zeta') \right) \fstop
\end{aligned}
\end{equation}
Indeed, the term $t$ in the right-hand sides of the lemmas cancels out with $b_j-a_j-1$ or $a_{j+1}-b_j+1$, and the term $m \, (\alpha_y + \ell-m)$ cancels out with the $\fc_K$-term except the $\mu_{\alpha,k}$ part, thus we obtain the right-hand sides in the last two displays.

Next, we estimate the ratios of $\mu_{\alpha,k}$ that appear in the right-hand sides of the last two displays.
If $a_j \ne x,y$, then
\begin{equation}
\begin{aligned}
& \frac { \mu_{\alpha,k}(\sigma_{m,\ell-m}^{xy}) \, m \tonde{\alpha_y + \ell-m} } { \mu_{\alpha,k} (\sigma_{m-1,1,\ell-m}^{xx_{a_j}y}) } = \frac{\pi_x(m) \, m \, \pi_{x_{a_j}} (\sigma_{x_{a_j}}) \, \tonde{\alpha_y + \ell-m}}{\pi_x(m-1) \,\pi_{x_{a_j}}(\sigma_{x_{a_j}}+1) } \\
& = (\alpha_x + m-1) \tonde{\alpha_y + \ell-m} \, \frac{\sigma_{x_{a_j}}+1}{\sigma_{x_{a_j}}+\alpha_{x_{a_j}}} \le 2 \, (1+\alpha_{\rm max}) \, k \, (\alpha_{\rm max} + k-1) \fstop
\end{aligned}
\end{equation}
Here, we used $\sigma_{x_{a_j}} \ge 1$ (which implies $\frac{\sigma_{x_{a_j}}+1}{\sigma_{x_{a_j}}+\alpha_{x_{a_j}}} \le 2$), $\alpha_y + \ell-m \le \alpha_{\rm max} + k-1$, and $\alpha_x + m-1 \le (1+\alpha_{\rm max}) \, k$. The same bound trivially holds if $a_j = x$ or $a_j = y$, and similarly,
\begin{equation}
\frac { \mu_{\alpha,k}(\sigma_{m,\ell-m}^{xy}) \, m \tonde{\alpha_y + \ell-m} } { \mu_{\alpha,k} (\sigma_{m-1,1,\ell-m}^{xx_{b_j-1}y}) } \le 2 \, (1+\alpha_{\rm max}) \, k \, (\alpha_{\rm max} + k-1) \fstop
\end{equation}
Substituting the last four estimates to \eqref{eq:CS-inter}, we arrive at
\begin{equation}
\begin{aligned}
\nabla_K^2 f ( \sigma_{m,\ell-m}^{xy} ; x,y) & \le \frac { 18t \, (1+\alpha_{\rm max}) \, k \, (\alpha_{\rm max} + k-1) \, 6^{\alpha_{\rm max} }} { \alpha_{\rm min} \, \alpha_{\rm ratio} \, c_{\rm min}} \\
& \left(  \sum_{j=1}^r \sum_{\{\zeta,\zeta'\} \sk{\in} \Omega_{x_{a_j},x_{b_j-1}}^{\ell,m,\sigma}} \nabla^2f(\zeta,\zeta') + \sum_{j=1}^{r-1} \sum_{\{\zeta,\zeta'\} \sk{\in} \Omega_{x_{b_j-1},x_{a_{j+1}}}^{\ell,m,\sigma}} \nabla^2f(\zeta,\zeta') \right) \, \fstop
\end{aligned}
\end{equation}
Applying $t \le {\rm diam}(G)$, $1+\alpha_{\rm max} \le e^{\alpha_{\rm max}}$, and \eqref{eq:Omega-dec},
we get,
}for any $\ell\in\bbr{1}{k}$, $m\in\bbr{1}{\ell}$
and $\sigma\in\Xi_{k-\ell}^{V\setminus\{x,y\}}$,
\begin{equation}
\sk{\nabla_K^2 f ( \sigma_{m,\ell-m}^{xy} ; x,y) } \le  \frac{\sk{18 \, k \tonde{\alpha_{\rm max}+k-1}  {\rm diam}(G) \, (6e)^{\alpha_{\rm max} }}}{\alpha_{\rm min} \, \alpha_{\rm ratio} \, c_{\rm min}}    \sum_{\{\zeta,\zeta'\} \sk{\in} \sk{\Omega_{x,y}^{\ell,m,\sigma}}} \nabla^2f(\zeta,\zeta')   \fstop
\end{equation}
Along with \eqref{eq:Diri-dec-sigma}, this implies 
\begin{equation}\label{eq11}
\cE_K^{x,y} \le \frac{\sk{18 \, k \tonde{\alpha_{\rm max}+k-1}  {\rm diam}(G) \, (6e)^{\alpha_{\rm max} }}}{\alpha_{\rm min} \, \alpha_{\rm ratio} \, c_{\rm min}}  \sum_{\ell=1}^k \sum_{m=1}^\ell \sum_{\sigma \in \Xi_{k-\ell}^{V\setminus \{x,y\}}}    \sum_{\{\zeta,\zeta'\} \sk{\in} \sk{\Omega_{x,y}^{\ell,m,\sigma}}} \nabla^2f(\zeta,\zeta')   \fstop
\end{equation}
Thus, to conclude the proof of \eqref{eq:WTS-xy}, according to \eqref{eq:Diri-nabla-dec} and \eqref{eq11}, we only need to count the number of \textit{overlaps}, i.e., the number of triples $(\ell,m,\sigma)$ in \eqref{eq11} that produce the same Dirichlet summand $\nabla^2f(\eta,\eta-\delta_z+\delta_w)$ in \eqref{eq:Diri-nabla-dec}. 

To this end, we fix $x<y$, the sequence $x=x_0, x_1, \dots, x_t=y$ with $t\in \bbr{1}{{\rm diam}(G)}$, and a triple $(\eta,z,w)$ in \eqref{eq:Diri-nabla-dec}, and \sk{then} count the number of triples $(\ell,m,\sigma)$ that produce the term $\nabla^2f(\eta,\eta-\delta_z + \delta_w)$ in \eqref{eq11} by distinguishing three cases.\medskip

\noindent {\bf (Case 1)} There exists no $s \in \bbr{1}{t}$ for which $\{ x_{s-1},x_s \} = \{ z,w \}$:

Clearly, in this case 	 the overlap is \textit{zero}.\medskip

\noindent {\bf (Case 2)} There exists $s\in \bbr{1}{t}$ for which $( x_{s-1},x_s )= (z,w)$:

By the minimal property of the path length $t$, there is no other $s' \in \bbr{1}{t}$ such that $\{x_{s'-1},x_{s'}\}=\{z,w\}$, thus we may only focus on the particle jump $z \to w$. First, we observe that if $t=1$, then the overlap is at most \textit{one}. Thus,  we only focus on the case 	 $t \ge 2$. We further divide into three subcases according to the type --- {\bf [I]}, {\bf [II]} and {\bf [III]} --- of the particle jump $z \to w$.

\begin{itemize}
\item {\bf (Case 2.I)} $\eta \to \eta-\delta_z + \delta_w$ belongs to type {\bf [I]}: 

In this case, recalling that the overall mechanism is to move a single particle from $x$ to $y$ along the path, we obtain 
\begin{equation}
\eta= \sigma_{m,\ell-m}^{xy} + ( \delta_z - \delta_x ) \fstop
\end{equation}
This observation characterizes $(\ell,m,\sigma)$ uniquely from the configuration $\eta$. Hence, in this subcase, there is exactly \textit{one} overlap.\medskip

\item {\bf (Case 2.II)} $\eta \to \eta-\delta_z + \delta_w$ belongs to type {\bf [II]}: 

Recall that the mechanism in Section \ref{sec4.3} is further divided into types {\bf (F)}, {\bf (S)} and {\bf (B)}, where type {\bf (B)} is impossible here since the particle jumps forward. Moreover, this {\bf [II]}-mechanism is happening in $x_{s'}, \dots, x_{s-1},x_s,\dots, x_{s''}$ with $s' \le s-1 < s \le s''$ such that\footnote{If $x_{s'}=x$ (resp.\ $x_{s''}=y$), then condition $\sigma_{x_{s'}} \ge 1$ (resp.\ $\sigma_{x_{s''}} \ge 1$) should be removed.}
\begin{equation}
\sigma_{x_{s'}} \ge 1 \comma \quad \sigma_{x_{s'+1}} = \cdots = \sigma_{x_{s''-1}} = 0  \quad \text{and} \quad \sigma_{x_{s''}} \ge 1 \fstop
\end{equation}
First, suppose that $x_{s'}\ne x$. If the jump $\eta \to \eta-\delta_z + \delta_w$ is of type {\bf (F)}, then
\begin{align}
&\eta + ( \eta_w \delta_z - \eta_w \delta_w ) \\
&\qquad= \sigma_{m,\ell-m}^{xy} + (\delta_{x_{s'}}  - \delta_x) + ( (\sigma_{x_{s'}}+1) \delta_z - (\sigma_{x_{s'}}+1) \delta_{x_{s'}} )  \fstop
\end{align}
\sk{Note that the jump $\eta \to \eta-\delta_z + \delta_w$ happens within the $\eta_z + \eta_w$ jumps of the same number of particles from $z$ to $w$. Thus, the left-hand side configuration $\eta + (\eta_w \delta_z - \eta_w \delta_w)$ detects the moment right before when the particle jumps $z \to w$ are about to happen. Then, the right-hand side configuration is the corresponding configuration represented in terms of $(\ell,m,\sigma)$ due to the following logic. First, from $\sigma_{m,\ell-m}^{xy}$, a particle would have traveled from $x$ to $x_{s'}$, giving the contribution $\delta_{x_{s'}} - \delta_x$. Then, the whole $\sigma_{x_{s'}}+1$ particles at $x_{s'}$ would have traveled all together from $x_{s'}$ to $z$, which gives the contribution $(\sigma_{x_{s'}}+1)(\delta_z - \delta_{x_{s'}})$. This instance is exactly right before when the jumps $z \to w$ are initiated, giving the equality in the display. Thus, $\sigma_{m,\ell-m}^{xy}$ has a unique representation in terms of $x,y,\eta,z,w$, thus the triple $(\ell,m,\sigma)$ is determined uniquely. More precisely, $m$ is exactly the number of particles at $x$, $\ell$ is the number of particles at both $x$ and $y$, and $\sigma$ is the remaining configuration on $V \setminus \{x,y\}$.}
\sk{On the other hand,} if $\eta \to \eta-\delta_z +\delta_w$ is of type {\bf (S)}, then $s=s''$, and we have
\begin{equation}
\eta = \sigma_{m,\ell-m}^{xy} + (\delta_{x_{s'}}  - \delta_x) + ( (\sigma_{x_{s'}}+1) \delta_z - (\sigma_{x_{s'}}+1) \delta_{x_{s'}} ) \fstop
\end{equation}
\sk{Here, both sides detect the instance right before when the single jump $z \to w$ occurs, which is clearly $\eta$ in the left-hand side and also the configuration in the right-hand side by the same logic as explained above. Hence, the triple $(\ell,m,\sigma)$ is again determined uniquely in this case.} Thus, we have two possibilities of triples $(\ell,m,\sigma)$. 
Now, suppose that $x_{s'}=x$. Then, we may similarly detect two possibilities:
\begin{equation}
\eta + ( \eta_w \delta_z - \eta_w \delta_w ) = \sigma_{m,\ell-m}^{xy} + ( m \delta_z - m \delta_{x} ) 
\end{equation}
for type {\bf (F)}, and
\begin{equation}
\eta = \sigma_{m,\ell-m}^{xy} + ( m \delta_z - m \delta_{x} )
\end{equation}
for type {\bf (S)}. They both yield a unique triple $(\ell,m,\sigma)$. Therefore, we conclude that {\bf (Case 2.II)} gives rise to at most \textit{two}  overlaps.\medskip

\item {\bf (Case 2.III)} $\eta \to \eta-\delta_z + \delta_w$ belongs to type {\bf [III]}: 

As done in {\bf (Case 2.II)}, we fix $x_{s'} , \dots , x_{s-1} , x_s , \dots , x_{s''}$ on which this {\bf [III]}-mechanism takes place. A forward jump in type {\bf [III]} is of one of the following three types: initial distribution of $i \in \bbr{1}{\widetilde{m}}$ particles $x_{s'} \to x_{s'+1}$ (downward arrows in Figure \ref{fig3}), movement of $i$ particles in the bulk (\sk{double horizontal} arrows in the left \sk{half} of Figure \ref{fig3}), or the single particle jump $x_{s''-1} \to x_{s''}$ (\sk{single} horizontal arrows in the middle part of Figure \ref{fig3}). Each type determines a triple $(\ell,m,\sigma)$ uniquely, as thoroughly explained in {\bf (Case 2.II)}. Thus, we omit the details here and conclude that no more than \textit{three} overlaps are possible in {\bf (Case 2.III)}.\medskip

\end{itemize}
Collecting these three subcases, {\bf (Case 2)} admits at most \textit{six} overlaps.

\medskip

\noindent {\bf (Case 3)} There exists $s\in \bbr{1}{t}$ for which $( x_{s-1},x_s )= (w,z)$:

The analysis in this third case is almost identical to the one done in {\bf (Case 2)}. Thus, we choose not to repeat the tedious computations and record here that, also in this case, at most \textit{six} overlaps are possible in {\bf (Case 3)}.\medskip

Finally,  given $x<y$ and $(\eta,z,w)$, collecting {\bf (Case 1)}, {\bf (Case 2)} and {\bf (Case 3)} ensure that there are at most \textit{six} possible overlaps of triples $(\ell,m,\sigma)$. Therefore, by \eqref{eq:Diri-nabla-dec}, we deduce
\begin{equation}
\sum_{\ell=1}^k \sum_{m=1}^\ell \sum_{\sigma \in \Xi_{k-\ell}^{V \setminus \{x,y\} }} \sum_{\{\zeta,\zeta'\} \sk{\in} \sk{\Omega_{x,y}^{\ell,m,\sigma}}} \nabla^2f(\zeta,\zeta') \le 6 \, \cE \fstop
\end{equation}
Combining this with \eqref{eq11}, the proof of \eqref{eq:WTS-xy} is now completed.
\end{proof}

\begin{remark}\label{rem:new-eps-to-zero}
The full combination of the three types {\bf [I]}, {\bf [II]} and {\bf [III]} is needed to obtain the correct dependence on $\alpha_{\rm min}$ in Theorem \ref{th:key-ing}. Alternatively, we may  combine only two of those types (that is, {\bf [I]} $+$ {\bf [II]} or {\bf [I]} $+$ {\bf [III]}) to obtain different bounds. Here, we briefly record these new results for potential  applications in future works.

Recall that we collected \eqref{eq:2.1-1-1}, \eqref{eq:2.1-1-2}, and \eqref{eq:2.1-1-3} to obtain the exact bound in Lemma \ref{lem:4.3}. If we bound ${\alpha_x \sum_{j=1}^{m-1} \frac1j}$ from above by ${\alpha_{\rm max} \tonde{1+\log k}}$ rather than by $1$,   we obtain an alternative bound
\begin{equation}
\sk{\nabla_K^2 f ( \sigma_{m,\ell-m}^{xy} ; x,y) } \le \sk{3 \, e^{\alpha_{\rm max}(1+\log k)} } \, \frac {\sk{m\tonde{\alpha_y + \ell-m}  t }} { c_{\rm min}\, \alpha_{\rm min}\, \alpha_{\rm ratio} }  \left( \sum_{\{\zeta,\zeta'\} \sk{\in} \Omega} \nabla^2f(\zeta,\zeta') \right) \comma
\end{equation}
which does not require \eqref{eq:cond-few}. Combining this inequality with Lemma \ref{lem:4.2} and applying the same logic explained in that subsection, we get
\begin{equation}
\cE_{K} \le \frac{ C \, e^{\alpha_{\rm max}(1+\log k)} \, k \tonde{\alpha_{\rm max}+k-1} |V|^2 \,  {\rm diam}(G) \, 2^{\alpha_{\rm max} }}{\alpha_{\rm min} \, c_{\rm min} \, \alpha_{\rm ratio} }   \, \cE \comma
\end{equation}
\sk{where $C>0$ is a global constant.}
For simplicity, if we take $\alpha=\eps\hat \alpha$ as in \eqref{eq:alpha-eps-beta}, take $\eps\to 0$, and \sk{toss the dependence on $G$ and $\hat\alpha$ to the constants}, this reduces to	
\begin{equation}\label{cor:eps-to-zero-1}
\cE_K \le \frac{ C_{G,\hat\alpha} \, e^{C_{G,\hat\alpha} \,\eps \log k} \, k^2}{\eps} \, \cE \qquad \text{as} \ \eps \to 0 \comma
\end{equation}
\sk{where $C_{G,\hat\alpha}>0$ now depends on $G,\hat\alpha$ but not on $k,\eps$.}
Similarly, if we do not lower bound $\alpha_x \sum_{j=1}^{m-1} \frac1j$ in \eqref{eq6} by $1$ and repeat the same logic as above, we obtain a different bound
\begin{equation}
\cE_{K} \le \frac{ C \, k \tonde{\alpha_{\rm max}+k-1} |V|^2 \,  {\rm diam}(G) \, 6^{\alpha_{\rm max} }}{\alpha_{\rm min}^2 \, c_{\rm min} \, \alpha_{\rm ratio} \log k }   \, \cE \comma
\end{equation}
which then implies, in the asymptotic regime \eqref{eq:alpha-eps-beta},
\begin{equation}\label{cor:eps-to-zero-2}
\cE_K \le \frac{ C_{G,\hat\alpha} \, k^2}{\eps^2 \log k} \, \cE \qquad \text{as} \ \eps \to 0 \comma
\end{equation}
\sk{with different constants $C>0$ (global) and $C_{G,\hat\alpha}>0$ (depends only on $G,\hat\alpha$).}
\end{remark}

\section{Proofs of Theorems \ref{th:gap-asymptotics1} and \ref{th:failure-gap}}\label{sec5}
In this section, we prove  Theorems \ref{th:gap-asymptotics1} and \ref{th:failure-gap}, following the outline of Section \ref{sec:outline-2}. The basic starting idea is to split the infinitesimal generator $L_{G,\alpha,k}$ into two parts: a slow part and a fast part. In view of this  decomposition,  standard limit theorems for slow--fast systems ensure that, in the small-diffusivity limit, $\rm SIP$  is well approximated by  the slow dynamics, after a suitable projection/thermalization according to the fast one. Understanding how and when this projection affects ${\rm SIP}$'s spectral gap is the main non-trivial task of this section.

Recall from \eqref{eq:alpha-eps-beta} that we assume $\alpha = \eps \hat{\alpha}$, where $\hat{\alpha}=(\hat{\alpha}_x)_{x\in V}$ are fixed site weights  and $\eps$ tends to $0$. Finally, since the graph $G$ is fixed all throughout the section, we \fs{often} drop it from the notation, and simply write, e.g., ${\rm SIP}_k(\alpha)={\rm SIP}_k(G,\alpha)$ and $L_{\alpha,k}=L_{G,\alpha,k}$.

\subsection{Slow and fast dynamics}\label{sec5.1} Fix $k\ge 2$, and consider ${\rm SIP}_k(\eps\hat \alpha)$ with time sped up by a factor $\eps^{-1}$. Recalling \eqref{eq:gen-conservative}, the corresponding
  generator is then simply given by $\eps^{-1}L_{\alpha,k}= \eps^{-1}L_{\eps \hat \alpha,k}$.
By separating the $\eps$-dependent terms from the rest, we obtain
\begin{equation}\label{Ak-Bk-dec}
	\eps^{-1} L_{\eps \hat \alpha,k}= A_{\hat\alpha,k} +\eps^{-1} B_k\comma
\end{equation}
where, for all $f \in \R^{\Xi_k}$ and $\eta \in \Xi_k$,
\begin{align}\label{eq:Ak}
A_{\hat\alpha,k} f(\eta)&\eqdef \sum_{x,y\in V}c_{xy}\, \eta_x \hat\alpha_y \left ( {f(\eta-\delta_x+\delta_y)-f(\eta)} \right ) \comma\\
\label{eq:Bk}
B_k f(\eta)&\eqdef \sum_{x, y \in V}c_{xy}\, \eta_x \eta_y \left(f(\eta-\delta_x+\delta_y)-f(\eta)\right)\fstop
\end{align}
Let $(\cA_{\hat\alpha,k}(t))_{t\ge0}$ and $(\cB_k(t))_{t\ge0}$ denote the Markov chains in $\Xi_k$ generated by $A_{\hat\alpha,k}$ and $B_k$, respectively. More in detail, $(\cA_{\hat\alpha,k}(t))_{t\ge0}$ describes $k$ independent particles on $G$, while $(\cB_k(t))_{t\ge0}$ is a particle system with absorbing set
\begin{equation}\label{Omegak-def}
\Omega_k\eqdef \left\{\eta\in \Xi_k: c_{xy}\, \eta_x\eta_y=0 \ \text{for all}\  x, y \in V\right\} \comma
\end{equation}
and transient set $\Delta_k\eqdef \Xi_k\setminus \Omega_k$, i.e.,
\begin{equation}\label{Deltak-def}
\Delta_k\eqdef  \left\{\eta\in \Xi_k: c_{xy}\, \eta_x\eta_y>0\ \text{for some}\ x, y \in V\right\}\fstop
\end{equation}
Let $\Pi_k: \R^{\Xi_k}\to \R^{\Xi_k}$ denote the $B_k$-harmonic projection operator, i.e., 
\begin{equation}
\Pi_k f\eqdef \lim_{t\to \infty} e^{t B_k}f \comma \qquad f \in \R^{\Xi_k}\fstop
\end{equation}
Remark that the range of $\Pi_k$, written as $\cR(\Pi_k)\subseteq \R^{\Xi_k}$, is a $|\Omega_k|$-dimensional subspace. \fs{More precisely, since $\Omega_k$ is the absorbing set of $B_k$, one has, for all $f\in \R^{\Xi_k}$, $\Pi_k f=f$ on $\Omega_k$, and the values of $\Pi_k f$ on $\Delta_k$ are fully determined by the absorption probabilities on $\Omega_k$. Thus, restriction to $\Omega_k$ gives a canonical linear isomorphism $\R^{\Xi_k}\supseteq\cR(\Pi_k)\cong \R^{\Omega_k}$, and $\cR(\Pi_k)$ consists of $B_k$-harmonic extensions to $\R^{\Xi_k}$ of functions in $\R^{\Omega_k}$.}

The following result builds on a powerful theorem \cite[Theorem 2.1]{kurtz_limit_theorem_1973} (see also \cite[\S1, Theorem 7.6]{ethier_kurtz_1986_Markov}) on limit theorems for Markovian slow--fast systems.

\begin{proposition}\label{lem:kurtz}
The operator \fs{${\mathscr G}_{\hat\alpha,k}\eqdef\Pi_k A_{\hat\alpha,k}|_{\cR(\Pi_k)}:\cR(\Pi_k)\to\cR(\Pi_k)$} is an infinitesimal generator  on $\cR(\Pi_k) \subseteq \R^{\Xi_k}$. Moreover,  we have the following  convergence:
\begin{equation}\label{eq:conv-kurtz}
\lim_{\eps\to 0} \max_{\eta \in \Xi_k}\left|e^{t\tttonde{A_{\hat \alpha,k}+\eps^{-1}B_k}}  f(\eta) - e^{t\fs{{\mathscr G}_{\hat\alpha,k}}} \Pi_k f(\eta) \right| = 0 \comma\qquad t>0\comma f \in \R^{\Xi_k} \fstop
\end{equation} 
\end{proposition}
\begin{proof}In this proof, we write $P_t^\eps = e^{t\tttonde{A_{\hat \alpha,k}+\eps^{-1}B_k}}$. The first claim and a convergence as in \eqref{eq:conv-kurtz} with $P_t^\eps f(\eta)$ replaced by $P_t^\eps \Pi_k f(\eta)$ are the content of \cite[\S1, Theorem 7.6]{ethier_kurtz_1986_Markov}. In order to obtain the desired claim (i.e., remove the projection operator $\Pi_k$), it suffices to prove that
\begin{equation}
\lim_{\eps\to 0} \max_{\eta\in \Xi_k} \abs{P_t^\eps \Pi_k f(\eta)-P_t^\eps f(\eta)}=0\comma\qquad t>0\fstop
\end{equation}
For this purpose, fix $f\in \R^{\Xi_k}$, and observe that
\begin{equation}
\Pi_k f(\eta)= f(\eta)\comma \qquad  \eta \in \Omega_k\comma
\end{equation}
because $(\cB_k(t))_{t\ge0}$ has $\Omega_k$ as absorbing states. Hence, since $\Pi_k$ is a contraction, we get
\begin{equation}
\abs{P_t^\eps \Pi_k f(\eta)-P_t^\eps f(\eta)} \le
\sum_{\xi \in \Delta_k } p_t^\eps(\eta,\xi)\abs{\Pi_k f(\xi)-f(\xi)}
\le \ttonde{2\max_{\xi\in \Xi_k}\abs{f(\xi)} }\, p_t^\eps(\eta,\Delta_k ) \comma
\end{equation}
where $p_t^\eps(\emparg,\emparg)$ denotes the transition kernel of the $\eps^{-1}$-sped up ${\rm SIP}_k(\eps\hat \alpha)$ at time $t$.
 Since
\begin{equation}
\lim_{\eps \to 0} p_t^\eps (\eta, \Delta_k)= 0\comma\qquad \eta \in \Xi_k\comma t>0 \comma
\end{equation}
 (see Proposition \ref{lem:fin-dim-conv} below for the full statement and proof of this fact), the proof of the proposition is concluded.
\end{proof}

Fix $t> 0$, and recall that $e^{t\tttonde{A_{\hat \alpha,k}+\eps^{-1}B_k}}$ has all real eigenvalues, just one equal to $1$, while all the others lying in the interval $(0,1)$.  Further,  Proposition \ref{lem:kurtz} ensures that all  eigenvalues (with multiplicity) of $e^{t\tttonde{A_{\hat \alpha,k}+\eps^{-1}B_k}}$ converge to those of $e^{t\fs{{\mathscr G}_{\hat\alpha,k}}}\fs{\Pi_k}$. \fs{Here, ${\mathscr G}_{\hat\alpha,k}$ acts on $\cR(\Pi_k)$, while $e^{t{\mathscr G}_{\hat\alpha,k}}\Pi_k$ is regarded as an operator on $\R^{\Xi_k}$. On $\cR(\Pi_k)$, this operator coincides with $e^{t{\mathscr G}_{\hat\alpha,k}}$, whereas on ${\rm Ker}(\Pi_k)$ it is identically zero.} In particular, the spectrum of $e^{t\fs{{\mathscr G}_{\hat\alpha,k}}}\fs{\Pi_k}$ entirely lies in $[0,1]$: \fs{one eigenvalue is $1$,} $|\Omega_k|-1$  eigenvalues \fs{lie in $(0,1)$}, while the  eigenvalue $0$ has multiplicity \fs{${\rm dim}({\rm Ker}(\Pi_k))= |\Xi_k|-|\Omega_k|=|\Delta_k|$}.

As a consequence, the second-largest eigenvalue  of $e^{t\tttonde{A_{\hat \alpha,k}+\eps^{-1}B_k}}$ converges to the second-largest eigenvalue of $ e^{t \fs{{\mathscr G}_{\hat\alpha,k}}}\fs{\Pi_k}$ \fs{on $\cR(\Pi_k)\subset \R^{\Xi_k}$}, which is given by $e^{-t w_k(\hat\alpha)}$, where
\begin{equation}\label{wk-def}
w_k(\hat\alpha)\eqdef 	\	 \text{the second-smallest eigenvalue of}\ \fs{-{\mathscr G}_{\hat\alpha,k}}:\cR(\Pi_k)\to \cR(\Pi_k) \fstop
\end{equation}
Thus, we proved the following claim. (Recall that ${\rm gap}_k(\eps \hat \alpha)={\rm gap}_k(G,\eps\hat \alpha)$.)
\begin{corollary}\label{lem:5.2}
For all integers $k\ge 2$, we have
\begin{equation}
\lim_{\eps \to 0} \frac{{\rm gap}_k( \eps \hat\alpha)}{\eps}= w_k(\hat\alpha)\fstop
\end{equation} 
\end{corollary}

In the next subsection, we analyze more in detail the limiting generator \fs{${\mathscr G}_{\hat\alpha,k}=\Pi_k A_{\hat \alpha,k}|_{\cR(\Pi_k)}$} 	 and its block-triangular structure.

\subsection{Recurrent and transient states}\label{sec5.2}
As already proved in Proposition \ref{lem:kurtz},  
\fs{${\mathscr G}_{\hat\alpha,k}=\Pi_k A_{\hat\alpha,k}|_{\cR(\Pi_k)}: \cR(\Pi_k)\to \cR(\Pi_k)$ may be interpreted as} the infinitesimal generator of a Markov chain  on $\Omega_k$, hereafter referred to as $(\cM_{\hat\alpha,k}(t))_{t\ge0}$. 
\fs{For notational simplicity, instead of analyzing directly the operator
 ${\mathscr G}_{\hat\alpha,k}$ on the $|\Omega_k|$-dimensional subspace $\cR(\Pi_k)$ of $\R^{\Xi_k}$, 
	we identify $\mathcal R(\Pi_k)$ with $\mathbb R^{\Omega_k}$ by restriction to $\Omega_k$,
	and analyze the transformed rate matrix --- defined as $M_{\hat\alpha,k}$ in  \eqref{eq:M} or \eqref{eq:iota-ahat-M-proof} below --- on $\mathbb R^{\Omega_k}$. Let us now 	describe this rate matrix in detail.
	
	Recall the chains $(\cA_{\hat\alpha,k}(t))_{t\ge 0}$ and $(\cB_k(t))_{t\ge 0}$ introduced at the beginning of Section \ref{sec5.1}.
	Let $\mathbf r^\cA_{\hat\alpha,k}$ be the transition-rate function of the process
	$(\cA_{\hat\alpha,k}(t))_{t\ge0}$. Moreover, let
	$\mathbf P^{\cB}_\zeta$ denote the law of the process
	$(\cB_k(t))_{t\ge0}$ started from $\zeta\in\Xi_k$, and define
	\begin{equation}
		\tau_\xi
		\eqdef
		\inf\{t\ge0:\mathcal B_k(t)=\xi\}
		\comma\qquad
		\tau_{\Omega_k}
		\eqdef
		\min_{\xi\in\Omega_k}\tau_\xi
		\fstop
	\end{equation}
	The effective process on $\Omega_k$, denoted by
	$(\cM_{\hat\alpha,k}(t))_{t\ge0}$, has non-diagonal transition rates
	\begin{equation}\label{eq:rM-def}
		\mathbf r^\cM_{\hat\alpha,k}(\eta,\xi)
		\eqdef
		\sum_{\zeta\in\Xi_k}
		\mathbf r^\cA_{\hat\alpha,k}(\eta,\zeta)\,
		\mathbf P^{\mathcal B}_\zeta
		\left[
		\tau_\xi=\tau_{\Omega_k}
		\right]
		\comma\qquad
		\eta,\xi\in\Omega_k\comma \eta\neq\xi
		\fstop
	\end{equation}
	The diagonal entries are chosen so that row sums vanish. We write
	\begin{equation}\label{eq:M}
		M_{\hat\alpha,k}
		:=
		\tonde{
		\mathbf r^\cM_{\hat\alpha,k}(\eta,\xi)
	}_{\eta,\xi\in\Omega_k}\in \R^{\Omega_x\times \Omega_k}
	\end{equation}
	for the corresponding square rate matrix of size $|\Omega_k|$.}

From the form of these rates, we immediately derive a classification of recurrent and transient states for this Markov chain; this will lead to a further characterization of the spectral gap $w_k(\hat \alpha)$ of \fs{${\mathscr G}_{\hat \alpha,k}$}. We  collect this classification in {Proposition \ref{pr:meta} below}; its direct consequence on the spectral gap is the content of the subsequent corollary. Before that, let us introduce, for  every integer $m\ge 1$, 
\begin{equation}\label{Omegakm-def}
	\Omega_{k,m} \eqdef  \left\{ \eta \in \Omega_k :| \{ x \in V : \eta_x > 0 \} | = m \right\} \comma
\end{equation}
that is,  the set of configurations in $\Omega_k$ with $m$ separated stacks of particles (cf.\ \eqref{Omegak-def}). 
Clearly, $\Omega_{k,m}=\emp$ implies $\Omega_{k,m+\ell}=\emp$ for all $\ell \ge 1$. Further, we have $\Omega_{k,1}\neq \emp$,  $\Omega_{k,k+1}=\emp$, and, thus, $	\Omega_k =  \sqcup_{m=1}^\infty \Omega_{k,m} = \Omega_{k,1} \sqcup \Omega_{k,2} \sqcup \cdots \sqcup \Omega_{k,k}$.	

\fs{To make the limiting dynamics more transparent, we first discuss a simple example.}

\fs{\begin{example}[Three-site path: block structure]\label{ex:effective-chain-path3}
		Let $G$ be the path $\{1,2,3\}$, with conductances $c_{12}, c_{23}>0$ and $c_{13}=0$, and consider the
		limiting chain on $\Omega_k$ generated by $M_{\hat\alpha,k}$. This example illustrates
		the block structure described above.
		
		For $k=2$, the absorbing set of the fast dynamics is
		\begin{equation}
		\Omega_2
		=
		\Omega_{2,1}\sqcup\Omega_{2,2}\comma
		\qquad
		\Omega_{2,1}=\{2\delta_1,2\delta_2,2\delta_3\}\comma
		\qquad
		\Omega_{2,2}=\{\delta_1+\delta_3\}\fstop
		\end{equation}
		The block $\Omega_{2,1}$ is recurrent: it describes a single stack of two particles moving
		effectively as one random walk on $G$. The block $\Omega_{2,2}$ is transient: from
		$\delta_1+\delta_3$, a slow jump of either particle to the middle site creates a configuration
		in $\Delta_2$, after which the fast dynamics immediately thermalizes into the recurrent
		block $\Omega_{2,1}$,  never returning to $\Omega_{2,2}$. 
		
		For $k=3$, one similarly has
		\begin{equation}
		\Omega_3
		=
		\Omega_{3,1}\sqcup\Omega_{3,2}\comma
		\qquad
		\Omega_{3,1}=\{3\delta_1,3\delta_2,3\delta_3\}\comma
		\qquad
		\Omega_{3,2}
		=
		\{2\delta_1+\delta_3,\ \delta_1+2\delta_3\}\comma
		\end{equation}
		while $\Omega_{3,m}=\emp$ for all $m\ge 3$.
		Again, $\Omega_{3,1}$ is recurrent, while $\Omega_{3,2}$ is transient.  In contrast with the case $k=2$, however, a slow jump from a state in $\Omega_{3,2}$ to $\Delta_3$ need not be followed by a fast relaxation directly into $\Omega_{3,1}$: the fast dynamics may also reach the other state of $\Omega_{3,2}$. Indeed, starting, e.g., from $2\delta_1+\delta_3\in\Omega_{3,2}$, a slow jump of one of the two particles sitting at site $1$ to site $2$ yields $\delta_1+\delta_2+\delta_3\in\Delta_3$. The subsequent fast thermalization may either merge all particles into a single stack, hence reaching $\Omega_{3,1}$, or merge the particle at site $2$ with the one at site $3$, thereby reaching $\delta_1+2\delta_3\in\Omega_{3,2}$. In conclusion, for $k=3$, the transient block is already a genuine two-state sub-Markov chain.
\end{example}}

\begin{proposition}\label{pr:meta}
	
For the chain $(\cM_{\hat\alpha,k}(t))_{t\ge0}$, the following two claims hold true:
	\begin{enumerate}[(a)]

		\item \label{it:meta0}For every $m, m'\ge 1$ with $\Omega_{k,m}\neq \emp\neq \Omega_{k,m'}$ and $\eta\in \Omega_{k,m}$, we have 
	\begin{equation}
		{\bf r}_{\hat \alpha,k}^\cM(\eta,\xi) = 0\comma\quad \text{for all}\ \xi \in \Omega_{k,m'}, \ m'>m\comma 
	\end{equation}
	while, when $m\ge 2$, there exist $\eta^1,\ldots,\eta^\ell \in \Omega_{k,m}$ satisfying
	\begin{equation}
{\bf r}_{\hat \alpha,k}^\cM(\eta,\eta^1)\tonde{	\prod_{h=1}^{\ell-1}	{\bf r}_{\hat \alpha,k}^\cM(\eta^h,\eta^{h+1})} {\bf r}_{\hat \alpha,k}^\cM(\eta^\ell,\xi)>0\comma\quad \text{for some}\  \xi\in \Omega_{k,m'}, \ m'<m\fstop
	\end{equation}
	As a consequence, $\Omega_{k,m}$, $m\ge 2$, consists of transient states only.
	
		\item \label{it:meta} The chain $(\cM_{\hat \alpha,k}(t))_{t\ge0}$ is irreducible on $\Omega_{k,1}\neq \emp$, and the position of the single stack evolves as ${\rm RW}(\hat \alpha)$.
		\end{enumerate}
\end{proposition}

\begin{proof} \ref{it:meta0} The desired claim follows from the connectedness of the graph $G$, and the observation that the rates ${\bf r}_{\hat\alpha,k}^\cM(\emparg,\emparg)$ in \eqref{eq:rM-def} describe  the following two subsequent mechanisms:  a step of the chain $(\cA_{\hat \alpha,k}(t))_{t\ge 0}$ increases the number of stacks at most by one, creating, in this case, at least two  neighboring stacks;  the chain $(\cB_k(t))_{t\ge 0}$  forces all  neighboring particles to eventually merge into isolated (i.e., non-neighboring) stacks.

	\smallskip
\noindent	\ref{it:meta} Observe that, for all  $x,y,z \in V$ with $c_{xz}>0$, we have
\begin{equation}
	{\bf P}^{\cB}_{(k-1)\delta_x+\delta_z}\,[\tau_{k\delta_y}=\tau_{\Omega_{k,1}}] = \frac{\car_{y=z}}k + \frac{(k-1) \, \car_{y=x}}k\comma
\end{equation} 
because the trajectories of $(\cB_k(t))_{t\ge 0}$ started from $(k-1)\delta_x+\delta_z\in \Delta_k$ coincide with those of the symmetric random walk on  $\bbr{0}{k}$ started from $k-1$. Plugging this expression into \eqref{eq:rM-def}, we obtain, for $x \ne y$,
\begin{multline}
	{\bf r}_{\hat \alpha,k}^\cM(k\delta_x,k\delta_y)\\ = \sum_{z\in V} {\bf r}_{\hat \alpha,k}^\cA(k\delta_x,(k-1)\delta_x+\delta_z)\, {\bf P}_{(k-1)\delta_x+\delta_z}^\cB \, [\tau_{k\delta_y}=\tau_{\Omega_{k,1}}] = k\,c_{xy}\,\hat \alpha_y\,\frac1k = c_{xy}\,\hat \alpha_y\comma
\end{multline}
namely, the transition rates of ${\rm RW}(\hat \alpha)$. Irreducibility is then an obvious consequence of  that  of ${\rm RW}(\hat \alpha)$.
\end{proof}

\begin{remark}
	The claim in Proposition \ref{pr:meta}\ref{it:meta} is the finite-particle analogue of the metastable picture of ${\rm SIP}$, fully detailed in \cite{grosskinsky_redig_vafayi_dynamics_2013,bianchi_metastability_2017}.
\end{remark}

\begin{remark}\label{rem:complex}
		For $k\ge m\ge 2$, the dynamics on $\Omega_{k,m}\neq \emp$ encoded by $M_{\hat\alpha,k}$   goes as follows. All stacks evolve as independent ${\rm RW}(\hat \alpha)$ as long as their mutual distances are at least three. When two or more stacks reach  mutual distance two, then there is a positive rate for them to merge. These rates --- as well as the new positions of the stacks and their new compositions ---  depend on $G$ and $\hat \alpha$, but also non-trivially on $k\ge m$ and, when $k>m$, on the stacks' composition. For more details, see Section \ref{sec:irreducible-dec} below.
\end{remark}

The state-classification in Proposition \ref{pr:meta} may be equivalently restated as follows: after a suitable conjugation with a permutation matrix, $M_{\hat \alpha,k}$ is turned into a block lower triangular matrix with $M_{\hat \alpha,k,1},\ldots, M_{\hat \alpha,k,k}$ as diagonal blocks, where,
 for all $m\in \bbr{1}{k}$,	 
\begin{equation}\label{eq:Mkm}
	M_{\hat \alpha,k,m} \in \R^{|\Omega_{k,m}|\times|\Omega_{k,m}|}
\end{equation} is the submatrix of $M_{\hat \alpha,k}$ obtained from the restriction on configurations in $\Omega_{k,m}$ (see Figure \ref{fig6}). Since the eigenvalues of block triangular matrices  are the union of those of the diagonal blocks, it suffices to focus on the spectrum of the blocks $M_{\hat \alpha,k,m}$, $m\in \bbr{1}{k}$.
As already showed in the previous section, the eigenvalues of $-M_{\hat \alpha,k,m}$, $m\in \bbr{1}{k}$, are all real and non-negative.  Remark that:
\begin{itemize}
	\item for $m=1$, the spectrum of $M_{\hat \alpha,k,m}$ coincides with that of the generator of ${\rm RW}(\hat \alpha)$;
	\item  for $m\ge 2$ and $\Omega_{k,m}\neq \emp$, $M_{\hat \alpha,k,m}$ is the rate matrix of a sub-Markovian chain killed upon exiting $\Omega_{k,m}$; thus, all its eigenvalues are strictly negative.
\end{itemize}   
In view of these simple observations, we  register the following useful characterization of the spectral gap $w_k(\hat \alpha)$ defined in \eqref{wk-def}.

\begin{corollary}\label{cor:wk} For all $m\ge 2$ such that $\Omega_{k,m}\neq\emp$, define
	\begin{equation}\label{eq:lambda-k-m}
		\lambda_{k,m}(\hat \alpha)\eqdef\ \text{smallest eigenvalue of the matrix}\  -M_{\hat \alpha,k,m}\in \R^{|\Omega_{k,m}|\times|\Omega_{k,m}|} \comma
	\end{equation} whereas  $\lambda_{k,m}(\hat \alpha)\eqdef +\infty$ if $\Omega_{k,m}=\emp$. Then, we have
\begin{equation}\label{eq:gap-km-dec}
	w_k(\hat\alpha) = {\rm gap}_{\rm RW}(\hat\alpha) \wedge \min_{m \in \bbr{2}{k}} \lambda_{k,m}(\hat\alpha) \fstop
\end{equation}
\end{corollary}
\fs{\begin{example}[Three-site path: eigenvalues]\label{ex:effective-chain-path3-gaps}
		We keep the notation of Example \ref{ex:effective-chain-path3}.
		
		For $k=2$, the only transient state is $\delta_1+\delta_3$. It exits the transient
		block as soon as one of the two particles jumps to site $2$, which happens at rate
		\begin{equation}
			p\eqdef \tonde{c_{12}+c_{23}}\hat\alpha_2\fstop
		\end{equation}
		Therefore,
		$
			\lambda_{2,2}(\hat\alpha)=p$.
		
		For $k=3$, the transient block is
		$
			\Omega_{3,2}
			=
			\{2\delta_1+\delta_3,\delta_1+2\delta_3\}$.
		Starting from $2\delta_1+\delta_3$, a slow jump of one of the two particles at site $1$
		to site $2$ occurs at rate $2c_{12}\hat\alpha_2$ and produces
		$\delta_1+\delta_2+\delta_3$. From this configuration, the fast dynamics reaches
		$\delta_1+2\delta_3$ with probability $c_{23}/(2(c_{12}+c_{23}))$. Hence, the effective jump rate from
		$2\delta_1+\delta_3$ to $\delta_1+2\delta_3$ is
		\begin{equation}
			q\eqdef
			2c_{12}\hat\alpha_2\frac{c_{23}}{2(c_{12}+c_{23})}
			=
			\frac{c_{12}c_{23}}{c_{12}+c_{23}}\hat\alpha_2\fstop
		\end{equation}
		By symmetry, the effective jump rate in the opposite direction is also $q$.	
		The effective killing rate from either transient state to the recurrent block is $p$.
		Indeed, starting again from $2\delta_1+\delta_3$, the slow jump of the particle at site
		$3$ to site $2$ contributes $c_{23}\hat\alpha_2$ to this killing rate. The slow jump of one of
		the two particles at site $1$ to site $2$ contributes instead
		$2c_{12}\hat\alpha_2\cdot\frac12=c_{12}\hat\alpha_2$, since from
		$\delta_1+\delta_2+\delta_3$ the fast dynamics reaches $\Omega_{3,1}$ with probability
		$1/2$. Thus, the total killing rate is
		$
			c_{23}\hat\alpha_2+c_{12}\hat\alpha_2=p$.
		Therefore, in the ordered basis
		$\tonde{2\delta_1+\delta_3,\delta_1+2\delta_3}$,
		\begin{equation}
			-M_{\hat\alpha,3,2}
			=
			\begin{pmatrix}
				p+q & -q \\
				-q & p+q
			\end{pmatrix}\fstop
		\end{equation}
		Its smallest eigenvalue is
		$\lambda_{3,2}(\hat\alpha)=p$, 
		the other eigenvalue being $p+2q>p$. Thus, in this example, the same quantity $p$
		gives the transient contribution to the limiting gap for both $k=2$ and $k=3$.
\end{example}}

The natural next step is to estimate the eigenvalue $\lambda_{k,m} (\hat\alpha)$ for each $k\ge 2$ and $m \in \bbr{2}{k}$, by bounding it from below by $\lambda_{2,2}(\hat \alpha)$. This task is carried out in the following two subsections. In the first one, we prove $\lambda_{k,m}(\hat \alpha)= \lambda_{m,m}(\hat \alpha)$; in the second one, we show $\lambda_{m,m}(\hat \alpha)\ge \lambda_{2,2}(\hat \alpha)$.

We conclude this subsection by observing that each submatrix $M_{\hat\alpha,k,m}$ in $\Omega_{k,m}$ is self-adjoint with respect to a  positive measure $\varsigma_{\hat\alpha,k,m}$ given in \eqref{eq:varsigma} below. On the one hand, this property is interesting and may become handy in future use, e.g., for building an $L^2$-theory for the corresponding continuum model. On the other hand, in our storyline, neither this property nor the precise form of the symmetrizing measure are  of primary importance. Thus, we only state the full statement  here,	 and postpone its proof to Appendix \ref{appenB}.
\begin{lemma}\label{lem:nu-self-adj}
For all $m\le k$, $M_{\hat\alpha,k,m}$ is self-adjoint as an operator on $L^2(\varsigma_{\hat\alpha,k,m})$, where 
\begin{equation}\label{eq:varsigma}
\varsigma_{\hat \alpha,k,m}(\eta)\eqdef \prod_{\substack{x\in V\\ \eta_x>0}}\frac{\hat \alpha_x}{\eta_x}\comma\qquad \eta \in \Omega_{k,m}\fstop
\end{equation}
\end{lemma}

\subsection{Comparing $\lambda_{k,m}(\hat \alpha)$ with $\lambda_{m,m}(\hat \alpha)$}\label{sec5.3}

The main goal of this subsection is to prove the following result.
\fs{\begin{proposition}\label{pr:lambda-km-mm}
For all  integers $m\ge 2$ with $\Omega_{m,m}\neq \emp$, we have
	\begin{equation}\label{eq:lambda-km-mm}
		\lambda_{k,m}(\hat\alpha)=\lambda_{m,m}(\hat\alpha)\comma\qquad k\ge m\fstop
	\end{equation}
	\end{proposition}}
We divide the proof of this identity into a few steps, and fix $m\ge 2$ so as to satisfy $\Omega_{m,m}\neq \emp$ all throughout.

\subsubsection{Consistency}\label{sec5.3.1} Recall the consistency property \eqref{eq:consistency-SIP} satisfied by  ${\rm SIP}$, as well as the corresponding annihilation operator $\mathfrak{a}_k : \R^{\Xi_{k-1}} \to \R^{\Xi_k}$ in \eqref{ann-op-def}. Since $M_{\hat\alpha,k}$ arises as a small-diffusivity limit of ${\rm SIP}$, it is natural to expect that an analogous property should hold also for $M_{\hat\alpha,k}$. This is the content of the following proposition. 
\begin{proposition}\label{pr:cons-M}
	Define, for all integers $k\ge 2$, 	the following \textquotedblleft restricted\textquotedblright\ annihilation operator $\hat{\mathfrak a}_k : \R^{\Omega_{k-1}}\to \R^{\Omega_k}$ as
	\begin{equation}\label{eq:ann-hat}
		\hat{\mathfrak a}_k g(\eta)\eqdef \sum_{x\in V} \eta_x\,g(\eta-\delta_x)\comma\qquad g\in \R^{\Omega_{k-1}}\comma\eta\in \Omega_k\fstop 
	\end{equation}
	Then, we have \begin{equation}
		\label{eq:Omega-intertwine}M_{\hat\alpha,k}\,\hat{\mathfrak a}_k = \hat{\mathfrak a}_k\, M_{\hat\alpha,k-1}\fstop
		\end{equation}
\end{proposition}
\begin{proof}	It is well known that a system of $k$ independent random walks is consistent, i.e.,  $A_{\hat\alpha,k} \, \mathfrak{a}_k = \mathfrak{a}_k \, A_{\hat\alpha,k-1}$. Moreover, by rewriting the consistency property of ${\rm SIP}(\eps\hat \alpha)$ in \eqref{eq:consistency-SIP} with the notation in \eqref{Ak-Bk-dec}, we have $(A_{\hat\alpha,k} + \eps^{-1} B_k) \, \mathfrak{a}_k  = \mathfrak{a}_k \, (A_{\hat\alpha,k-1} + \eps^{-1} B_{k-1})$. Thus, by linearity, we must also have $B_k\,\mathfrak{a}_k=\mathfrak a_k\,B_{k-1}$, and a similar relation for the harmonic projection $\Pi_k$ of $B_k$. \fs{By combining these observations, we obtain
		\begin{equation}\label{eq:Pi-a}
			\Pi_k\,\mathfrak a_k = \mathfrak a_k\, \Pi_{k-1}\quad\text{as maps from}\ \R^{\Xi_{k-1}}\ \text{to}\ \R^{\Xi_k}\comma
		\end{equation}
		and
		\begin{equation}\label{eq:PiA-cons-M-proof}
			\Pi_k A_{\hat\alpha,k}\,\mathfrak a_k = \mathfrak a_k\,\Pi_{k-1}A_{\hat\alpha,k-1}\quad\text{as maps from}\ \R^{\Xi_{k-1}}\ \text{to}\ \R^{\Xi_k}\fstop
		\end{equation}
		In particular, \eqref{eq:Pi-a} ensures that $\mathfrak a_k$ maps $\cR(\Pi_{k-1})$ into $\cR(\Pi_k)$.

	For all $j\ge 1$, let
		$
			\rho_j:\cR(\Pi_j)\to\R^{\Omega_j}
		$
		be the restriction map to $\Omega_j$, and let
		$
			\iota_j:\R^{\Omega_j}\to\cR(\Pi_j)
	$
		be its inverse given by $B_j$-harmonic extension. By definition of the
		rate matrix in \eqref{eq:M},
		\begin{equation}\label{eq:M-conj-M-proof}
			M_{\hat\alpha,j}
			=
			\rho_j\Pi_j A_{\hat\alpha,j}\iota_j
			\quad\text{on }\R^{\Omega_j}
			\fstop
		\end{equation}
		Moreover, since $\mathfrak a_k$ maps $\cR(\Pi_{k-1})$ into $\cR(\Pi_k)$, we have
		\begin{equation}\label{eq:iota-ahat-M-proof}
			\iota_k\hat{\mathfrak a}_k
			=
			\mathfrak a_k\iota_{k-1}\quad\text{as maps from}\ \R^{\Omega_{k-1}}\ \text{to}\  \cR(\Pi_k)
			\fstop
		\end{equation}
		Indeed, the right-hand side belongs to $\cR(\Pi_k)$. Moreover, for all
		$g\in\R^{\Omega_{k-1}}$ and $\eta\in\Omega_k$,
		\begin{align}
			\mathfrak a_k\iota_{k-1}g(\eta)
			&=
			\sum_{x\in V}\eta_x\,\iota_{k-1}g(\eta-\delta_x)
			=
			\sum_{x\in V}\eta_x\,g(\eta-\delta_x)
			=
			\hat{\mathfrak a}_k g(\eta)
			\comma
		\end{align}
		where we used that $\eta-\delta_x\in\Omega_{k-1}$ whenever $\eta\in\Omega_k$ and
		$\eta_x\ge1$. Hence the two sides of \eqref{eq:iota-ahat-M-proof} are the same
		$B_k$-harmonic extension.
		
		Finally, by
		\eqref{eq:M-conj-M-proof}, \eqref{eq:iota-ahat-M-proof}, and
		\eqref{eq:PiA-cons-M-proof}, we get
		\begin{align}
			M_{\hat\alpha,k}\hat{\mathfrak a}_k 
			&=
			\rho_k\Pi_k A_{\hat\alpha,k}\iota_k
			\hat{\mathfrak a}_k 
			\notag\\
			&=
			\rho_k\Pi_k A_{\hat\alpha,k}\mathfrak a_k
			\iota_{k-1}
			\notag\\
			&=
			\rho_k\mathfrak a_k\Pi_{k-1}A_{\hat\alpha,k-1}
			\iota_{k-1}
			\notag\\
			&=
			\rho_k\mathfrak a_k\iota_{k-1}
			\rho_{k-1}\Pi_{k-1}A_{\hat\alpha,k-1}
			\iota_{k-1}
			\notag\\
			&=
			\hat{\mathfrak a}_k M_{\hat\alpha,k-1}
			\fstop
		\end{align}
		In the fourth line, we used that
		$\Pi_{k-1}A_{\hat\alpha,k-1}\iota_{k-1}$ has  $\cR(\Pi_{k-1})$ as an image, so that
		$\iota_{k-1}\rho_{k-1}$ acts as the identity on it. This concludes the proof of the proposition}
\end{proof}

 \subsubsection{Irreducible decompositions}\label{sec:irreducible-dec}
 \fs{By an \textit{irreducible component of $(\cM_{\hat\alpha,k}(t))_{t \geq 0}$ on
 	$\Omega_{k,m}$} we mean a communicating class for the restriction of the
 	transition graph to $\Omega_{k,m}$: two configurations $\eta,\xi \in
 	\Omega_{k,m}$ belong to the same component if there exist paths
 	$\eta=\eta_0,\eta_1,\ldots,\eta_\ell=\xi$ and
 	$\xi=\xi_0,\xi_1,\ldots,\xi_r=\eta$, all contained in $\Omega_{k,m}$, such that
 	\begin{equation}
 	\mathbf r^\cM_{\hat\alpha,k}(\eta_{h-1},\eta_h)>0\comma\qquad
 	\mathbf r^\cM_{\hat\alpha,k}(\xi_{h-1},\xi_h)>0\comma
 	\end{equation}
 	for all admissible indices. An irreducible decomposition is the
 	partition of $\Omega_{k,m}$ into such irreducible components.}
 
The chain $(\cM_{\hat\alpha,k}(t))_{t\ge 0}$ is not necessarily irreducible on $\Omega_{k,m}$.	For instance, when $G=\mathbb T_{2m}$ is the one-dimensional discrete torus of size $2m$, $$\fs{\Omega_{m,m}=\{\delta_1+\delta_3+\ldots+\delta_{2m-1},\delta_0+\delta_2+\ldots+\delta_{2m-2}\}\defeq\{\eta_{m,\rm odd},\eta_{m,\rm even}\}}$$ consists of exactly two configurations (one with particles/stacks on the odd numbers, the other one on the even numbers), and  \fs{${\bf r}_{\hat\alpha,k}^\cM(\eta_{m,\rm odd},\eta_{m,\rm even})={\bf r}_{\hat\alpha,k}^\cM(\eta_{m,\rm even},\eta_{m,\rm odd})=0$}.
Moreover, the structure of the irreducible components of the chain $(\cM_{\hat\alpha,k}(t))_{t\ge 0}$ on $\Omega_{k,m}$ is always \textit{finer} than the corresponding structure on $\Omega_{m,m}$, in the sense of the following proposition.
\begin{proposition}\label{pr:irreducibility}
Decompose $\Omega_{m,m}$ into irreducible  components for the chain $(\cM_{\hat\alpha,m}(t))_{t\ge0}$: for some integer $\mathfrak n_m\ge 1$ (depending only on $m\ge 2$ and $G$),
\begin{equation}\label{eq:Omegamm-dec}
\Omega_{m,m}=\Omega_{m,m}^1\sqcup\cdots\sqcup \Omega_{m,m}^{\mathfrak n_m}\fstop
\end{equation}  For all $k> m$ and $i\in \bbr{1}{\mathfrak n_m}$, define, iteratively,
\begin{equation}\label{Omegakm-dec-structure}
\Omega_{k,m}^i = \{ \eta + \delta_x \in \Omega_{k,m} : \eta \in \Omega_{k-1,m}^i, \  x \in V \ \text{such that} \ \eta_x \ge 1 \} \fstop
 \end{equation}
Then, there exist partitions $\Omega_{k,m}^i=\Omega_{k,m}^{i,1} \sqcup \cdots \sqcup \Omega_{k,m}^{i,\mathfrak{N}^k_i}$ for all $i \in \bbr{1}{\mathfrak n_m}$ such that
\begin{equation}\label{Omegakm-dec}
\Omega_{k,m}= \bigsqcup_{i=1}^{\mathfrak n_m} \bigsqcup_{j=1}^{\mathfrak N_i^k} \Omega_{k,m}^{i,j}
\end{equation}
is an \emph{irreducible}  decomposition of $\Omega_{k,m}$ for the chain $(\cM_{\hat\alpha,k}(t))_{t\ge0}$.
\end{proposition}
\begin{proof}
First, the fact that $\Omega_{k,m}$, for each $k \ge m \ge 2$, can be decomposed into irreducible components follows at once from Lemma \ref{lem:nu-self-adj}. Thus, to prove the validity of Proposition \ref{pr:irreducibility}, it suffices to check the following two properties:
\begin{enumerate}[(a)]
\item\label{it:irr1} $\Omega_{k,m} = \Omega_{k,m}^1 \sqcup \cdots \sqcup \Omega_{k,m}^{\mathfrak n_m}$;
\item \label{it:irr2} if $i \ne i'$, $\eta \in \Omega_{k,m}^i$, and $\xi \in \Omega_{k,m}^{i'}$, then $\eta$ and $\eta'$ belong to different irreducible components of $\Omega_{k,m}$ for $(\cM_{\hat\alpha,k}(t))_{t\ge0}$.
\end{enumerate}
To verify item \ref{it:irr1}, take $\eta \in \Omega_{k,m}$ and define a new configuration $\eta^*\in \Omega_{m,m}$ as
\begin{equation}\label{eq:eta-star-def}
\eta^* \eqdef \sum_{\substack{x \in V \\ \eta_x \ge 1}} \delta_x \fstop
\end{equation}
Since $\eta$ has $m$ separated stacks of particles,  $\eta^*$ belongs to $\Omega_{m,m}$. Then, according to \eqref{eq:Omegamm-dec}, $\eta^* \in \Omega_{m,m}^i$ for some $i$, thus $\eta \in \Omega_{k,m}^i$ by \eqref{Omegakm-dec-structure}. This proves that $\Omega_{k,m}=\Omega_{k,m}^1 \sqcup \cdots \sqcup \Omega_{k,m}^{\mathfrak n_m}$.

We move on to item \ref{it:irr2}. It suffices to prove that, for any $\eta,\xi \in \Omega_{k,m}$ with ${\bf r}_{\hat\alpha,k}^\cM(\eta,\xi)>0$, the corresponding configurations $\eta^*,\xi^* \in \Omega_{m,m}$ (cf.\ \eqref{eq:eta-star-def}) satisfy either $\eta^* = \xi^*$ or ${\bf r}_{\hat\alpha,m}^\cM(\eta^*,\xi^*)>0$.

In view of \eqref{eq:rM-def}, there exists $\zeta \in \Xi_k$ such that ${\bf r}_{\hat \alpha,k}^\cA ( \eta , \zeta ) > 0$ and ${\bf P}_\zeta^\cB \, [\tau_\xi = \tau_{\Omega_k}] > 0$. In turn, there exist $x,y \in V$ with $c_{xy}>0$ such that $\zeta = \eta - \delta_x + \delta_y$. First, suppose that $\zeta \in \Omega_k$, i.e., $\zeta = \xi$. Then, it is straightforward that $\zeta^*=\xi^* \in \Omega_m$ and ${\bf r}_{\hat\alpha,m}^\cA(\eta^*,\zeta^*)>0$, thus we have ${\bf r}_{\hat\alpha,m}^\cM(\eta^*,\xi^*)>0$ by \eqref{eq:rM-def}. On the contrary, suppose that $\zeta \in \Delta_k$. This means that $y$ has at least one neighboring site (possibly $x$) on which $\zeta$ has a particle. We divide into several cases.

\begin{itemize}
\item[\bf 1.] If $\zeta_x \ge 1$ and $\zeta_z = 0$ for all $z \ne x$ with $c_{yz}>0$, then the dynamics $(\cB_k(t))_{t\ge0}$ can only move particles between sites $x$ and $y$. Thus, by ${\bf P}_\zeta^\cB \, [\tau_\xi = \tau_{\Omega_k} ] >0$ it necessarily holds that $\xi = \eta - \eta_x \delta_x + \eta_x \delta_y \in \Omega_k$. This implies that $\xi^* = \eta^* - \delta_x + \delta_y \in \Omega_m$, thus ${\bf r}_{\hat\alpha,m}^\cM(\eta^*,\xi^*) > 0$.
\item[\bf 2.] Suppose that $\zeta_x \ge 1$ and there exists $z \ne x$ with $c_{yz}>0$ such that $\zeta_z \ge 1$. Then, the dynamics $(\cB_k(t))_{t\ge0}$ gets absorbed exactly when $y$ becomes empty. For $\xi$ to have the same number of stacks, $m$, as $\eta$, it  follows that $\xi^*=\eta^*$.
\item[\bf 3.] Suppose that $\zeta_x = 0$. Then, since the absorption of $(\cB_k(t))_{t\ge0}$ triggers at least one additional loss of a stack of particles, the absorbed configuration $\xi$ has strictly less number of stacks than $\eta$, which contradicts the assumption that $\eta,\xi \in \Omega_{k,m}$.
\end{itemize}

The above three cases conclude the proof of Proposition \ref{pr:irreducibility}.
\end{proof}
 
\begin{remark}\label{rem:strictly-finer}
The irreducible structure in \eqref{Omegakm-dec} may indeed get \emph{strictly} finer than the one in \eqref{eq:Omegamm-dec}. For example, consider the following \emph{H-shape} graph $G$ with $V=\bbr{1}{6}$, $c_{12}=c_{23}=c_{25}=c_{45}=c_{56}=1$, and $c_{xy}=0$ otherwise (see Figure \ref{fig5}). Then, $\Omega_{4,4}$ is a singleton set with $\Omega_{4,4}^1 = \{ \delta_1 + \delta_3 + \delta_4 + \delta_6 \}$, whereas $\Omega_{5,4}$ has two irreducible components $\Omega_{5,4}^{1,1}$ and $\Omega_{5,4}^{1,2}$, given as
\begin{align}
\Omega_{5,4}^{1,1} &= \{ 2\delta_1 + \delta_3 + \delta_4 + \delta_6 , \delta_1 + 2\delta_3 + \delta_4 + \delta_6  \}
\\
\Omega_{5,4}^{1,2} &= \{ \delta_1 + \delta_3 + 2\delta_4 + \delta_6 , \delta_1 + \delta_3 + \delta_4 + 2\delta_6  \} \fstop
\end{align}
\end{remark}

\begin{figure}
\begin{tikzpicture}
\draw[very thick] (0,0.1)--(0,0.9); \draw[very thick] (0,1.1)--(0,1.9); \draw[very thick] (0.1,1)--(0.9,1);
\draw[very thick] (1,0.1)--(1,0.9); \draw[very thick] (1,1.1)--(1,1.9);
\foreach \i in {0,1} { \foreach \j in {0,1,2} {
\draw (\i+0.1,\j)--(\i,\j+0.1)--(\i-0.1,\j)--(\i,\j-0.1)--(\i+0.1,\j); } }
\foreach \i in {0,1} { \foreach \j in {0,2} {
\fill[red] (\i,\j+0.08) circle (0.08); \draw (\i,\j+0.08) circle (0.08); } }
\draw (0.5,-0.3) node[below]{$\Omega_{4,4}^1$};

\draw(-0.1,2) node[left]{\tiny \textit{1}};
\draw(-0.1,1) node[left]{\tiny \textit{2}};
\draw(-0.1,0) node[left]{\tiny \textit{3}};
\draw(1.1,2) node[right]{\tiny \textit{4}};
\draw(1.1,1) node[right]{\tiny \textit{5}};
\draw(1.1,0) node[right]{\tiny \textit{6}};

\begin{scope}[shift={(3,0)}]
\draw[very thick] (0,0.1)--(0,0.9); \draw[very thick] (0,1.1)--(0,1.9); \draw[very thick] (0.1,1)--(0.9,1);
\draw[very thick] (1,0.1)--(1,0.9); \draw[very thick] (1,1.1)--(1,1.9);
\foreach \i in {0,1} { \foreach \j in {0,1,2} {
\draw (\i+0.1,\j)--(\i,\j+0.1)--(\i-0.1,\j)--(\i,\j-0.1)--(\i+0.1,\j); } }
\foreach \i in {0,1} { \foreach \j in {0,2} {
\fill[red] (\i,\j+0.08) circle (0.08); \draw (\i,\j+0.08) circle (0.08); } }
\fill[red] (0,2.24) circle (0.08); \draw (0,2.24) circle (0.08);
\end{scope}
\begin{scope}[shift={(5,0)}]
\draw[very thick] (0,0.1)--(0,0.9); \draw[very thick] (0,1.1)--(0,1.9); \draw[very thick] (0.1,1)--(0.9,1);
\draw[very thick] (1,0.1)--(1,0.9); \draw[very thick] (1,1.1)--(1,1.9);
\foreach \i in {0,1} { \foreach \j in {0,1,2} {
\draw (\i+0.1,\j)--(\i,\j+0.1)--(\i-0.1,\j)--(\i,\j-0.1)--(\i+0.1,\j); } }
\foreach \i in {0,1} { \foreach \j in {0,2} {
\fill[red] (\i,\j+0.08) circle (0.08); \draw (\i,\j+0.08) circle (0.08); } }
\fill[red] (0,0.24) circle (0.08); \draw (0,0.24) circle (0.08);
\end{scope}
\draw (4.5,-0.3) node[below]{$\Omega_{5,4}^{1,1}$};

\begin{scope}[shift={(8,0)}]
\draw[very thick] (0,0.1)--(0,0.9); \draw[very thick] (0,1.1)--(0,1.9); \draw[very thick] (0.1,1)--(0.9,1);
\draw[very thick] (1,0.1)--(1,0.9); \draw[very thick] (1,1.1)--(1,1.9);
\foreach \i in {0,1} { \foreach \j in {0,1,2} {
\draw (\i+0.1,\j)--(\i,\j+0.1)--(\i-0.1,\j)--(\i,\j-0.1)--(\i+0.1,\j); } }
\foreach \i in {0,1} { \foreach \j in {0,2} {
\fill[red] (\i,\j+0.08) circle (0.08); \draw (\i,\j+0.08) circle (0.08); } }
\fill[red] (1,2.24) circle (0.08); \draw (1,2.24) circle (0.08);
\end{scope}
\begin{scope}[shift={(10,0)}]
\draw[very thick] (0,0.1)--(0,0.9); \draw[very thick] (0,1.1)--(0,1.9); \draw[very thick] (0.1,1)--(0.9,1);
\draw[very thick] (1,0.1)--(1,0.9); \draw[very thick] (1,1.1)--(1,1.9);
\foreach \i in {0,1} { \foreach \j in {0,1,2} {
\draw (\i+0.1,\j)--(\i,\j+0.1)--(\i-0.1,\j)--(\i,\j-0.1)--(\i+0.1,\j); } }
\foreach \i in {0,1} { \foreach \j in {0,2} {
\fill[red] (\i,\j+0.08) circle (0.08); \draw (\i,\j+0.08) circle (0.08); } }
\fill[red] (1,0.24) circle (0.08); \draw (1,0.24) circle (0.08);
\end{scope}
\draw (9.5,-0.3) node[below]{$\Omega_{5,4}^{1,2}$};
\end{tikzpicture}\caption{\label{fig5}The H-shape graph $G$ and the corresponding absorbing spaces $\Omega_{4,4}=\Omega_{4,4}^1$ and $\Omega_{5,4}=\Omega_{5,4}^{1,1} \sqcup \Omega_{5,4}^{1,2}$ explained in Remark \ref{rem:strictly-finer}.}
\end{figure}
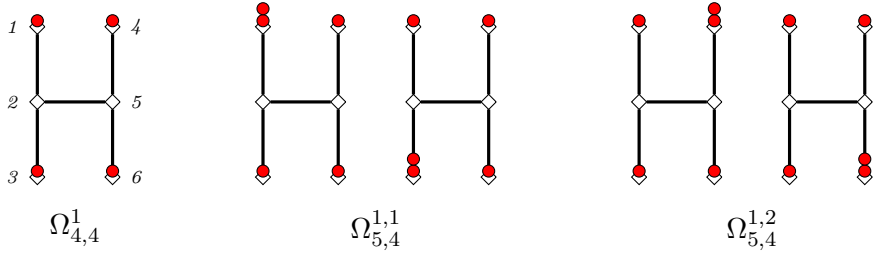

\subsubsection{Conclusion of the proof of \fs{Proposition \ref{pr:lambda-km-mm}}}\label{sec5.3.3}
\fs{We are now ready to prove  \eqref{eq:lambda-km-mm}.}

\begin{proof}[Proof of Proposition \ref{pr:lambda-km-mm}]
Recall the definition of the matrix $M_{\hat\alpha,k}$ and its submatrices $M_{\hat\alpha,k,m}$ from \eqref{eq:M} and \eqref{eq:Mkm}, respectively.

 First, we deal with the case $k=m$. We split $M_{\hat\alpha,m,m}$ into $\mathfrak n_m$ diagonal blocks $$M_{\hat\alpha,m,m}^1,\ldots,M_{\hat\alpha,m,m}^{\mathfrak n_m}\comma$$ each corresponding to an irreducible sub-Markovian dynamics. Adopting an analogous notation, we further write the eigenvalue $\lambda_{m,m}(\hat\alpha)$ in \eqref{eq:lambda-k-m} as
\begin{equation}
\lambda_{m,m}(\hat\alpha) = \lambda_{m,m}^1(\hat\alpha) \wedge \cdots \wedge \lambda_{m,m}^{\mathfrak n_m}(\hat\alpha) \comma
\end{equation}
where $\lambda_{m,m}^i(\hat\alpha)$ stands for the smallest eigenvalue of $-M_{\hat\alpha,m,m}^i$. By the Perron-Frobenius theorem, in view of sub-Markovianity, this eigenvalue is strictly positive. Furthermore, because of irreducibility, this is the only eigenvalue admitting a non-negative (actually, strictly positive) eigenfunction, which we refer to as
\begin{equation}
\psi_{\hat\alpha,m,m}^i \in \R^{\Omega_{m,m}^i} \fstop
\end{equation}
Let, for all $i\in \bbr{1}{\mathfrak n_m}$,
\begin{equation}\label{eq:mm-extend-def}
	\Psi_{\hat\alpha,m,m}^i (\eta) \eqdef \begin{cases} \psi_{\hat\alpha,m,m}^i (\eta) & \text{if } \ \eta \in \Omega_{m,m}^i  \\
		0 & \text{if } \ \eta \in \Omega_m \setminus \Omega_{m,m}^i \comma
	\end{cases}
\end{equation}
be the extension of the eigenfunction $\psi_{\hat\alpha,m,m}^i\in \R^{\Omega_{m,m}^i}$ to $\R^{\Omega_m}$. \fs{In particular, $\Psi_{\hat\alpha,m,m}^i$ is nonnegative on $\Omega_m$ and strictly positive on $\Omega_{m,m}^i$.} Then, due to the block lower triangular structure of $M_{\hat\alpha,m}$, we have
\begin{equation}\label{eq:eig-eq-mm}
	M_{\hat \alpha,m}\Psi_{\hat \alpha,m,m}^i = -\lambda_{m,m}^i(\hat \alpha)\, \Psi_{\hat\alpha,m,m}^i\fstop
\end{equation}
Since $ \R^{\Omega_m}\ni\Psi_{\hat\alpha,m,m}^i\neq 0$, this shows that a straightforward extension of $\psi_{\hat\alpha,m,m}^i$ produces an eigenfunction  for $M_{\hat\alpha,m}$.

	This canonical extension works well only for the case $k=m$, while it does not for $k>m$ in view of the block lower triangular structure of $M_{\hat\alpha,k}$.

For $k>m$, we follow a different route. In view of  Proposition \ref{pr:irreducibility}, as done above for $k=m$, we split $M_{\hat\alpha,k,m}$ into $\mathfrak N^k_1 + \cdots + \mathfrak N^k_{\mathfrak n_m}$ diagonal blocks, say, $M_{\hat\alpha,k,m}^{i,j}$ for $i \in \bbr{1}{\mathfrak n_m}$ and $j \in \bbr{1}{\mathfrak N_i^k}$, each corresponding to the irreducible sub-Markovian dynamics on $\Omega_{k,m}^{i,j}$. We refer to  Figure \ref{fig6}  for the overall structure of the matrix $M_{\hat\alpha,k}$. Then, we write the eigenvalue $\lambda_{k,m}(\hat\alpha)$ in \eqref{eq:lambda-k-m} as
 \begin{equation}\label{eq:lambdakm-dec}
 	\lambda_{k,m}(\hat\alpha) = \min_{i \in \bbr{1}{\mathfrak n_m}} \min_{j \in \bbr{1}{\mathfrak N_i^k}} \lambda_{k,m}^{i,j}(\hat\alpha)  \comma
 \end{equation}
 where $\lambda_{k,m}^{i,j}(\hat\alpha)$ is the smallest eigenvalue of $-M_{\hat\alpha,k,m}^{i,j}$, which admits a strictly positive eigenfunction $\psi_{\hat\alpha,k,m}^{i,j}\in \R^{\Omega_{k,m}^{i,j}}$, again by the Perron-Frobenius theorem.

Now, instead of the canonical extension as in \eqref{eq:mm-extend-def}, we define $\Psi_{\hat\alpha,k,m}^i \in \R^{\Omega_k}$ as the \textquotedblleft lifting\textquotedblright\ of $\Psi_{\hat \alpha,m,m}^i$ through the annihilation operators  in \eqref{eq:ann-hat}:
\begin{equation}\label{eq:lift-fcn-def}
	\Psi_{\hat\alpha,k,m}^i \eqdef ( \hat{\mathfrak{a}}_k \circ \cdots \circ \hat{\mathfrak{a}}_{m+1} ) \, \Psi_{\hat\alpha,m,m}^i \fstop
\end{equation}
\fs{In view of the definition of the annihilation operators in \eqref{eq:ann-hat} and $\Psi_{\hat\alpha,m,m}^i\ge 0$ (see \eqref{eq:mm-extend-def}), one proves, by induction on $k>m$,
\begin{equation}\label{eq:dom}
	\Psi_{\hat\alpha,k,m}^i(\eta+\xi)\ge \Psi_{\hat\alpha,m,m}^i(\eta)\comma\qquad \eta\in \Omega_m\comma \eta+\xi \in \Omega_k\fstop
	\end{equation}
Further,  $\Psi_{\hat\alpha,m,m}^i>0$ on $\Omega_{m,m}^i$ implies
\begin{equation} \label{eq:Psi-nonzero}
	\Psi_{\hat\alpha,k,m}^i\neq 0\fstop
\end{equation}}
By the intertwining relation \eqref{eq:Omega-intertwine} and the eigenvalue equation \eqref{eq:eig-eq-mm}, we obtain
\begin{align}
	\label{eq:eig-eqn-km}
	\begin{aligned}
		M_{\hat\alpha,k}  \Psi_{\hat\alpha,k,m}^i&=M_{\hat\alpha,k} \, \hat{\mathfrak{a}}_k \cdots \hat{\mathfrak{a}}_{m+1} \, \Psi_{\hat\alpha,m,m}^i\\
		& =  \hat{\mathfrak{a}}_k \cdots \hat{\mathfrak{a}}_{m+1} \, M_{\hat\alpha,m} \,  \Psi_{\hat\alpha,m,m}^i
		= -\lambda_{m,m}^i(\hat\alpha) \, \Psi_{\hat\alpha,k,m}^i \fstop
	\end{aligned}
\end{align}
The above identity \fs{and \eqref{eq:Psi-nonzero}}
 ensure that $\Psi_{\hat\alpha,k,m}^i\in \R^{\Omega_k}$ is an eigenfunction for $-M_{\hat\alpha,k}$ associated to the eigenvalue $\lambda_{m,m}^i(\hat\alpha)$. Moreover, \fs{after recalling \eqref{eq:lift-fcn-def} and \eqref{eq:ann-hat}}, we have\fs{
\begin{equation}
\Psi_{\hat\alpha,k,m}^i(\eta)
=
\sum_{\xi\in \Omega_m\,:\,
\xi\le \eta}
b_{k,m}(\eta,\xi)\,\Psi_{\hat\alpha,m,m}^i(\xi)
=0\comma
\qquad
\eta\in\Omega_{k,m'}^i,\ m'<m\comma
\end{equation}
where the coefficients $b_{k,m}(\eta,\xi)$ are nonnegative and supported on \begin{equation}\Omega_m(\eta)\eqdef\{\xi\in \Omega_m: \xi_x\le \eta_x\ \text{for all}\ x\in V\}\comma
	\end{equation} and the last identity used \eqref{eq:mm-extend-def}, $m'<m$, and that} \begin{equation}\fs{
	\eta \in \Omega_{k,m'}\ \text{and}\ \xi\in \Omega_m(\eta)\qquad \Longrightarrow\qquad \xi \in \Omega_{m,m''}\ \text{for some}\ m''\le m'\fstop}
\end{equation}
Because of the block lower triangular structure of $M_{\hat\alpha,k}$ and the irreducible decomposition \eqref{Omegakm-dec}, this ensures that the restriction
\begin{equation}
	\Psi_{\hat\alpha,k,m}^i\big|_{\Omega_{k,m}^{i,j}}\in \R^{\Omega_{k,m}^{i,j}}
\end{equation} of $\Psi_{\hat\alpha,k,m}^i\in \R^{\Omega_{k}}$ to $\R^{\Omega_{k,m}^{i,j}}$ is an eigenfunction for $-M_{\hat\alpha,k,m}^{i,j}$ with corresponding eigenvalue $\lambda_{m,m}^i(\hat\alpha)$. Finally,  by the aforementioned positivity of $\psi_{\hat\alpha,m,m}^i\in \R^{\Omega_{m,m}^i}$, 	we have
 \begin{equation}
 	\Psi_{\hat\alpha,k,m}^i (\eta) > 0\comma \qquad \text{for all} \ \eta \in \Omega_{k,m}^i \fstop
 \end{equation}
 \fs{Indeed, for $\eta\in\Omega_{k,m}^i$, the configuration
 	$\eta^\ast:=\sum_{x:\eta_x\ge 1}\delta_x$ belongs to $\Omega_{m,m}^i$, and
 	\eqref{eq:dom} applies with $\xi=\eta-\eta^\ast$} (see Figure \ref{fig6} for a visual proof).
 Since $M_{\hat\alpha,k,m}^{i,j}$ admits $\psi_{\hat\alpha,k,m}^{i,j}\in \R^{\Omega_{k,m}^i}$ as the only  positive eigenfunction (up to constant multiples), \fs{by the Perron-Frobenius theorem,} we  have 
 \begin{equation}
 	\Psi_{\hat\alpha,k,m}^i\big|_{\Omega_{k,m}^{i,j}} = c \, \psi_{\hat\alpha,k,m}^{i,j} \qquad \text{for some } c>0 \comma
 \end{equation}
 and, thus,  $\lambda_{m,m}^i(\hat\alpha) = \lambda_{k,m}^{i,j}(\hat\alpha)$. By \eqref{eq:lambdakm-dec}, the desired claim in \eqref{eq:lambda-km-mm} follows.
 \end{proof}

\begin{figure}
		\begin{tikzpicture}[scale=0.2]
			\draw[very thick] (0,35) rectangle (5,30); \draw[very thick] (5,30) rectangle (15,20); \draw[very thick] (20,15) rectangle (30,5);
			\draw[thick,densely dashed] (8,27) rectangle (12,23); \draw[thick,densely dashed] (23,12) rectangle (27,8);
			\draw (17.5,17.5) node[rotate=135]{$\bf{\cdots}$}; \draw (32.5,2.5) node[rotate=135]{$\bf{\cdots}$};
			\draw (6.5,28.5) node[rotate=135]{$\cdots$}; \draw (13.5,21.5) node[rotate=135]{$\cdots$}; \draw (21.5,13.5) node[rotate=135]{$\cdots$}; \draw (28.5,6.5) node[rotate=135]{$\cdots$};
			\draw (2.5,25) node{$\bf\star$}; \draw (2.5,17.5) node{$\bf\star$}; \draw (2.5,10) node{$\bf\star$}; \draw (2.5,2.5) node{$\bf\star$};
			\draw (10,17.5) node{$\bf\star$}; \draw (10,10) node{$\bf\star$}; \draw (10,2.5) node{$\bf\star$};
			\draw (17.5,10) node{$\bf\star$}; \draw (17.5,2.5) node{$\bf\star$}; \draw (25,2.5) node{$\bf\star$};
			\draw (10,32.5) node{$\bf 0$}; \draw (17.5,32.5) node{$\bf 0$}; \draw (25,32.5) node{$\bf 0$}; \draw (32.5,32.5) node{$\bf 0$};
			\draw (17.5,25) node{$\bf 0$}; \draw (25,25) node{$\bf 0$}; \draw (32.5,25) node{$\bf 0$};
			\draw (25,17.5) node{$\bf 0$}; \draw (32.5,17.5) node{$\bf 0$}; \draw (32.5,10) node{$\bf 0$};
			\draw (10,28.5) node{$0$}; \draw (13.5,28.5) node{$0$}; \draw (13.5,25) node{$0$}; \draw (6.5,25) node{$0$}; \draw (6.5,21.5) node{$0$}; \draw (10,21.5) node{$0$};
			\draw (25,13.5) node{$0$}; \draw (28.5,13.5) node{$0$}; \draw (28.5,10) node{$0$}; \draw (21.5,10) node{$0$}; \draw (21.5,6.5) node{$0$}; \draw (25,6.5) node{$0$};
			\draw[very thick] (0,0) rectangle (35,35);
			\draw (-4,32.5) node{$\Omega_{k,1}$}; \draw (-4,25) node{$\Omega_{k,2}$}; \draw (-4,10) node{$\Omega_{k,m}$};
			\draw (-4,17.5) node[rotate=90]{$\cdots$}; \draw (-4,2.5) node[rotate=90]{$\cdots$};
			\draw (2.5,32.5) node{\footnotesize $M_{k,1}$}; \draw (10,25) node{\footnotesize $M_{k,2}^{i,j}$}; \draw (25,10) node{\footnotesize $M_{k,m}^{i,j}$};
			\draw (17.5,-3) node{$M_k$};
			
			\begin{scope}[shift={(38,0)}]
				\fill[red!30!white] (0,8) rectangle (4,12);
				\draw[very thick] (0,0) rectangle (4,35); \draw[very thick] (0,30)--(4,30); \draw[very thick] (0,20)--(4,20); \draw[very thick] (0,15)--(4,15); \draw[very thick] (0,5)--(4,5);
				\draw[thick,densely dashed] (0,12)--(4,12); \draw[thick,densely dashed] (0,8)--(4,8);
				\foreach \i in {32.5,25,17.5} { \draw (2,\i) node{$0$}; } \foreach \i in {13.5,6.5,2.5} { \draw (2,\i) node{ $\star$}; } \draw (2,10) node{\footnotesize$\blacktriangle$};
				\draw (2,-3) node{$\Psi_{k,m}^i$};
			\end{scope}
			
			\draw (46,17.5) node{$=$};
			
			\begin{scope}[shift={(50,0)}]
				\fill[red!30!white] (0,8) rectangle (4,12);
				\draw[very thick] (0,0) rectangle (4,35); \draw[very thick] (0,30)--(4,30); \draw[very thick] (0,20)--(4,20); \draw[very thick] (0,15)--(4,15); \draw[very thick] (0,5)--(4,5);
				\draw[thick,densely dashed] (0,12)--(4,12); \draw[thick,densely dashed] (0,8)--(4,8);
				\foreach \i in {32.5,25,17.5} { \draw (2,\i) node{$0$}; } \foreach \i in {13.5,6.5,2.5} { \draw (2,\i) node{ $\star$}; }
				\draw (2,10) node{\footnotesize $\blacktriangledown$};
				\draw (2,-3) node{$-\lambda_{m,m}^i \, \Psi_{k,m}^i$};
			\end{scope}
		\end{tikzpicture}\caption{\label{fig6}In this figure, we depict the eigenvalue equation   in \eqref{eq:eig-eqn-km} (for  simplicity, we dropped $\hat\alpha$ from the notation). On the left-hand side, we show the $|\Omega_k| \times |\Omega_k|$ matrix $M_{\hat \alpha,k}$ in its  block lower triangular form. Stars ($\star$) indicate some possibly non-zero entries; triangles ($\blacktriangle$, resp.\ $\blacktriangledown$) strictly positive (resp.\ negative) ones. The  diagonal blocks in bold are the $|\Omega_{k,m}| \times |\Omega_{k,m}|$-matrices $M_{\hat\alpha,k,m}$, $m \in \bbr{1}{k}$, defined in \eqref{eq:Mkm}. Furthermore,  each $M_{\hat\alpha,k,m}$, $m\ge 2$, is further decomposed into block diagonal matrices  $M_{\hat\alpha,k,m}^{i,j}$,  $i \in \bbr{1}{\mathfrak{n}_m}$, $j \in \bbr{1}{\mathfrak N_i^k}$. Focusing on the red region $\Omega_{k,m}^{i,j}$, we readily verify that $\Psi_{\hat\alpha,k,m}^i \big|_{\Omega_{k,m}^{i,j}}$ is a strictly positive eigenfunction of $M_{\hat\alpha,k,m}^{i,j}$ associated to $-\lambda_{m,m}^i(\hat\alpha)$.}
\end{figure}

\fs{
	\subsection{$\lambda_{k,k}(\hat\alpha)$ and $k$ coalescing  particles}\label{sec5.4}
	This subsection is devoted to the proof of a monotonicity in $k$ of  $\lambda_{k,k}(\hat\alpha)$, defined in \eqref{eq:lambda-k-m} when $\Omega_{k,k}\neq \emp$, and set equal to $+\infty$ when $\Omega_{k,k}=\emp$ (see Corollary \ref{cor:wk}).

	\begin{proposition} For all integers $k\ge3$, we have
		\begin{equation}\label{eq:lambda-kk-monotone}
		\lambda_{k,k}(\hat\alpha)
		\geq
		\lambda_{k-1,k-1}(\hat\alpha)\fstop
	\end{equation}
		Consequently,
			\begin{equation}\label{eq:5.5-claim}
			\lambda_{k,k}(\hat\alpha) \ge \lambda_{2,2}(\hat\alpha)\comma \qquad  k \ge 2 \fstop
		\end{equation}
	\end{proposition}
	
	\begin{proof} Fix $k\ge 3$.
	As $\Omega_{k-1,k-1}=\emp$ implies $\Omega_{k,k}=\emp$,  we may assume that $\Omega_{k,k}\neq \emp$.

		On $\Omega_{k,k}$, the  dynamics described by $M_{\hat\alpha,k,k}$ coincides with $k$ independent copies of ${\rm RW}(\hat\alpha)={\rm RW}(G,\hat\alpha)$ killed upon exiting $\Omega_{k,k}$. 
		 For $\xi\in\Omega_{k,k}$, write $\mathsf P_\xi$ for the law of this process started from $\xi$, and define $\tau_k\ge 0$ as its killing time.
	Since the killed chain is finite-state, 
	\begin{equation}\label{eq:survival-principal-eigenvalue}
		\lambda_{k,k}(\hat\alpha)
		=
		-\lim_{t\to\infty}
		\frac1t
		\log
		\max_{\xi\in\Omega_{k,k}}
		\mathsf P_\xi[\tau_k>t] \fstop
	\end{equation}
	
	Fix $\xi\in\Omega_{k,k}$, and choose $z\in V$ such that $\xi_z=1$. Set
	$
		\xi^-\eqdef \xi-\delta_z\in\Omega_{k-1,k-1}$.
	We couple the $(k-1)$-particle process started from $\xi^-$ with the process obtained from the $k$-particle dynamics started from $\xi$ by ignoring the particle initially at $z$. Under this coupling,
	\begin{equation}
		\tau_k
		\leq
		\tau_{k-1} \fstop
	\end{equation}
	Indeed, any contact among the retained $k-1$ particles kills both systems, while contacts involving the ignored particle kill only the $k$-particle system. Hence, for all $t\geq 0$,
	\begin{equation}
		\mathsf P_\xi[\tau_k>t]
		\leq
		\mathsf P_{\xi^-}[\tau_{k-1}>t]
		\leq
		\max_{\xi'\in\Omega_{k-1,k-1}}
		\mathsf P_{\xi'}[\tau_{k-1}>t] \fstop
	\end{equation}
	By taking the maximum over $\xi\in\Omega_{k,k}$, the above bound and  \eqref{eq:survival-principal-eigenvalue} give \eqref{eq:lambda-kk-monotone}.
\end{proof}}

\subsection{Proofs of Theorems \ref{th:gap-asymptotics1} and \ref{th:failure-gap}}\label{sec5.5}
We are  ready to prove Theorems \ref{th:gap-asymptotics1} and \ref{th:failure-gap}. 

\begin{proof}[Proof of Theorem \ref{th:gap-asymptotics1}]
By Corollary \ref{lem:5.2} and \eqref{eq:gap-km-dec} \fs{(we omit $G$ from the notation)}, 
\begin{equation}\label{eq:limit}
\lim_{\eps \to 0} \frac{{\rm gap}_k(\eps \hat\alpha)}{\eps}= {\rm gap}_{\rm RW} (\hat\alpha) \wedge \min_{m \in \bbr{2}{k}} \lambda_{k,m}(\hat\alpha)  \comma
\end{equation}
whereas by  \eqref{eq:lambda-km-mm} and \eqref{eq:5.5-claim},
\begin{equation}\label{eq:limit2}
\min_{m \in \bbr{2}{k}} \lambda_{k,m}(\hat\alpha) = \lambda_{2,2}(\hat\alpha) \fstop
\end{equation}
Hence, since the limit in \eqref{eq:limit} is strictly positive,  Theorem \ref{th:gap-asymptotics1} holds true.
\end{proof}

\begin{proof}[Proof of Theorem \ref{th:failure-gap}] Recall that, for all graphs $G$ and  site weights $\hat \alpha$, we have ${{\rm gap}_1(G, \eps \hat\alpha)} =\eps\, {\rm gap}_{\rm RW} (G,\hat\alpha) $, $\eps>0$; moreover,  by \eqref{eq:limit} for $k=2$, we have
\begin{equation}
\lim_{\eps \to 0} \frac{{\rm gap}_2(G, \eps \hat\alpha)}{\eps} = {\rm gap}_{\rm RW} (G,\hat\alpha) \wedge \lambda_{2,2}(G,\hat\alpha)  \fstop
\end{equation}
Thus, it suffices to verify the existence of $G$ and $\hat\alpha$ such that
\begin{equation}\label{eq:fail-gap-pf}
\lambda_{2,2} (G,\hat\alpha) < {\rm gap}_{\rm RW} (G,\hat\alpha) \fstop
\end{equation}
In particular, let us recall from the discussion in Section \ref{sec5.4}, that $\lambda_{2,2}$ is the smallest eigenvalue of (the negative of) the generator describing two independent ${\rm RW}(G,\hat\alpha)$,  having $\hat\alpha=(\hat\alpha_x)_{x\in V}$ as a (non-normalized) reversible measure, and which get instantaneously killed when attempting to get to a mutual graph-distance strictly smaller than two. 

In view of this, it is not surprising that the standard  discrete torus $\mathbb{T}_N^d$, with dimension $d\ge2$, size $N\in \N$, and homogeneous site weights, satisfies \eqref{eq:fail-gap-pf}, provided that $N$ is large enough. From now on, we fix  $\hat \alpha\equiv 1$\fs{, $d\ge2$, $N\in \N$,} and
\begin{equation}G=\mathbb{T}_N^d\comma \text{with}\ c_{xy}\equiv1\ \text{for all neighboring sites}\ x,y\in \mathbb{T}_N^d\comma
\end{equation}
	  and drop them from the notation, simply writing, e.g., $\lambda_{2,2}$ and ${\rm gap}_{\rm RW}$ instead of $\lambda_{2,2}(G,\hat \alpha)$ and ${\rm RW}(G,\hat \alpha)$. \fs{Later on, we will take the limit  $N\to \infty$.}
	  
	  Consider the positions $(X_t,Y_t)_{t\ge0}$ of two independent particles on  $\mathbb{T}_N^d$, each one jumping with unit rate to any of its neighbors.
	  Observe that $(X_t, Y_t)_{t\ge0}$ is reversible with respect to the uniform measure $\kappa=\kappa_{N,d}$ on $\mathbb{T}_N^d \times \mathbb{T}_N^d$. In what follows, let $\mathsf P_{(x,y)}$ and $\mathsf E_{(x,y)}$ denote the law and corresponding expectation \fs{with respect to the law of} $(X_t,Y_t)_{t\ge 0}$ when starting from $(x,y)\in \mathbb{T}_N^d\times\mathbb{T}_N^d$.  It is well known (see, e.g., \cite[Sections 12.3.1 \& 12.4]{levin2017markov}) that, for all $N\in \N$ large enough and all $d\ge 1$, we have
	  \begin{equation}\label{eq:diffusive}
	  	\fs{{\rm gap}_{\rm RW}   =2\tonde{1-\cos\frac{2\pi}{N}} \ge  \frac{\pi^2}{N^2}\fstop}
	  \end{equation} Letting
	  \begin{equation}
	  	\cW \eqdef \{(x,y)\in \mathbb T^d_N: {\rm dist}_{\mathbb{T}_N^d}(x,y)\ge 2\} \comma
	  \end{equation}
by the second inequality in \cite[Eq.\ (1.7)]{hermon2023relaxation}, we have
\begin{equation}\label{eq:lambda22-UB}
\lambda_{2,2} \le \frac{1}{\max_{ \cU\subseteq \cW} \mathsf E_{\kappa( \emparg| \cU)}  [ \tau_{\cU^\complement}] }\le \frac{1}{\mathsf E_{\kappa( \emparg| \cW)}  [ \tau_{\cW^\complement}] } \comma
\end{equation}
where in the second step we substituted $\cU=\cW$. 
The last denominator in \eqref{eq:lambda22-UB} is almost the mean meeting time of two independent random walks, both being uniformly distributed on $\mathbb{T}_N^d$ (and independently) at time $t=0$. This latter quantity is known \cite[Theorem 4]{cox_coalescing_1989} to asymptotically diverge, as $N\to \infty$, like $s_d(N)$ \fs{given in \eqref{eq:cox}}.
In the remainder of this proof, we show that the denominator in \eqref{eq:lambda22-UB} follows the same asymptotic behavior.

Since $\kappa(\cW) = 1 - \frac{2d+1}{N^d}$, we calculate
\begin{equation}\label{eqeqeq1}
\mathsf E_{\kappa( \emparg\mid \Omega)} 	  [ \tau_{\cW^\complement}] = \left(1-\frac{2d+1}{N^d}\right)^{-1} \frac{1}{N^{2d}} \sum_{(x,y)\in \cW} \mathsf E_{(x,y)} [\tau_{\cW^\complement}] \fstop
\end{equation}
Moreover,  since  $Z_t\eqdef X_t-Y_t\in \mathbb T^d_N$ is a simple random walk (jumping to a nearest neighbor at rate $2$), we have
\begin{equation}\label{eqeqeq2}
\mathsf E_{(x,y)}  [\tau_{\cW^\complement}] = \widehat {\mathsf E}_{x-y} [\hat\tau_1] \comma\qquad (x,y)\in \cW\comma
\end{equation}
where we use the hat-notation to refer to quantities related to $Z_t$,  while $\hat\tau_\ell$, $\ell \in \bbr{0}{\lfloor N/2\rfloor} $,  denotes the first hitting time of the set
\begin{equation}
\{z\in \mathbb T^d_N: {\rm dist}_{\mathbb{T}_N^d}(z,0)=\ell\}\fstop
\end{equation}
By the strong Markov property and translation invariance of the dynamics, we have
\begin{equation}
\widehat {\mathsf E}_{x-y}  [\hat\tau_0] = \widehat {\mathsf E}_{x-y}  [\hat\tau_1] + \widehat {\mathsf E}_{e_1}  [\hat\tau_0] = \widehat {\mathsf E}_{x-y}  [\hat\tau_1] + \widehat {\mathsf E}_0 [\hat\tau_0^+] - \widehat {\mathsf E}_0  [\hat\tau_1] \comma
\end{equation}
where $\hat\tau_0^+$ denotes  the first return time to $0$ for the walk $Z_t$ started in $0$, i.e., 
\begin{equation}
\hat\tau_0^+\eqdef \inf\{t> \hat\tau_1: Z_t=0\}\fstop
\end{equation}
Clearly, $\widehat {\mathsf E}_0\, [\hat\tau_1] =  \frac{1}{4d}$, while,  by Kac's formula for return times (see, e.g., \cite[p.\ 34, Lemma 2.25]{aldous-fill-2014}), we obtain
\begin{equation}
\widehat {\mathsf E}_0  [\hat\tau_0^+] = \frac{N^d}{4d}\fstop
\end{equation}
Hence, we get
\begin{equation}\label{eqeqeq3}
\begin{aligned}
\frac{1}{N^{2d}} \sum_{(x,y) \in \cW} \widehat {\mathsf E}_{x-y} [\hat\tau_1] &= \frac{1}{N^{2d}} \sum_{(x,y) \in \cW} \tonde{\widehat {\mathsf E}_{x-y} [\hat\tau_0] - \widehat {\mathsf E}_0 	 [\hat\tau_0^+] + \widehat {\mathsf E}_0 [\hat\tau_1] } \\
&\ge \left\{\frac{1}{N^{2d}}\sum_{(x,y) \in \cW} \widehat {\mathsf E}_{x-y} [\hat\tau_0]\right\} - \frac{N^d}{4d} \fstop
\end{aligned}
\end{equation}
By \cite[Theorem 4]{cox_coalescing_1989} (see also \cite[Lemma 7.3.2]{durrett_dynamics_book} for a more recent textbook version), the expression between curly brackets is, for all $N\in \N$ large enough, bounded below by $\mathfrak g_d\,s_d(N)$,   for some positive constant $\mathfrak g_d$ (which, for $d\ge 3$, is strictly larger than $\frac1{4d}$, cf.\ \cite[Eq.\ (1.2)]{cox_coalescing_1989}).	
Hence, collecting \eqref{eqeqeq1}, \eqref{eqeqeq2} and \eqref{eqeqeq3}, we obtain, for some $C>0$ \fs{(possibly depending only on $d\ge 2$, but not on $N\in \N$)} and for all $N\in \N$ large enough,
\begin{equation}
	{\mathsf E}_{\kappa( \emparg| \cW)}  [ \tau_{\cW^\complement}] \ge C s_d(N)\comma
\end{equation}which, along with \eqref{eq:lambda22-UB}, implies
\begin{equation}
\lambda_{2,2} \le \frac1{Cs_d(N)} \fstop
\end{equation}
Comparing this inequality with \eqref{eq:diffusive} (cf.\ \eqref{eq:cox}) completes the proof of the theorem.
\end{proof}

\fs{We conclude this section with a remark on the role of the directional limits in \eqref{eq:alpha-eps-beta} in proving Theorem \ref{th:gap-asymptotics1}.

\begin{remark}[Directional \textit{vs.}\ general limits]\label{rem:general-alpha-limit}
	Theorem \ref{th:gap-asymptotics1} is stated for directional limits
	$\alpha=\varepsilon\hat\alpha$, with $\hat\alpha\in(0,\infty)^V$ fixed and
	$\varepsilon\to0$; see \eqref{eq:alpha-eps-beta}. Let us briefly explain what would be needed in order to pass from \eqref{eq:alpha-eps-beta} to a completely general limit $\alpha\to0$.
	
	A natural target would be the lower bound (we omit $G$ from the notation)
	\begin{equation}\label{eq:general-liminf-gap-ratio}
		\liminf_{\alpha\to0}
		\frac{{\rm gap}_k(\alpha)}{{\rm gap}_2(\alpha)}
		\ge 1\comma
		\qquad k\ge3\fstop
	\end{equation}
	Since, by \eqref{eq:gap-consistency}, one has
	${\rm gap}_k(\alpha)\le{\rm gap}_2(\alpha)$ for all $k\ge3$,
	\eqref{eq:general-liminf-gap-ratio} would imply the full limit
	\begin{equation}
	\lim_{\alpha\to0}
	\frac{{\rm gap}_k(\alpha)}{{\rm gap}_2(\alpha)}
	=1\fstop
	\end{equation}
	By the sequential characterization of the liminf, it suffices to consider arbitrary
	sequences $\alpha_n\to0$, as $n\to \infty$, in $(0,\infty)^V$. By setting
	\begin{equation}
\eps_n\eqdef(\alpha_n)_{\rm max}=\|\alpha_n\|_\infty\comma
\qquad	\beta_n\eqdef \eps_n^{-1}\alpha_n\comma
	\end{equation}
	we have $\eps_n\to0$ and, up to extracting a subsequence, $\beta_n\to \hat\alpha$, 
	for some $\hat\alpha\in[0,1]^V$ with $\max_{x\in V}\hat\alpha_x=1$. Along such a
	subsequence, the analogue of \eqref{Ak-Bk-dec} reads as
	\begin{equation}
	\eps_n^{-1}L_{\alpha_n,k}
	=
	 A_{\beta_n,k}+\eps_n^{-1}B_k\comma
	\end{equation}
	where $A_{\beta_n,k}$ is the slow part and  $B_k$ is the fast part of the generator, see \eqref{eq:Ak} and \eqref{eq:Bk}. Since
	$A_{\beta_n,k}\to A_{\hat\alpha,k}$ as finite matrices, the slow--fast reduction used in
	Proposition \ref{lem:kurtz} applies, with minor notational changes, also to this
	$n$-dependent situation. The limiting effective generator $\mathscr G_{\hat\alpha,k}$ is given by the same formula there.
	
	If $\hat\alpha\in(0,\infty)^V$, then the rest of the proof of Theorem
	\ref{th:gap-asymptotics1} goes through without changes: the limiting quantities
	$w_k(\hat\alpha)$ and $\lambda_{k,m}(\hat\alpha)$ appearing in \eqref{wk-def} and
	\eqref{eq:lambda-k-m} are exactly those analyzed in Sections
	\ref{sec5.2}--\ref{sec5.4}. Thus, strictly positive subsequential directions do not create
	any additional difficulty.
	
The only genuinely new case is when $\hat\alpha_x=0$ for some $x\in V$. Then, the
limiting effective generator may acquire additional zero eigenvalues: sites with
$\hat\alpha_x=0$ may generate further transient states and, depending on the graph,
may also disconnect irreducible components of the limiting effective dynamics.
In particular, the effective gap at scale $\varepsilon_n=(\alpha_n)_{\rm max}$ may vanish,
so that the relevant gaps may be of smaller order than $\varepsilon_n$. In such a
case, the first-order slow--fast limit does not determine the ratio
${\rm gap}_k(\alpha_n)/{\rm gap}_2(\alpha_n)$; one would have to analyze the next
non-vanishing scales of the coordinates of $\alpha_n$. Although one expects the same
two-particle mechanism to arise also through this further multi-scale analysis, we do
not pursue this extension here, and prefer to state Theorem \ref{th:gap-asymptotics1}
in the clean directional regime with $\hat\alpha\in(0,\infty)^V$ fixed.
	\end{remark}}

\section{Non-conservative SIP. Setting and main results}\label{sec:open-SIP}
In this  section, we attach particle reservoirs to some sites of the graph. We describe this by introducing two sets of non-negative site parameters,  $\omega=(\omega_x)_{x\in V}$ and $\theta=(\theta_x)_{x\in V}$. The former one represents the  rates of interaction between a site and its reservoir. The latter one prescribes the reservoirs' particle densities. Next to the inclusion particles' motion considered so far, we now let particles be created and annihilated with rates resembling those of ${\rm SIP}$.

\subsection{Model}\label{sec:non-conservative-stat-measures}
Let  $G=(V,(c_{xy})_{x,y\in V})$ and $\alpha=(\alpha_x)_{x\in V}$ be as in Section \ref{sec2}, and introduce some non-negative site weights $\omega=(\omega_x)_{x\in V}$ and $\theta=(\theta_x)_{x\in V}$. Then,  we write ${\rm SIP}(G,\alpha,\omega,\theta)$ for the Markov process  on\footnote{$\Xi_0\eqdef \{\emptyset\}$ consists of the empty configuration $\emptyset$ with no particles.} $\Xi\eqdef \cup_{k=0}^\infty\, \Xi_k$ with infinitesimal generator $L_{G,\alpha,\omega,\theta}$ given, for a bounded function $f\in \R^{\Xi}$ and a configuration $\eta\in \Xi$, by
\begin{equation}\label{eq:gen-res}
	\begin{aligned}
		&L_{G,\alpha,\omega,\theta}f(\eta)= \sum_{x,y\in V}c_{xy}\,\eta_x\tonde{\alpha_y+\eta_y}\tonde{f(\eta-\delta_x+\delta_y)-f(\eta)}\\
		&+ \sum_{x\in V}\omega_x\,\big\{  \eta_x\tonde{1+\theta_x}\tonde{f(\eta-\delta_x)-f(\eta)}+\theta_x\tonde{\alpha_x+\eta_x}\tonde{f(\eta+\delta_x)-f(\eta)}\big\}\fstop
	\end{aligned}
\end{equation}
Here, the first summation on the right-hand side corresponds with the conservative dynamics described by $L_{G,\alpha,k}$ in \eqref{eq:gen-conservative}, and is referred to as bulk dynamics. The second summation concerns the particle creation-annihilation mechanism, and may be interpreted as a bulk-boundary dynamics. More in detail, particles at $x\in V$ get killed independently, each at rate $\omega_x\tonde{1+\theta_x}$, while a new particle is created at $x\in V$ with a configuration-dependent rate $\omega_x\,\theta_x\tonde{\alpha_x+\eta_x}$.
Furthermore, albeit the configuration space $\Xi$ is countably infinite,  ${\rm SIP}(G,\alpha,\omega,\theta)$ does not explode in finite time. Indeed, the total number of particles is stochastically dominated by a pure birth process with linearly growing rates (see, e.g., \cite[Proposition 2.1]{franceschini2020symmetric}).

The long-time behavior of ${\rm SIP}(G,\alpha,\omega,\theta)$ clearly depends on the values of $\omega$ and $\theta$. To start with, we mention that $\omega_x=0$ means that site $x\in V$ does not directly exchange particles with any reservoirs, whereas
$\theta_x=0$ corresponds to having a purely absorbing reservoir at $x\in V$. As a consequence, if $\omega \equiv 0$, we recover the conservative case having, for any total number of particles $k\in \N$, $\mu_{\alpha,k}$ in \eqref{mu-def} as a unique reversible measure. If $\theta\equiv 0$ and $\omega \not\equiv 0$, the system eventually empties out. In other words, if $\theta$ vanishes on the support of $\omega$, that is,
\begin{equation}
	{\rm supp}(\omega)\eqdef \{x\in V: \omega_x > 0\}\comma
\end{equation}
then ${\rm SIP}(G,\alpha,\omega,\theta)$ has  the empty configuration $\emptyset \in \Xi_0$ as its unique stationary state, which, in turn, is absorbing.
More generally, for $\omega \not \equiv 0$, there exists a unique stationary measure for ${\rm SIP}(G,\alpha,\omega,\theta)$, referred to as $\mu_{G,\alpha,\omega,\theta}$.  
In this situation, one typically distinguishes between two scenarios, commonly referred to as {equilibrium} and {non-equilibrium}, respectively. 

More in detail, \emph{equilibrium} occurs if $\theta $ is constant  on 
${\rm supp}(\omega)$. In this case,  $\mu_{G,\alpha,\omega,\theta}$ is reversible, in product form, and explicit \fs{(see, e.g., \cite{carinci_duality_2013-1,floreani_boundary2020})}: assuming that  $\theta\equiv \varrho > 0$, then $	\mu_{G,\alpha,\omega,\theta}= \nu_{\alpha,\varrho}$, where,  for all $\eta \in \Xi$,
\begin{equation}\label{eq:nu-sigma}
 \nu_{\alpha,\varrho}(\eta)\eqdef\frac{1}{\cZ_{\alpha,\varrho}} \tonde{\frac{\varrho}{1+\varrho}}^{|\eta|} \prod_{x\in V}\frac{\Gamma(\alpha_x+\eta_x)}{\Gamma(\alpha_x)\,\eta_x!}\comma\qquad \text{with}\ \cZ_{\alpha,\varrho}\eqdef (1+\varrho)^{|\alpha|}\fstop
\end{equation} 
Note that $\mu_{\alpha,k}$ in \eqref{mu-def} is the canonical measure associated to the grand-canonical $\nu_{\alpha,\varrho}$.
One speaks of \emph{non-equilibrium} if $\theta$ is non-constant on ${\rm supp}(\omega)$. Then, $\mu_{G,\alpha,\omega,\theta}$ is, in general, neither reversible, nor in product form, nor explicit. We will not need to know more on the  structure of this steady state; for further details, see, e.g.,  \cite{floreani_boundary2020} and references therein.

\subsection{Spectral gap identity when $\theta\equiv 0$}
We start with the analysis of the purely absorbing ${\rm SIP}$. As already anticipated in Section \ref{sec:intro-non-consevative}, we shall see that understanding the case $\theta\equiv 0$ is one of the key ingredients for the spectral analysis of the setting with general  $\theta$.

Observe that substituting $\theta\equiv 0$ in \eqref{eq:gen-res} gives rise to a  system in which all particles are never created,  and get annihilated one at the time, until the system eventually gets empty. In particular,  $L_{G,\alpha,\omega,0}$, interpreted as an infinite matrix, is similar to a block lower triangular one. Moreover, the blocks on the diagonal --- each for any total number of particles $k\in \N_0$ in the bulk of the system --- are of finite size and describe the following  finite-state Markov chains.
The case   $k=1$  corresponds to a process with just one particle in the bulk, which  moves as ${\rm RW}(G,\alpha)$ defined in Section \ref{sec2},   additionally killed at rate $\omega_x\ge 0$ when sitting on $x\in V$. We refer to this process as ${\rm RW}(G,\alpha,\omega)$. The case $k\ge 2$ is similar, with $k$ interacting particles evolving in the bulk as ${\rm SIP}_k(G,\alpha)$, each independently killed with rate $\omega_x\ge 0$ when sitting at $x\in V$. Each of these Markov chains, say ${\rm SIP}_k(G,\alpha,\omega)$, \fs{is} sub-stochastic. Moreover, since the conservative ${\rm SIP}_k(G,\alpha)$ admits $\mu_{\alpha,k}$ in \eqref{mu-def} as its unique reversible measure, all eigenvalues of (the negative of) the generator of ${\rm SIP}_k(G,\alpha,\omega)$ (see \eqref{eq:gen-res-abs}) are real; \fs{as the conservative system is irreducible, ${\rm supp}(\omega)\neq \emp$ further ensures that such eigenvalues are} strictly positive. 
Letting ${\rm gap}_k(G,\alpha,\omega)$, $k\ge 1$, be the lowest of such eigenvalues, we write	(cf.\ \eqref{eq:gap-sip-def})
\begin{equation}\label{eq:gap-sip-gap-k-non-conservative}
	{\rm gap}_{\rm SIP}(G,\alpha,\omega)=\inf_{k\ge 1}{\rm gap}_k(G,\alpha,\omega)\fstop
\end{equation}
Similarly, we introduce the spectral gap of ${\rm RW}(G,\alpha,\omega)$ as \begin{equation}\label{eq:gap-rw-abs}
	{\rm gap}_{\rm RW}(G,\alpha,\omega)= {\rm gap}_1(G,\alpha,\omega)\fstop\end{equation}

Our  main result of this section establishes that, as soon as the killing rate $\omega$ is not identically zero, the spectral gaps in \eqref{eq:gap-sip-gap-k-non-conservative} and \eqref{eq:gap-rw-abs} coincide,    regardless of the underlying geometry and of the site weights ---  even when $\alpha_{\rm min}\in (0,1)$.	 This spectral gap identity stands in stark contrast to the behavior observed in the conservative setting, where  restrictions are imposed on the value of  $\alpha_{\rm min}$.
\begin{theorem}[$\theta\equiv0$]\label{th:non-conservative} For all graphs $G=(V,(c_{xy})_{x,y\in V})$, and site weights $\alpha=(\alpha_x)_{x\in V}$, $\omega=(\omega_x)_{x\in V}$ with $\omega\neq 0$, we have
	\begin{equation}
		{\rm gap}_{\rm SIP}(G,\alpha,\omega)= {\rm gap}_{\rm RW}(G,\alpha,\omega)\fstop
	\end{equation}
\end{theorem}
 Remarkably, the proof of the above identity (deferred to Section \ref{sec:proof-gap-abs}) is quite elementary.

\subsection{General setting}
When dealing with the general case $\theta\neq 0$, we must specify the infinite-dimensional function analytic setting that we work in. For this purpose, we define a class of functions on $\Xi$ (Definition \ref{def:polynomials}), written in terms of  orthogonal \textquotedblleft Meixner\textquotedblright\ polynomial  dualities introduced in \cite[Section 4]{floreani_boundary2020}, which take the following form: for all $\varrho\ge 0$,
\begin{equation}\label{eq:duality}
	D_{\alpha,\varrho}(\xi,\eta)=\prod_{x\in V}d_{\alpha_x,\varrho}(\xi_x,\eta_x)\comma\qquad \eta\in \Xi\comma  \xi\in \Xi_k\comma k\ge 0\comma
\end{equation}
where, for $\varrho=0$,
\begin{align}
	d_{\alpha_x,0}(\xi_x,\eta_x) = \frac{\eta_x!}{\tonde{\eta_x-\xi_x}!}\frac{\Gamma(\alpha_x)}{\Gamma(\alpha_x+\xi_x)}\car_{\set{\xi_x\le \eta_x}}\comma
\end{align}
whereas, for $\varrho>0$, 
\begin{equation}\label{eq:meixner}
	d_{\alpha_x,\varrho}(\xi_x,\eta_x) =\sum_{\ell_x=0}^{\xi_x}\binom{\xi_x}{\ell_x}\,d_{\alpha_x,0}(\ell_x,\eta_x)\tonde{-\varrho}^{\xi_x-\ell_x}=	\tonde{-\varrho}^{\xi_k}\fs{{}_2F_1(-\xi_x,-\eta_x;\alpha_x;-\varrho^{-1})} \fstop
\end{equation}
This last identity is just a rewriting in terms of the ordinary hypergeometric function. Let us observe that $D_{\alpha,\varrho}(\emptyset,\emparg)\equiv 1$ and that, for all $k\in \N$ and $\xi\in \Xi_k$, $\eta \in \Xi\longmapsto D_{\alpha,\varrho}(\xi,\eta)\in \R
$  is a $k$-th order polynomial in  $\eta=(\eta_x)_{x\in V}$.	Moreover, as we shall show in Section \ref{sec:proof-th-eigen} below,	 \begin{equation}\label{eq:meixner-duality}
	\{D_{\alpha,\varrho}(\xi,\emparg): |\xi|\le k\}\end{equation} is a linearly independent set which spans the space of polynomials in $\eta$ of order $\le k$. 	This and the two other properties presented in Proposition \ref{pr:duality} below turn the functions in \eqref{eq:duality} into a convenient basis of polynomials on $\Xi$.

\begin{definition}\label{def:polynomials}
		For all $\varrho> 0$, integers $k\ge 1$, and functions $\psi\in \R^{\Xi_k}$, define
	\begin{equation}\label{eq:f-lift}
		F_{\alpha,\varrho}^\psi(\eta)\eqdef \sum_{\xi\in \Xi_k}\mu_{\alpha,k}(\xi)\,\psi(\xi)\, D_{\alpha,\varrho}(\xi,\eta)\comma \qquad \eta\in \Xi\fstop
	\end{equation}
	When $k=0$ and, thus, $\psi\in \R$ is a scalar, we conventionally set $F_{\alpha,\varrho}^\psi\equiv \psi\in \R$. 
\end{definition}
Remark that, as long as $\psi\neq 0$, the function in \eqref{eq:f-lift} is a non-vanishing $k$-th order polynomial in the variables $\eta=(\eta_x)_{x\in V}$. In particular, expressions like $L_{G,\alpha,\omega,\theta} F_{\alpha,\varrho}^\psi$ will always be well defined.
Finally,  introduce the operators
$\mathfrak b_{\alpha,\omega,\varrho,k-1}^\theta: \R^{\Xi_k}\to \R^{\Xi_{k-1}}$, $k\ge 1$, acting on functions $\psi \in \R^{\Xi_k}$ as
\begin{equation}\label{eq:mathfrak-b}
	\mathfrak b_{\alpha,\omega,\varrho,k-1}^\theta \psi(\zeta)\eqdef  k\sum_{x\in V}\omega_x\tonde{\theta_x-\varrho}\frac{\alpha_x+\zeta_x}{|\alpha|+k-1}\,\psi(\zeta+\delta_x)\comma\qquad \zeta \in \Xi_{k-1}\fstop
\end{equation}
Clearly, we have
\begin{equation}\label{eq:mathfrak-b-zero}
	\mathfrak b_{\alpha,\omega,\varrho,k-1}^\theta = 0\comma\qquad \text{if}\ \theta\equiv \varrho\fstop
\end{equation}
The following theorem, whose proof is postponed to Section \ref{sec:proof-th-eigen} below, is stated for any graph $G$ and site weights $\alpha$, $\omega$, and $\theta$.
\begin{theorem}\label{th:lifting-general}
 Fix $\varrho>0$ and recall Definition \ref{def:polynomials}.
	\begin{enumerate}[(a)]
		\item\label{it:lifting-general} Given $k\ge 0$, any eigenpair  $\lambda,\psi$ of the $k$-particle purely-absorbing system solves
			\begin{equation}\label{eq:eigen-lift-2}
			L_{G,\alpha,\omega,\theta}F_{\alpha,\varrho}^\psi = -\lambda\, F_{\alpha,\varrho}^{\psi} + F_{\alpha,\varrho}^{\mathfrak b_{\alpha,\omega,\varrho,k-1}^\theta\psi}\fstop
		\end{equation}
		\item \label{it:cons}  The functions $F_{\alpha,\varrho}^\psi$  as in Definition \ref{def:polynomials} may be chosen to form an orthogonal basis of $L^2(\nu_{\alpha,\varrho})$.	
	\end{enumerate}
\end{theorem}
Let us comment on two powerful consequences of the above theorem in the reversible case $\theta\equiv \varrho$ (we discuss the non-reversible case in Remark \ref{rem:non-rev} below\fs{; see also Remark \ref{rem:SEP} for a comparison with the symmetric \textit{exclusion} process}).
 The first item tells us that  any eigenvalue of the purely absorbing system with $k\ge 1$ particles is one also for the general system and, moreover, it provides a constructive procedure to build  eigenfunctions from those of the $k$-particle absorbing system.  Indeed,  by \eqref{eq:mathfrak-b-zero}, the identity in \eqref{eq:eigen-lift-2} becomes  the eigenvalue equation
 \begin{equation}
 	L_{G,\alpha,\omega,\theta}F_{\alpha,\varrho}^\psi = -\lambda\, F_{\alpha,\varrho}^\psi\comma
 \end{equation}
 because the second term on the right-hand side of \eqref{eq:eigen-lift-2} vanishes when $\theta\equiv \varrho$.
 The second item guarantees the existence of a \fs{complete orthonormal system} of eigenfunctions of $L_{G,\alpha,\omega,\theta}$. This, combined with the fact that the  generator is self-adjoint in the associated $L^2$-space, ensures that $-L_{G,\alpha,\omega,\theta}$ has a pure point spectrum therein, all contained in $[0,\infty)$, with a gap after the zero eigenvalue coinciding with that of the purely absorbing system. 
 Letting \begin{equation}{\rm gap}_{\rm SIP}(G,\alpha,\omega,\varrho)\end{equation} denote this spectral gap when $\theta\equiv \varrho>0$, Theorem \ref{th:non-conservative} readily implies the following result.

\begin{corollary}[$\theta\equiv \varrho$]\label{cor:non-conservative}
	For all graphs $G=(V,(c_{xy})_{x,y\in V})$,  and site weights $\alpha=(\alpha_x)_{x\in V}$, $\omega=(\omega_x)_{x\in V}$ with $\omega\neq 0$, we have
	\begin{equation}
		{\rm gap}_{\rm SIP}(G,\alpha,\omega,\varrho) = {\rm gap}_{\rm RW}(G,\alpha,\omega)\comma \qquad \varrho > 0\fstop
	\end{equation}
\end{corollary}

\fs{The above result is stated in the reversible setting $\theta\equiv \varrho$. The next two remarks clarify the corresponding non-reversible picture: for the infinite-state $\mathrm{SIP}$ one has to distinguish algebraic eigenfunction lifting from a full spectral statement on a prescribed function space (Remark \ref{rem:non-rev}), whereas in finite-state analogues such as the partial symmetric exclusion process this distinction disappears (Remark \ref{rem:SEP}).}	

\begin{remark}[Non-reversible setting]
	\label{rem:non-rev} In view of Theorem \ref{th:lifting-general}, it is plausible to expect  that Corollary \ref{cor:non-conservative} should extend, in some sense, to the non-reversible setting as well.	Indeed, since the two functions on the right-hand side of \eqref{eq:eigen-lift-2} are polynomials of order $k$ and $\ell \le k-1$, this guarantees that $\lambda$ therein lies in the point spectrum of $L_{G,\alpha,\omega,\theta}$, for all $\theta=(\theta_x)_{x\in V}$. However, in absence of a natural Hilbert space ensuring that such eigenfunctions form a complete orthonormal system  when $\theta\not\equiv{\rm const.}$ on ${\rm supp}(\omega)$, the presence, e.g., of a continuous part of the spectrum (a property which also depends on the choice of the underlying reference function space) is what  Theorem \ref{th:lifting-general} does not exclude.	
	
\end{remark}
\begin{remark}[Symmetric exclusion processes \& Co.]\label{rem:SEP}The proof of Theorem \ref{th:lifting-general}\ref{it:lifting-general} --- based on orthogonal-polynomial dualities, annihilation and creation operators --- carries over to  other non-equilibrium systems enjoying similar properties, e.g., \cite{redig_factorized_2018,floreani_boundary2020,giardina_redig_tol_intertwining_2024}. As an example, we mention the symmetric (partial) exclusion process (${\rm SEP}$) in contact with reservoirs (see, e.g., \cite{floreani_boundary2020,salez2022universality,salez_spectral_2024}), whose generator is readily obtained from $L_{G,\alpha,\omega,\theta}$ in \eqref{eq:gen-res} by further imposing 	 $\alpha=(\alpha_x)_{x\in V}\subset \N$ and $\theta=(\theta_x)_{x\in V}\in [0,1]$, and changing the plus signs in $\alpha_x+\eta_x$ and $1+\theta_x$ therein into minus signs. 
	In this case, the configuration space is simply
	$\prod_{x\in V}\{0,1,\ldots,\alpha_x\}$, thus finite, and the technical issues explained in
	Remark~\ref{rem:non-rev} do not arise: the spectrum is just the union of finitely many
	eigenvalues. 
	
\fs{As our proof arguments
	used for ${\rm SIP}$ adapt to the open ${\rm SEP}$, one obtains analogues of Theorems \ref{th:non-conservative}
	and \ref{th:lifting-general}: on the one hand, the eigenvalues of the open ${\rm SEP}$ with arbitrary
	reservoir densities $\theta\in[0,1]^V$ coincide with those of the purely absorbing case
	$\theta\equiv0$, and are therefore independent of $\theta$; on the other hand, importing the notation used
	above, one has ${\rm gap}_{\rm SEP}(G,\alpha,\omega)={\rm gap}_{\rm RW}(G,\alpha,\omega)$. These two results further ensure that
	\begin{equation}\label{eq:SEP}
		{\rm gap}_{\rm SEP}(G,\alpha,\omega,\theta)
		=
		{\rm gap}_{\rm RW}(G,\alpha,\omega)\comma
		\qquad \theta\in[0,1]^V \fstop
	\end{equation}
This, thus, recovers, in the reversible 
	case, the spectral-gap identity of \cite[Corollary 1]{salez2022universality}, and, for
	general reservoir densities, the corresponding result of \cite[Theorem 2]{salez_spectral_2024}.}
\end{remark}

\section{Proofs of Theorems \ref{th:non-conservative} and \ref{th:lifting-general}}\label{sec:proof-res}
The following two sections may be essentially read independently.
\subsection{Proof of Theorem \ref{th:non-conservative}}
\label{sec:proof-gap-abs}
For every $k\ge 1$,  the $k$-particle purely absorbing system (i.e., when $\theta\equiv0$)	is described by the infinitesimal generator 
\begin{equation}\label{eq:gen-res-abs}
	L_{G,\alpha,\omega,k}f(\eta)\eqdef L_{G,\alpha,k}f(\eta)	- f(\eta)\sum_{x\in V}\omega_x\, \eta_x \comma
\end{equation} 
where $f\in \R^{\Xi_k}$, $\eta\in \Xi_k$, and $L_{G,\alpha,k}$  given in \eqref{eq:gen-conservative}. \fs{Equivalently, if $f\in\R^{\Xi_k}$ is extended to $\Xi$ by setting $f\equiv0$ on $\Xi\setminus\Xi_k$, then the restriction to $\Xi_k$ of the generator in \eqref{eq:gen-res} with $\theta\equiv0$ is precisely the operator in \eqref{eq:gen-res-abs}.} Note that, when  the process is in configuration $\eta\in \Xi_k$, it gets killed with rate $\sum_{x\in V}\omega_x\,\eta_x\ge 0$. 

Recall the definition of the probability measure $\mu_{\alpha,k}$ in \eqref{mu-def}. Since $L_{G,\alpha,\omega,k}$ is the sum of a self-adjoint operator on $L^2(\mu_{\alpha,k})$ and a multiplicative one,  it is self-adjoint on $L^2(\mu_{\alpha,k})$ and, thus, all the eigenvalues of $-L_{G,\alpha,\omega,k}$ are real and non-negative. Additionally, since the conservative dynamics is irreducible and the multiplication operator is non-positive, all such eigenvalues are actually strictly positive. In particular, we obtain that ${\rm gap}_k(G,\alpha,\omega)$, the lowest of such eigenvalues, is positive and, by the Perron-Frobenius theorem, characterized as the worst-case exponential killing rate for large times: letting $\tau_k$ denote the first time at which  one of the $k$ particles gets killed, 
\begin{equation}
	{\rm gap}_k(G,\alpha,\omega)= - \lim_{t\to \infty}\frac1t\log \max_{\eta\in \Xi_k} \P_\eta^{G,\alpha,\omega,k}[\tau_k>t] \fstop
\end{equation}
Thanks to this asymptotic identity, the assertion of Theorem \ref{th:non-conservative} boils down to show that, for all $k\ge 2$,	
\begin{equation}\label{eq:final-boundary}
\max_{\eta\in \Xi_k}\P_\eta^{G,\alpha,\omega,k}[\tau_k>t]\le \max_{x\in V} \P_{\delta_x}^{G,\alpha,\omega,1}[\tau_1>t]\comma\qquad t\ge 0\fstop
\end{equation}

For this purpose, we introduce a lookdown representation of the particle system with killing, extending the one known within the conservative setting  \cite[Appendix A]{kim_sau_spectral_2023}.
We need to introduce a bit of  notation. For every $k\ge 2$, labeled configuration $\mathbf x=(x_1,\ldots, x_k)\in V^k$, and permutation $\varsigma \in S_k$, we write $\varsigma \mathbf x= (x_{\varsigma(1)},\ldots,x_{\varsigma(k)})\in V^k$. Further, $\cS_k: \R^{V^k}\to \R^{V^k}$ denotes the symmetrization operator, that is, 
\begin{equation}
	\cS_k \varphi(\mathbf x)= \frac1{k!}\sum_{\varsigma \in S_k}\varphi(\varsigma\mathbf x)\comma\qquad \varphi \in \R^{V^k}\comma
\end{equation}
whereas $\Phi_k(\mathbf x)=\sum_{i=1}^k \delta_{x_i}\in \Xi_k$ is the operator which removes the particles' labels.

\begin{proposition}[Lookdown representation of ${\rm SIP}$ with killing]For all $k\ge 2$, we have
	\begin{equation}\label{eq:lookdown-sip}
		\cS_k  {\mathscr L}_{G,\alpha,\omega,k}(f\circ \Phi_k)= (L_{G,\alpha,\omega,k}f)\circ \Phi_k\comma\qquad f \in \R^{\Xi_k}\comma
	\end{equation}
	where ${\mathscr L}_{G,\alpha,\omega,k}$ denotes the generator of the lookdown ${\rm SIP}_k(G,\alpha,\omega)$ given, for all $\varphi \in \R^{V^k}$ and $\mathbf x\in V^k$, by 
	\begin{equation}\label{eq:gen-lookdown}
		{\mathscr L}_{G,\alpha,\omega,k}\varphi(\mathbf x)= {\mathscr L}_{G,\alpha,k}\varphi(\mathbf x)
			 - \varphi(\mathbf x)\sum_{i=1}^k \omega_{x_i}\fstop
	\end{equation}	
	Here, ${\mathscr L}_{G,\alpha,k}\varphi(\mathbf x)$ reads as
	\begin{equation}
	{\mathscr L}_{G,\alpha,k}\varphi(\mathbf x) = \sum_{i=1}^k \sum_{x,y\in V}c_{xy}\,\delta_{x,x_i}\tonde{\alpha_y+2\sum_{j=1}^{i-1}\delta_{y,x_j}}\tonde{\varphi(\mathbf x_i^y)-\varphi(\mathbf x)}\comma
	\end{equation}and stands for the conservative part of the lookdown dynamics, with $\mathbf x_i^y\in V^k$ denoting the configuration obtained from $\mathbf x\in V^k$ by setting the $i$-th coordinate equal to $y\in V$.
\end{proposition}
\begin{proof}
By \cite[Proposition A.4]{kim_sau_spectral_2023}, an identity which is analogous to that in \eqref{eq:lookdown-sip} holds for the conservative dynamics, i.e., 
\begin{equation}
	\cS_k{\mathscr L}_{G,\alpha,k}(f\circ \Phi_k)= (L_{G,\alpha,k}f)\circ \Phi_k\comma\qquad f \in \R^{\Xi_k}\fstop
\end{equation}
Hence, by linearity, it suffices to prove the identity only for the multiplication operators in \eqref{eq:gen-res-abs} and \eqref{eq:gen-lookdown} describing the killing part. Nevertheless, this identity holds trivially because particles are killed independently (or, equivalently, because $\mathbf x=(x_1,\ldots, x_k)\mapsto (f\circ \Phi_k)(\mathbf x)\sum_{i=1}^k \omega_{x_i}$ is invariant under permutations).
\end{proof}

\fs{We now conclude the proof of Theorem \ref{th:non-conservative}.}
\begin{proof}[Proof of Theorem \ref{th:non-conservative}]
Let $\cP_t^{G,\alpha,\omega,k} = e^{t{\mathscr L}_{G,\alpha,\omega,k}}$, $t\ge0$, be the lookdown semigroup. In conclusion, taking $f=\car_{\Xi_k}\in \R^{\Xi_k}$ in \eqref{eq:lookdown-sip} and observing that $\car_{\Xi_k}\circ \Phi_k= (\otimes_{i=1}^k\,\car_V)\in \R^{V^k}$, we obtain, for all $\mathbf x\in V^k$ and $t> 0$, 
\begin{align}
\max_{\eta\in \Xi_k}\P_\eta^{G,\alpha,\omega,k}[\tau_k>t]&\le \max_{\mathbf x\in V^k} \cP_t^{G,\alpha,\omega,k}(\otimes_{i=1}^k\,\car_V)(\mathbf x)\\
&\le \max_{x\in V}\cP_t^{G,\alpha,\omega,1}\car_V(x)\comma
\end{align}
where the second inequality follows from neglecting all particles with labels $\ell=2,\ldots, k$ (in other words,  
bounding $\otimes_{i=2}^k\,\car_V\le 1$), the first being independent from them. Since this last expression equals the right-hand side of \eqref{eq:final-boundary}, this proves the inequality therein and, thus, concludes the proof of the theorem.
\end{proof}

\subsection{Proof of Theorem \ref{th:lifting-general}}\label{sec:proof-th-eigen}
Our main ingredient will be the duality property of ${\rm SIP}$ in contact with reservoirs (e.g., \cite{giardina_duality_2009,carinci_duality_2013-1}), in particular, the orthogonal polynomial duality functions found in \cite{floreani_boundary2020}. We shall not assume any prior knowledge on this topic, but rather exploit a few properties taken from \cite{koekoek_hyper_2010,redig_factorized_2018,floreani_boundary2020}, which, for the reader's convenience, we collect in the following proposition. For notational convenience, we simply write
 $L_\theta=L_{G,\alpha,\omega,\theta}$ and $L_k=L_{G,\alpha,\omega,k}$ (defined in \eqref{eq:gen-res} and \eqref{eq:gen-res-abs}, respectively).
\begin{proposition}\label{pr:duality} Fix  $\varrho > 0$.
	\begin{enumerate}[(a)
		]
		\item \emph{Orthogonality and completeness.} Let $\scalar{\emparg}{\emparg}_{\alpha,\varrho}$  denote the inner product in $L^2(\nu_{\alpha,\varrho})$ (see  \eqref{eq:nu-sigma}). Then, for all integers $k,\ell\ge 1$ and $\xi\in \Xi_k$, $\zeta\in \Xi_\ell$,  
		\begin{align}
			\begin{aligned}\label{eq:norm-orthogonal}
				\scalar{D_{\alpha,\varrho}(\xi,\emparg)}{D_{\alpha,\varrho}(\zeta,\emparg)}_{\alpha,\varrho}&= {\car_{\{\xi=\zeta\}}}\,\frac{\varrho^k\tonde{1+\varrho}^k}{Z_{\alpha,k}}\,{\mu_{\alpha,k}(\xi)^{-1}}
			\end{aligned}\fstop
		\end{align}
		Moreover,  $\set{D_{\alpha,\varrho}(\xi,\emparg)}_{k\ge0,\,\xi\in \Xi_k}$ is a complete system of $L^2(\nu_{\alpha,\varrho})$.

		\item \emph{Duality.} For all  $\theta=(\theta_x)_{x\in V}$  and $\eta\in \Xi$, we have
		\begin{align}\label{eq:duality-rel}
			\begin{aligned}
				L_\theta D_{\alpha,\varrho}(\xi,\emparg)(\eta)
				&= L_k D_{\alpha,\varrho}(\emparg,\eta)(\xi) \\
				&+ \sum_{x\in V}\omega_x\,\xi_x\tonde{\theta_x-\varrho} D_{\alpha,\varrho}(\xi-\delta_x,\eta) \fstop
			\end{aligned}
		\end{align}
	
	\end{enumerate}	
\end{proposition}
\begin{remark}\label{rem:poly}
Since $\set{D_{\alpha,\varrho}(\xi,\emparg)}_{\ell\le k,\,\xi\in \Xi_\ell}$ spans all polynomials of order $\le k$, by \eqref{eq:duality-rel},  $L_\theta=L_{G,\alpha,\omega,\theta}$ leaves invariant  such  a subspace of polynomials.
\end{remark}
\begin{proof}[Proof of Theorem \ref{th:lifting-general}]
The rest of the proof of Theorem \ref{th:lifting-general} is routine. Indeed, Theorem \ref{th:lifting-general}\ref{it:lifting-general} follows by linearity and the duality relation in \eqref{eq:duality-rel}: recalling $F_{\alpha,\varrho}^\psi$ from Definition \ref{def:polynomials},  first we obtain
\begin{align}
	L_\theta F_{\alpha,\varrho}^\psi(\eta)&= \sum_{\xi\in \Xi_k} \mu_{\alpha,k}(\xi)\, \psi(\xi)\, L_\theta D_{\alpha,\varrho}(\xi,\emparg)(\eta)\\
		&= \sum_{\xi\in \Xi_k}\mu_{\alpha,k}(\xi)\,\psi(\xi)\, L_k D_{\alpha,\varrho}(\emparg,\eta)(\xi)\\
		&+\sum_{\xi\in \Xi_k}\mu_{\alpha,k}(\xi)\,\psi(\xi) \sum_{x\in V}\omega_x\,\xi_x\tonde{\theta_x-\varrho} D_{\alpha,\varrho}(\xi-\delta_x,\eta)\comma
		\end{align}
		which, \fs{by \eqref{eq:basic-id} and the definition of $\mathfrak b_{\alpha,\omega,\varrho,k-1}^\varrho$ in	 \eqref{eq:mathfrak-b}}, rearranges as 
		\begin{align}
		&\sum_{\xi\in \Xi_k}\mu_{\alpha,k}(\xi)\,\psi(\xi)\, L_k D_{\alpha,\varrho}(\emparg,\eta)(\xi)+F_{\alpha,\varrho}^{\mathfrak b_{\alpha,\omega,\theta,k-1}^\varrho\psi}\fstop
		\end{align}
		By the symmetry of $L_k$ in $L^2(\mu_{\alpha,k})$, we finally obtain
		\begin{align}
	L_\theta F_{\alpha,\varrho}^\psi(\eta)=\sum_{\xi\in \Xi_k}\mu_{\alpha,k}(\xi)\, L_k\psi(\xi)\,  D_{\alpha,\varrho}(\xi,\eta)+F_{\alpha,\varrho}^{\mathfrak b_{\alpha,\omega,\theta,k-1}^\varrho\psi}\comma
\end{align}
and the desired claim follows if $L_k\psi=-\lambda\,\psi$.

	 The second item in Theorem \ref{th:lifting-general}  is a direct consequence of the symmetry of $L_k$ in $L^2(\mu_{\alpha,k})$ (thus, the orthogonality relation of eigenfunctions $\psi.\varphi\in \R^{\Xi_l}$ associated to distinct eigenvalues) and the orthogonality relation in \eqref{eq:norm-orthogonal}: for all integers $k,\ell\ge0$, all eigenfunctions $\psi\in \R^{\Xi_k}$ and $\varphi\in \R^{\Xi_\ell}$ of $L_k$ and $L_\ell$, respectively, we have  
	 \begin{align}
	 	\tscalar{F_{\alpha,\varrho}^\psi}{F_{\alpha,\varrho}^\varphi}_{\alpha,\varrho}= \car_{\set{k=\ell}}\,\mathfrak c_{\alpha,\varrho,k} \scalar{\psi}{\varphi}_{\alpha,k}\comma
	 \end{align}
	 for some constant $\mathfrak c_{\alpha,\varrho,k}>0$.

\fs{The completeness of the functions in \eqref{eq:meixner-duality} follows from a
	one-dimensional density result. Each marginal of $\nu_{\alpha,\varrho}$ is a
	negative-binomial measure, hence has exponential tails and a determinate
	Hamburger moment problem. By the classical theorem of M.\ Riesz \cite{riesz1924probleme} (see also \cite[Chapter 2, Section 3 and Corollary 2.3.3]{akhiezer_classical_1965},	 \cite{berg_density_1981,berg_moment_1996}, and references therein), determinacy
	implies density of algebraic polynomials in the corresponding $L^2$-space. Since
	$V$ is finite and $\nu_{\alpha,\varrho}$ is a product measure on $\mathbb N_0^V$,
	tensor products of these univariate polynomial bases are dense in
	$L^2(\nu_{\alpha,\varrho})$, yielding the desired completeness.}
	\end{proof}

\section{Extensions and related models}\label{sec:extension}
In this last section, we extend  our results to the Brownian energy process, and briefly discuss their connections with what is known for some closely related models. 
\subsection{Brownian energy process}\label{sec:BEP}
Given a graph $G=(V,(c_{xy})_{x,y\in V})$ and site weights $\alpha=(\alpha_x)_{x\in V}$, the corresponding Brownian energy process --- referred to as ${\rm BEP}(G,\alpha)$ --- is the diffusion on the probability simplex 
\begin{equation}\label{eq:simplex}
	\varSigma\eqdef \set{\zeta \in [0,1]^V: \sum_{x\in V}\zeta_x=1}\comma
\end{equation}
 with infinitesimal generator
\begin{equation}
	\cL_{G,\alpha}= \frac12\sum_{x,y\in V}c_{xy}\set{-\left(\alpha_y \zeta_x-\alpha_x\zeta_y\right)\left(\partial_{\zeta_x}-\partial_{\zeta_y}\right)+\zeta_x\zeta_y\left(\partial_{\zeta_x}-\partial_{\zeta_y}\right)^2}\comma
\end{equation}
and the following Dirichlet distribution on $\varSigma$ as its unique reversible measure
\begin{equation}\label{eq:dirichlet-distribution}
	\gamma_\alpha(\dd\zeta)= \tonde{\frac{1}{{\rm B}(\alpha)}\prod_{x\in V} \zeta_x^{\alpha_x-1}}\dd \zeta\comma\qquad \text{with}\  {\rm B}(\alpha)\eqdef \frac{\prod_{x\in V}\Gamma(\alpha_x)}{\Gamma(|\alpha|)}\fstop
\end{equation}
Here, $\dd \zeta$ stands for the uniform measure on $\varSigma$. We refer to Section \ref{sec1}, \cite{kim_sau_spectral_2023} and references therein, for more information and history on this model. Here, we just mention that its mean field version (i.e., when $c_{xy}\equiv {\rm const.}$) coincides with the  renowned multi-type Wright-Fisher diffusion with parent independent mutations (see, e.g., \cite{shimakura_equations1977,wang_zhang_nash_2019}).  

As shown in \cite{kim_sau_spectral_2023},   ${\rm BEP}(G,\alpha)$ inherits its eigenstructure from  that of ${\rm SIP}(G,\alpha)$. 	More precisely, the Poisson intertwining relation \cite[Eq.\ (4.1)]{kim_sau_spectral_2023} between the generators $\cL_{G,\alpha}$ and $L_{G,\alpha,k}$ allows to lift eigenfunctions of ${\rm SIP}(G,\alpha)$ to eigenfunctions of ${\rm BEP}(G,\alpha)$, both corresponding to the same eigenvalue. Using this intertwining relation \fs{(and, more specifically, that $\cL_{G,\alpha}$ maps the finite-dimensional subspace of polynomials of degree at most $k$ into itself)} and the fact $\cL_{G,\alpha}$ is a self-adjoint operator on $L^2(\gamma_\alpha)$, one concludes that this lifting saturates all \fs{the spectral information (i.e., the spectrum of ${\rm BEP}$ is pure point and coincides with that of ${\rm SIP}$; see \cite[Sections 1.2 and 4]{kim_sau_spectral_2023} for further details)}. In particular, the corresponding spectral gaps coincide
\begin{equation}
	{\rm gap}_{\rm BEP}(G,\alpha)={\rm gap}_{\rm SIP}(G,\alpha)\fstop
\end{equation}

Combining the findings in \cite{kim_sau_spectral_2023} with our Theorems \ref{th:sharp-alpha-min} and \ref{th:failure-gap} on ${\rm gap}_{\rm SIP}(G,\alpha)$, we readily deduce the following analogues of \eqref{eq:failure-gap} and \eqref{eq:lb-alpha-min-intro}.
\begin{corollary}\label{cor:BEP}
	For all graphs $G=(V,(c_{xy})_{x,y\in V})$ and site weights $\alpha=(\alpha_x)_{x\in V}$, we have
	\begin{equation}
	 \alpha_{\rm min}\set{	\frac1{21} \, \frac{ \alpha_{\rm ratio}}{ 6^{\alpha_{\rm min}/\alpha_{\rm ratio}} } \, \frac{ c_{\rm min}}{|V|^2\, {\rm diam}(G) }}\le {\rm gap}_{\rm BEP}(G,\alpha)\le {\rm gap}_{\rm RW}(G,\alpha)\comma
	\end{equation}
	and, in some cases, we have	
	\begin{equation}
		{\rm gap}_{\rm BEP}(G,\alpha)< {\rm gap}_{\rm RW}(G,\alpha)\fstop
	\end{equation}
	\end{corollary}

Interestingly, also our results on the spectrum of the non-conservative ${\rm SIP}$, via similar arguments, translate to the context of the ${\rm BEP}$. Indeed, the aforementioned Poisson intertwining is effective also when one adds the interaction with external reservoirs. As shown, e.g., in \cite[Section 4.1]{giardina_redig_tol_intertwining_2024}, using this operator, the generator $L_{G,\alpha,\omega,\theta}$ in \eqref{eq:gen-res} of ${\rm SIP}(G,\alpha,\omega,\theta)$ satisfies, for all $\omega=(\omega_x)_{x\in V}$ and $\theta=(\theta_x)_{x\in V}$, an analogous intertwining relation with 
\begin{equation}
	\cL_{G,\alpha,\omega,\theta} = \cL_{G,\alpha} + \sum_{x\in V}\omega_x\set{\tonde{\alpha_x\theta_x-\zeta_x}\partial_{\zeta_x} + \theta_x\zeta_x\,\partial_{\zeta_x}^2}\comma
\end{equation}
the generator of what we refer to as ${\rm BEP}(G,\alpha,\omega,\theta)$.

When $\theta\equiv \varrho > 0$, this non-conservative diffusion on $[0,\infty)^V$ is well known to have the following product measure $\gamma_{\alpha,\varrho} = \otimes_{x\in V}\, {\rm Gamma}(\alpha_x,\varrho)$ as its unique reversible measure. Hence, arguing as above, we derive the corresponding spectral gap identity
\begin{equation}
	{\rm gap}_{\rm BEP}(G,\alpha,\omega,\varrho)={\rm gap}_{\rm SIP}(G,\alpha,\omega,\varrho)\comma
\end{equation}
and, as a consequence, the following analogue of Corollary \ref{cor:non-conservative}.
\begin{corollary}
		For all graphs $G=(V,(c_{xy})_{x,y\in V})$,  and site weights $\alpha=(\alpha_x)_{x\in V}$, $\omega=(\omega_x)_{x\in V}$ with $\omega\neq 0$, we have
	\begin{equation}
		{\rm gap}_{\rm BEP}(G,\alpha,\omega,\varrho) = {\rm gap}_{\rm RW}(G,\alpha,\omega)\comma \qquad \varrho > 0\comma
	\end{equation}
	where ${\rm gap}_{\rm RW}(G,\alpha,\omega)$ is defined in \eqref{eq:gap-rw-abs}.
\end{corollary}

\subsection{Other models}\label{sec:extensions}
As already anticipated in Section \ref{sec1}, ${\rm SIP}$ and  ${\rm BEP}$ come with several closely related models,  having a different dynamics, but exactly the same steady states. The models we have in mind fall into the class of mass redistribution models. Here, we shall focus only on those models with configuration-independent edge updates (see, e.g., \cite{sasada_spectral_2015} and references therein for configuration-dependent mass redistribution models with Dirichlet-type steady states): after an edge, say $xy$, is selected	 with rate $c_{xy}\ge0$,   particles/energies sitting on $x$ and $y$ get randomly redistributed therein. Among them, let us mention:
\begin{enumerate}[(a)]
	\item the Beta-Binomial splitting process and its continuous version, also known as Beta splitting process or (generalized) KMP model (see, e.g., \cite{pymar2023mixing,demasi_ferrari_gabrielli_hidden_2024,giardina_redig_tol_intertwining_2024});
	\item the  discrete and continuous harmonic processes (see, e.g., \cite{franceschini_frassek_giardina_integrable_2023,giardina_redig_tol_intertwining_2024}). 
\end{enumerate}
All models in (a) and (b) have been studied in their conservative and non-conservative variants (see the aforementioned references for the precise definitions of the redistribution rules), and, just like  ${\rm SIP}$ and ${\rm BEP}$,  discrete and corresponding continuous models are isospectral (when reversible), enjoying completely analogous  Poisson  intertwining relations.

To the best of our knowledge, the only sharp quantitative results on their convergence to equilibrium are confined to a handful of results, see \cite{carlen_carvalho_loss_determination_2003,smith_gibbs_2014,caputo_mixing_2019,pymar2023mixing} and references therein. For what concerns   the spectral gap, the only sharp results that we are aware of concern the Beta and Beta-Binomial splitting processes  in two specific settings, both conservative:
\begin{enumerate}[(i)]
	\item  the complete graph with $\alpha\equiv {\rm const.}$ \cite{carlen_carvalho_loss_determination_2003} (see also \cite[Remark 4]{caputo_mixing_2019});
	\item the segment with $\alpha\equiv{\rm const.}\ge 1$ \cite{caputo_mixing_2019}.
	\end{enumerate}
	In particular, in (i), the gap is attained by the system with $k=2$ particles (which is strictly smaller than ${\rm gap}_1$), while in (ii) it coincides with that of the corresponding single particle.

In order to obtain  spectral gap asymptotics for small diffusivity, the   proof ideas of Theorems \ref{th:gap-asymptotics1} and \ref{th:failure-gap}    apply also to these models (actually, with considerable simplifications due to the presence of edge updates only, and edges are a.s.\ never updated simultaneously). For instance, it is not difficult to show that, on  graphs like that in Section \ref{sec5.5} and for    $\alpha$ small enough, the spectral gap of the two particles must be strictly smaller than that of the single random walk.
	Similarly, the non-conservative (reversible) models satisfy a one-particle spectral gap identity, analogous to the one contained in Corollary \ref{cor:non-conservative}. Indeed,  the key step in that proof (Section \ref{sec:proof-gap-abs}) carries over to these other models in contact with reservoirs, because of the analogous intertwining and consistency/duality relations available also for them, see, e.g., \cite{giardina_redig_tol_intertwining_2024}. We leave the verification of these claims  to the interested reader. What remains largely open is the problem of obtaining spectral gap's sharp lower bounds  --- and, possibly, identities in terms of simpler processes --- for the conservative variants of these mass redistribution models on general geometries.

\appendix

\section{Being in $\Omega_k$ at each positive time}\label{appenA}
In this appendix, we prove a technical result on the probability that the $\eps^{-1}$-accelerated $\rm SIP$ sits in the absorbing set $\Omega_k$ (see \eqref{Omegak-def}) at a given positive time. Since  $\Omega_k = \Xi_k$ when $k=1$, we  fix  $k\ge 2$ all throughout this appendix. 

We use the notation employed in Section \ref{sec5}.
Recall from the proof of Proposition \ref{lem:kurtz} that $p_t^\eps(\cdot,\cdot)$ denotes the transition kernel of the $\eps^{-1}$-sped up $\rm SIP(G,\eps\hat\alpha)$ with $k$ particles, i.e., 
\begin{equation}
	p_t^\eps(\eta,\xi)= e^{t\eps^{-1}L_{\eps\hat\alpha,k}}\car_\xi(\eta)\comma \qquad \eta,\xi\in \Xi_k\fstop
\end{equation} The following proposition is the main result of this appendix.

\begin{proposition}\label{lem:fin-dim-conv}
For any configuration $\eta\in\Xi_k$ and  time $t>0$, we have
\begin{equation}
\lim_{\eps \to 0} p_t^\eps (\eta,\Omega_k) = 1 \fstop
\end{equation}
\end{proposition}

To prove Proposition \ref{lem:fin-dim-conv}, we need a series of technical lemmas. Let $(\eta^\eps(t))_{t\ge 0}$
 be the $\eps^{-1}$-sped up $\rm SIP$ with generator $\eps^{-1}L_{\eps\hat\alpha,k}$. Let $\P_\eta^\eps$ and $\E_\eta^\eps$ denote the law and 	 corresponding expectation of $(\eta^\eps(t))_{t\ge0}$ starting from $\eta\in\Xi_k$.

The first lemma states that the time to reach the absorbing set $\Omega_k$ from the transient collection $\Delta_k$ (defined in \eqref{Deltak-def}) is uniformly negligible. Given a subset $A$ of $\Xi_k$, let $\tau_A$ be its (random) first hitting time. Moreover, define
\begin{equation}\label{cmin-cmax-def}
c_{\rm min} = \min_{\substack{x,y\in V\\c_{xy}>0}} c_{xy}> 0 \comma \qquad c_{\rm max} = \max_{x,y\in V}c_{xy}  \fstop
\end{equation}

\begin{lemma}\label{lem1}
Suppose that $\eps|\hat\alpha| \le \frac{c_{\rm min}}{4\,c_{\rm max}\,k}$. Then,  for some \fs{$C_1=C_1(G)>0$} and for all $\eps>0$, $k\ge 2$,
\begin{equation}
\max_{\zeta \in \Delta_k} \E_\zeta^\eps \,[ \tau_{\Omega_k}] \le \frac{\fs{C_1}\eps}{c_{\rm min}} \, \log k \fstop
\end{equation}
\end{lemma}
\begin{proof}
Define a test function $\phi : \Xi_k \to \R$ as
\begin{equation}
\phi(\eta) \eqdef - \sum_{x\in V } \left( 1+\frac12 + \cdots + \frac1{\eta_x}  \right) \fstop
\end{equation}
Then, there exists \fs{$C_1=C_1(G)>0$} such that
\begin{equation}\label{eqeq0.5}
\max_{\eta \in \Xi_k} \phi(\eta) - \min_{\eta \in \Xi_k} \phi(\eta) \le \frac{\fs{C_1}}4 \log k \fstop
\end{equation}
For each $\zeta \in \Delta_k$, we may express $L_{\eps\hat\alpha,k} \phi(\zeta)$ as
\begin{equation}\label{eqeq1}
\sum_{ \{x,y \} \subseteq V} c_{xy}  \left[ \zeta_x \,(\eps\hat\alpha_y +\zeta_y ) \, ( \phi(\zeta-\delta_x +\delta_y ) - \phi(\zeta)) + \zeta_y \, (\eps\hat\alpha_x + \eta_x ) \,( \phi (\zeta-\delta_y + \delta_x ) - \phi(\zeta))    \right]  \fstop
\end{equation}
If $\zeta_x , \zeta_y \ge1$, then the term inside the bracket in \eqref{eqeq1} becomes
\begin{equation}
\frac{\zeta_x}{\zeta_y+1} + \frac{\zeta_y}{\zeta_x+1} + \eps\hat\alpha_y \left( 1 -  \frac{\zeta_x}{\zeta_y+1}  \right)  + \eps\hat\alpha_x \left( 1 -  \frac{\zeta_y}{\zeta_x+1}  \right)  \ge  \frac{ (1-\eps\hat\alpha_y) \,\zeta_x}{\zeta_y+1} +  \frac{ (1-\eps\hat\alpha_x) \, \zeta_y}{\zeta_x+1} \fstop
\end{equation}
By the assumption, we have $\eps\hat\alpha_{\rm max} \le \eps|\hat\alpha| \le \frac{c_{\rm min}}{4 \, c_{\rm max} \, k} < \frac12$. Thus, the right-hand side is bounded from below by
\begin{equation}\label{eqeq2}
\frac12 \left( \frac{ \zeta_x}{\zeta_y+1} +  \frac{  \zeta_y}{\zeta_x+1}   \right) \ge \frac12 \comma
\end{equation}
where the inequality holds because $\zeta_x ,\zeta_y \ge 1$. Next, if $\zeta_x\ge1$ and $\zeta_y=0$, then the term inside the bracket in \eqref{eqeq1} equals
\begin{equation}\label{eqeq3}
\zeta_x \,\eps\hat\alpha_y  \left( \frac1{\zeta_x} - 1 \right) = - \eps\hat\alpha_y  \, (\zeta_x -1) \ge -\eps\hat\alpha_y \, \zeta_x  \fstop
\end{equation}
Similarly, if $\zeta_x=0$ and $\zeta_y\ge1$, then the term is at least $-\eps\hat\alpha_x \, \zeta_y$. As $\zeta \in \Delta_k$, there exists at least one $\{x^*,y^*\}$ such that $c_{x^*y^*}>0$ and $\zeta_{x^*},\zeta_{y^*} \ge 1$. Thus, by \eqref{cmin-cmax-def}, \eqref{eqeq1}, \eqref{eqeq2}, and \eqref{eqeq3},
\begin{equation}\label{eqeq4}
L_{\eps\hat\alpha,k}\phi (\zeta) \ge \frac{c_{x^*y^*}}2 - \sum_{\substack{x\in V \\ \zeta_x \ge 1}} \sum_{\substack{y\in V \\ \zeta_y = 0}} c_{xy} \,\eps\hat\alpha_y \,\zeta_x  \ge \frac{c_{\rm min}}{2} - c_{\rm max} \,k\, \eps|\hat\alpha| \ge \frac{c_{\rm min}}{4} \comma
\end{equation}
where the last inequality holds since $\eps|\hat\alpha| \le \frac{c_{\rm min}}{4 \, c_{\rm max} \, k}$.

Now, it is well known that
\begin{equation}
\phi(\eta^\eps(t)) - \phi(\eta^\eps(0)) - \int_0^t ( \eps^{-1} L_{\eps\hat\alpha,k} \phi ) (\eta^\eps(s)) \, \dd s  \comma \quad t\ge0 \comma
\end{equation}
is a $\P_\zeta^\eps$-martingale. Setting $t=\tau_{\Omega_k} \wedge r $, applying \eqref{eqeq0.5} and \eqref{eqeq4}, taking expectations with respect to $\E_\zeta^\eps$, and then sending $r \to \infty$, we obtain 
\begin{equation}
\frac{c_{\rm min} \, \eps^{-1}}{4} \,\E_\zeta^\eps \, [\tau_{\Omega_k}] \le \E_\zeta^\eps \left[ \int_0^{\tau_{\Omega_k}} (\eps^{-1} L_{\eps\hat\alpha,k} \phi)(\eta^\eps(s)) \, \dd s  \right] \le \frac{\fs{C_1}}4 \log k \fstop
\end{equation}
This concludes the proof of the lemma.
\end{proof}

For each $t\ge0$, let $T_\eps(t)$ be the local time spent by $(\eta^\eps(s))_{s\ge0}$ in $\Omega_k$ up to time $t$:
\begin{equation}\label{Talpha-def}
T_\eps(t) \eqdef \int_0^t \IND \{ \eta^\eps(s) \in \Omega_k \} \, \dd s \fstop
\end{equation}
Then, let $S_\eps(s)$, $s\ge0$, be its generalized inverse:
\begin{equation}\label{Salpha-def}
S_\eps(s) \eqdef \sup \, \{ t \ge 0 : T_\eps(t) \le s \} \fstop
\end{equation}
By definition, $T_\eps(S_\eps(s)) = s$ and $T_\eps(S_\eps(s)+u)>s$ for any $s\ge0$ and $u>0$. In turn, we have
\begin{equation}\label{Salpha-in-Ak}
\eta^\eps (S_\eps(s)) \in \Omega_k\comma\qquad s \ge 0  \fstop
\end{equation}
Note also that $\{ S_\eps(s) \ge s+u \} = \{ T_\eps(s+u) \le s \}$ for $s\ge0$ and $u>0$.

\begin{lemma}\label{lem2}
For all $\eta\in\Xi_k$, $t\ge0$ and $\delta >0$, we have
\begin{equation}
\lim_{\eps \to 0} \P_\eta^\eps \, [ S_\eps(t) \ge t+ \delta ]   = 0 \fstop
\end{equation}
\end{lemma}
\begin{proof}
First, suppose that $\eta \in \Delta_k$.  We may bound
\begin{align}
\P_\eta^\eps \, [ S_\eps(t) \ge t+\delta ] & = \P_\eta^\eps \, [ T_\eps(t+\delta) \le t ]  \\
& \le \P_\eta^\eps \, [ T_\eps(t+\delta) \le t  , \ \tau_{\Omega_k} < \delta/2 ] + \P_\eta^\eps  \, [ \tau_{\Omega_k} \ge \delta/2 ]  \fstop
\end{align}
By Lemma \ref{lem1} and the Markov inequality, we obtain
\begin{equation}
\P_\eta^\eps \, [ \tau_{\Omega_k} \ge \delta/2 ] \le \frac2\delta \, \E_\eta^\eps \, [ \tau_{\Omega_k}] \le \frac2\delta \, \frac{\fs{C_1} \eps}{c_{\rm min}} \, \log k \xrightarrow{\eps \to 0} 0 \fstop
\end{equation}
Thus, by the strong Markov property of $(\eta^\eps(s))_{s\ge0}$ and the fact that $t \mapsto T_\eps(t)$ is non-decreasing, we deduce
\begin{equation}
\limsup_{\eps \to 0}  \P_\eta^\eps \, [ S_\eps(t) \ge t+\delta ] \le \limsup_{\eps \to 0} \sup_{\xi \in \Omega_k} \P_\xi^\eps \, [ T_\eps( t+ \delta/2 ) \le t] \fstop
\end{equation}
Thus, it suffices to prove Lemma \ref{lem2} for $\eta \in \Omega_k$.

It remains to prove that, for $\eta \in \Omega_k$, $t\ge0$ and $\delta>0$,
\begin{equation}\label{eq-lem2:WTS}
\lim_{\eps \to 0} \P_\eta^\eps \, [ T_\eps(t+\delta) \le t ] = 0 \fstop
\end{equation}
Let $0<\tau_1 < \tau_2 < \cdots$ be the consecutive jump times from $\Omega_k$ to $\Delta_k$. Then, for any  integer $m \ge 1$, we get
\begin{equation}\label{eq2-1}
\P_\eta^\eps \,[ T_\eps(t+\delta) \le t ]  \le  \P_\eta^\eps \,[ T_\eps(t+\delta) \le t , \ \tau_{m+1} > t+\delta ] + \P_\eta^\eps \,[ \tau_{m+1} \le t + \delta ] \fstop
\end{equation}
Conditioned on the first event in the right-hand side of \eqref{eq2-1}, there occurs at most $m$ jumps from $\Omega_k$ to $\Delta_k$ up to time $t+\delta$. Let $\tau_1' < \tau_2' < \cdots$ denote the consecutive return times to $\Omega_k$, such that we have
\begin{equation}
0 < \tau_1 < \tau_1'  < \tau_2 < \tau_2' < \cdots \fstop
\end{equation}
Then, by \eqref{Talpha-def}, we may bound the first probability in the right-hand side of \eqref{eq2-1} by
\begin{equation}
\P_\eta^\eps \, [ (\tau_1'-\tau_1) + (\tau_2' - \tau_2) + \cdots + (\tau_m' - \tau_m ) \ge \delta  ] \fstop
\end{equation}
We may further bound this by
\begin{equation}\label{eq2-2}
\begin{aligned}
\sum_{i=1}^m \P_\eta^\eps  [ \tau_i'-\tau_i \ge \delta/m ] & \le m \sup_{\zeta \in \Delta_k} \P_\zeta^\eps [ \tau_{\Omega_k} \ge \delta / m ] \\
& \le \frac{m^2}{\delta}  \sup_{\zeta \in \Delta_k}  \E_\zeta^\eps  [ \tau_{\Omega_k}] \le \frac{m^2  \fs{C_1} 	\eps  \log k}{\delta  c_{\rm min}} \xrightarrow{\eps \to 0 } 0 \fstop
\end{aligned}
\end{equation}
Here, the first inequality is due to the strong Markov property at $\tau_i$ for each $i \in \bbr{1}{m}$, the second inequality is due to the Markov inequality, and the third inequality is due to Lemma \ref{lem1}.
Further, for $(\eta^\eps(s))_{s\ge0}$, any direct jump from $ \Omega_k$ to $\Delta_k$ has rate bounded above by
\begin{equation}
\sup_{\eta\in \Omega_k}\sup_{x,y \in V} \eps^{-1} \,  c_{xy} \, \eta_x\eps\hat\alpha_y \le c_{\rm max} \, k \, \hat\alpha_{\rm max} \comma
\end{equation}
where $\hat\alpha_{\rm max} \eqdef \max_{x\in V} \hat\alpha_x$. Hence, each of the random variables $\tau_1 , \tau_2 - \tau_1' , \dots , \tau_{m+1} - \tau_m'$ stochastically dominates an exponential random variable with rate $ c_{\rm max} \, k \, \hat\alpha_{\rm max} $. Thus, letting $\tau_1^*,\tau_2^* , \dots$ be some i.i.d.\ exponentials of rate $ c_{\rm max} \, k \, \hat\alpha_{\rm max}$, by Markovianity we get	
\begin{equation}\label{eq2-3}
\begin{aligned}
\P_\eta^\eps \, [ \tau_{m+1} \le t+\delta ] & \le \P \, [ \tau_1^* + \cdots + \tau_{m+1}^* \le t + \delta  ] \\
& = \P \, [{\rm Gamma} \,(m+1, c_{\rm max} \, k \, \hat\alpha_{\rm max} ) \le t+\delta ]  \xrightarrow{m \to \infty} 0 \comma
\end{aligned}
\end{equation}
where ${\rm Gamma}\,(u,\beta)$ denotes the Gamma distribution with shape $u>0$ and rate $\beta>0$. Therefore, by \eqref{eq2-1}, \eqref{eq2-2} and \eqref{eq2-3}, we conclude that
\begin{align}
\limsup_{\eps \to 0} \P_\eta^\eps \, [T_\eps(t+\delta) \le t ] &\le   \limsup_{m \to \infty} \limsup_{\eps \to 0} \frac{m^2 \,\fs{C_1}\, \eps \, \log k}{ \delta \, c_{\rm min}} \\
& + \limsup_{m \to \infty} \P \,  [{\rm Gamma}\,(m+1, c_{\rm max} \, k \, \hat\alpha_{\rm max} ) \le t+\delta ]   = 0 \comma
\end{align}
as desired.
\end{proof}

\begin{lemma}\label{lem3}
We have
\begin{equation}
\lim_{\delta \to 0} \lim_{\eps \to 0} \sup_{\xi \in \Omega_k} \P_\xi^\eps \, [ \tau_{\Delta_k} < 3\delta  ] =  0 \fstop
\end{equation}
\end{lemma}
\begin{proof}
As explained in the proof of Lemma \ref{lem2}, starting from any $\xi \in \Omega_k$, $\tau_{\Delta_k}$ stochastically dominates an exponential random variable of rate $c_{\rm max}\, k\, \hat\alpha_{\rm max}$. Thus, we get
\begin{equation}
\limsup_{\delta \to 0} \sup_{\eps \in (0,1)}  \P_\xi^\eps \,[ \tau_{\Delta_k} < 3\delta  ]  \le \limsup_{\delta \to 0} \left(  1 - e^{- c_{\rm max} \,k\, \hat\alpha_{\rm max} \,3\delta}  \right) = 0 \fstop
\end{equation}
This concludes the proof.
\end{proof}

Now, we are ready to prove Proposition \ref{lem:fin-dim-conv}.
\begin{proof}[Proof of Proposition \ref{lem:fin-dim-conv}]
We prove
\begin{equation}\label{fdc-WTS}
\lim_{\eps \to 0} \P_\eta^\eps\, [ \eta^\eps(t) \in \Omega_k ] =1\comma\qquad t>0\comma\eta\in \Xi_k \fstop
\end{equation}
We follow the proof ideas of \cite[Lemma 3.1]{landim_loulakis_mourragui_2018}. By \eqref{Salpha-in-Ak}, for $\delta \in (0,t/3)$, we estimate
\begin{align}
1 = \P_\eta^\eps \, [ \eta^\eps(S_\eps(t-3\delta)) \in \Omega_k ] \le & \ \P_\eta^\eps \left[ \eta^\eps(S_\eps(t-3\delta)) \in \Omega_k, \ S_\eps(t-3\delta) \in [t-3\delta, t-2\delta]  \right] \\
& + \P_\eta^\eps \,[ S_\eps(t-3\delta) > t-2\delta ] \fstop
\end{align}
By Lemma \ref{lem2}, the second probability in the right-hand side vanishes as $\eps \to 0$. Moreover, by \eqref{Salpha-in-Ak}, the first probability in the right-hand side is bounded by
\begin{equation}
\P_\eta^\eps \,  [ \, \exists s \in [t-3\delta, t-2\delta] , \ \eta^\eps(s) \in \Omega_k  \, ] \fstop
\end{equation}
We can further divide this into two events and bound the last probability as
\begin{equation}
\P_\eta^\eps \,[ \eta^\eps(t) \in \Omega_k ] + \P_\eta^\eps \,  [ \, \exists s \in [t-3\delta, t-2\delta] , \ \eta^\eps(s) \in \Omega_k , \ \eta^\eps(t) \in \Delta_k  \, ] \fstop
\end{equation}
Then, by the Markov property, the second probability above is bounded by
\begin{equation}
\sup_{\xi \in \Omega_k} \P_\xi^\eps \, [ \tau_{\Delta_k} < 3\delta ] \comma
\end{equation}
which,  by Lemma \ref{lem3}, vanishes by first taking the limit $\eps\to 0$, and then $\delta\to 0$.
 Collecting all displayed inequalities, we conclude that \eqref{fdc-WTS} holds.
\end{proof}

\section{Reversibility of the limit process}\label{appenB}
In this appendix, we prove Lemma \ref{lem:nu-self-adj} and simply write $\varsigma_{\hat\alpha,k}=\varsigma_{\hat\alpha,k,m}$, $k\ge m$.
\begin{proof}[Proof of Lemma \ref{lem:nu-self-adj}]
We start by introducing an auxiliary continuous-time Markov chain on $\Xi_k$ with rate function  (cf.\ \eqref{Ak-Bk-dec}--\eqref{Deltak-def})
\begin{equation}\label{eq:aux-dynamics}
{\bf s}(\eta,\emparg)\eqdef \begin{dcases}
{\bf r}^{\cA}_{\hat\alpha,k}(\eta,\emparg) &\text{if}\ \eta\in \Omega_k\comma\\
{\bf r}^{\cB}_k(\eta,\emparg) &\text{if}\ \eta\in \Delta_k\fstop
\end{dcases}
\end{equation}
In words, the above dynamics follows the rule of $A_{\hat\alpha,k}$ in $\Omega_k$, while that  of $B_k$ in $\Delta_k$. ({Remark that this is not a ${\rm SIP}$-dynamics.})
Now, recall the definition of $\Omega_{k,m}$ from \eqref{Omegakm-def} and, similarly, introduce the set 
\begin{equation}
\Delta_{k,m} \eqdef \left\{ \zeta \in \Delta_k : |\{x\in V: \zeta_x > 0 \}| = m \right\} \fstop
\end{equation}
Observe that this auxiliary chain can jump from $\Omega_{k,m}$ to either $\Delta_{k,m}$ or $\Delta_{k,m+1}$ (if $m\le k-1$), but a jump from $\Omega_{k,m}$ to $\Delta_{k,m}$ has rate zero to backtrack.

Now, fix $m\le k-1$. We claim that this dynamics, {restricted to $\Omega_{k,m}\cup \Delta_{k,m+1}$}, is reversible with respect to $\varsigma_{\hat\alpha,k}$.
For this purpose, it suffices to consider pairs $\eta, \xi = \eta-\delta_x+\delta_y \in \Omega_{k,m}\cup \Delta_{k,m+1}$ (hence, $\eta_x\ge 1$); rates are zero otherwise. If $\eta \in \Omega_{k,m}$ and $\xi \in \Omega_{k,m}$ (such that $\eta_x =1$), then
\begin{equation}
\varsigma_{\hat\alpha,k}(\eta) \, {\bf r}^\cA_{\hat\alpha,k} (\eta,\xi)  =  c_{xy} \, \eta_x \hat\alpha_y \, \prod_{\substack{z\in V\\ \eta_z > 0}} \frac{\hat\alpha_z}{\eta_z}  = c_{xy} \, \hat\alpha_x \xi_y \,  \prod_{\substack{z\in V \\ \xi_z > 0}} \frac{\hat\alpha_z}{\xi_z}  = \varsigma_{\hat\alpha,k} (\xi) \, {\bf r}^\cA_{\hat\alpha,k} ( \xi,\eta) \fstop
\end{equation}
If $\eta \in \Omega_{k,m}$ and $\xi \in \Delta_{k,m+1}$ (such that $\eta_x \ge 2$), we have
\begin{equation}
\varsigma_{\hat\alpha,k}(\eta) \, {\bf r}^\cA_{\hat\alpha,k} (\eta,\xi)  =  c_{xy} \, \eta_x \hat\alpha_y \, \prod_{\substack{z\in V\\ \eta_z > 0}} \frac{\hat\alpha_z}{\eta_z}  = c_{xy} \, \xi_x \xi_y \,  \prod_{\substack{z\in V \\ \xi_z > 0}} \frac{\hat\alpha_z}{\xi_z}  = \varsigma_{\hat\alpha,k} (\xi) \, {\bf r}^\cB_k ( \xi,\eta) \fstop
\end{equation}
If $\eta,\xi \in \Delta_{k,m+1}$, we have
\begin{equation}
\varsigma_{\hat\alpha,k}(\eta) \, {\bf r}^\cB_k (\eta,\xi)  =  c_{xy} \, \eta_x \eta_y \, \prod_{\substack{z\in V \\ \eta_z > 0}} \frac{\hat\alpha_z}{\eta_z}  
= c_{xy} \, \xi_x \xi_y \, \prod_{\substack{z\in V \\ \xi_z > 0}} \frac{\hat\alpha_z}{\xi_z}  = \varsigma_{\hat\alpha,k} (\xi)\, {\bf r}^\cB_k ( \xi,\eta) \fstop
\end{equation}
This shows that $\varsigma_{\hat \alpha,k}$ is reversible for the Markov chain described in \eqref{eq:aux-dynamics}, restricted to $\Omega_{k,m}\cup\Delta_{k,m+1}$.
Furthermore, by \eqref{eq:rM-def},  we have, for all distinct $\eta,\xi \in \Omega_{k,m}$, 
\begin{equation}
\varsigma_{\hat\alpha,k}(\eta) \, {\bf r}^\cM_{\hat\alpha,k} (\eta,\xi)  = \varsigma_{\hat\alpha,k}(\eta)  \sum_{\zeta \in \Delta_{k,m+1}} {\bf r}^\cA_{\hat\alpha,k} (\eta,\zeta) \, {\bf P}_\zeta^{\cB} \, [\tau_\xi = \tau_{\Omega_k}]  \fstop
\end{equation}
Following $(\cB_k(t))_{t\ge0}$, from $\zeta\in\Delta_{k,m+1}$ to arrive at $\xi \in\Omega_{k,m}$, the trajectory must stay in $\Delta_{k,m+1}$. Thus, by reversibility, the right-hand side is equal to
\begin{equation}
\varsigma_{\hat\alpha,k}(\xi)  \sum_{\zeta \in \Delta_{k,m+1}} {\bf r}^\cA_{\hat\alpha,k} (\xi,\zeta)  \, {\bf P}_\zeta^{\cB} \, [\tau_\eta = \tau_{\Omega_k}] = \varsigma_{\hat\alpha,k}(\xi) \, {\bf r}^\cM_{\hat\alpha,k} (\xi,\eta)  \fstop
\end{equation}
This concludes the proof of the lemma.
\end{proof}

\subsection*{Notation guide}
\fs{
For the reader's convenience, we collect some recurring notation. The table is only meant
as a quick reference; all objects are defined in detail at the indicated locations.}

\begin{center}
	\scriptsize
	\setlength{\tabcolsep}{3pt}
	\renewcommand{\arraystretch}{1.18}
	\begin{tabular}{|
			>{\centering\arraybackslash}m{0.18\textwidth}|
			>{\raggedright\arraybackslash}m{0.68\textwidth}|
			>{\centering\arraybackslash}m{0.09\textwidth}|}
		\hline
		\textbf{Symbol} & \textbf{Brief description} & \textbf{Ref.} \\
		\hline
		$G=(V,(c_{xy})_{x,y\in V})$ &
		finite connected weighted graph; $c_{xy}=c_{yx}\ge0$ conductance of edge $xy$ &
		Sec.~\ref{sec2} \\
		\hline
		$\alpha$ &
		site weights; $\alpha_{\min}:=\min_x\alpha_x$, $\alpha_{\max}:=\max_x\alpha_x$,
		$\alpha_{\rm ratio}:=\alpha_{\min}/\alpha_{\max}$ &
		\eqref{alpha-c-def} \\
		\hline
		$\Xi_k$ &
		$k$-particle configurations; $\Xi_k=\{\eta\in\mathbb N_0^V:|\eta|=k\}$ &
		Sec.~\ref{sec2} \\
		\hline
		$L_{G,\alpha,k}$ &
		generator of the conservative process ${\rm SIP}_k(G,\alpha)$ on $\Xi_k$ &
		\eqref{eq:gen-conservative} \\
		\hline
		$\mu_{\alpha,k}$ &
		reversible Dirichlet-Multinomial measure of ${\rm SIP}_k(G,\alpha)$ &
		\eqref{mu-def} \\
		\hline
		$\mathcal E_{G,\alpha,k}$ &
		Dirichlet form of $L_{G,\alpha,k}$ in $L^2(\mu_{\alpha,k})$ &
		\eqref{eq:dir-form-SIP} \\
		\hline
		${\rm gap}_k(G,\alpha)$ &
		spectral gap of the $k$-particle conservative process ${\rm SIP}_k(G,\alpha)$ &
		Sec.~\ref{sec:spectral-gap} \\
		\hline
		${\rm gap}_{\rm RW}(G,\alpha)$ &
		one-particle gap; ${\rm gap}_{\rm RW}(G,\alpha)={\rm gap}_1(G,\alpha)$ &
		\eqref{eq:gap-RW} \\
		\hline
		${\rm gap}_{\rm SIP}(G,\alpha)$ &
		conservative interacting gap; ${\rm gap}_{\rm SIP}(G,\alpha):=\inf_{k\ge2}{\rm gap}_k(G,\alpha)$ &
		\eqref{eq:gap-sip-def} \\
		\hline
		$\mathfrak a_k$ &
		annihilation operator; $(k-1)$-particle observables $\longrightarrow$ $k$-particle observables &
		\eqref{ann-op-def} \\
		\hline
		$\mathfrak a^\dagger_{\alpha,k-1}$ &
		creation operator; adjoint of $\mathfrak a_k$ in the corresponding $L^2(\mu_{\alpha,k})$ spaces &
		\eqref{eq:cre-def} \\
		\hline
		$\alpha=\varepsilon\hat\alpha$ &
		small-diffusivity scaling; fixed direction $\hat\alpha$, parameter $\varepsilon\downarrow0$ &
		\eqref{eq:alpha-eps-beta} \\
		\hline
		$A_{\hat\alpha,k}$ &
		slow part of $\varepsilon^{-1}L_{G,\varepsilon\hat\alpha,k}$; diffusive moves of particles &
		\eqref{Ak-Bk-dec} \\
		\hline
		$B_k$ &
		fast part of $\varepsilon^{-1}L_{G,\varepsilon\hat\alpha,k}$; attractive interaction/coalescence mechanism &
		\eqref{Ak-Bk-dec} \\
		\hline
		$\Omega_k$ &
		absorbing set of $B_k$; no two occupied sites are neighboring &
		\eqref{Omegak-def} \\
		\hline
		$\Delta_k$ &
		bad/transient complement; $\Delta_k:=\Xi_k\setminus\Omega_k$ &
		\eqref{Deltak-def} \\
		\hline
		$\Pi_k$ &
		$B_k$-harmonic projection; $\Pi_k f=\lim_{t\to\infty}e^{tB_k}f$ &
		Sec.~\ref{sec5.1} \\
		\hline
		${\mathscr G}_{\hat\alpha,k}$ &
		effective slow--fast generator $\Pi_k A_{\hat\alpha,k}$ on $\cR(\Pi_k)\subset \R^{\Xi_k}$
		&
		Prop.~\ref{lem:kurtz} \\
		\hline
		$w_k(\hat\alpha)$ &
		spectral gap of the limiting metastable dynamics generated by ${\mathscr G}_{\hat\alpha,k}$ &
		\eqref{wk-def} \\
		\hline
		$M_{\hat\alpha,k}$ &
		rate matrix on $\Omega_k$ associated with ${\mathscr G}_{\hat\alpha,k}: \cR(\Pi_k)\to \cR(\Pi_k)$ &
		\eqref{eq:M} \\
		\hline
		$r^M_{\hat\alpha,k}$ &
		jump rates of the effective chain $M_{\hat\alpha,k}$ on $\Omega_k$ &
		\eqref{eq:rM-def} \\
		\hline
		$\Omega_{k,m}$ &
		states in $\Omega_k$ with exactly $m$ separated stacks &
		\eqref{Omegakm-def} \\
		\hline
		$M_{\hat\alpha,k,m}$ &
		restriction of $M_{\hat\alpha,k}$ to $\Omega_{k,m}$; transient block for $m\ge2$ &
		\eqref{eq:Mkm} \\
		\hline
		$\lambda_{k,m}(\hat\alpha)$ &
		principal eigenvalue of $-M_{\hat\alpha,k,m}$ &
		\eqref{eq:lambda-k-m} \\
		\hline
		$\Xi$ &
		full non-conservative space; $\Xi=\mathbb N_0^V=\cup_{k\ge0}\Xi_k$ &
		Sec.~\ref{sec:open-SIP} \\
		\hline
		$\omega,\theta$ &
		reservoir parameters; $\omega_x \leftrightarrow$ interaction rate w/ res.\ at $x$;
		$\theta_x \leftrightarrow$ density res.\ at $x$ &
		Sec.~\ref{sec:open-SIP} \\
		\hline
		$L_{G,\alpha,\omega,\theta}$ &
		generator of the non-conservative/open  ${\rm SIP}(G,\alpha,\omega,\theta)$ &
		\eqref{eq:gen-res} \\
		\hline
		$\nu_{\alpha,\varrho}$ &
		product reversible measure for the open SIP when $\theta\equiv\varrho>0$ &
		\eqref{eq:nu-sigma} \\
		\hline
		$L_{G,\alpha,\omega,k}$ &
		generator of the $k$-particle absorbing system; case $\theta\equiv0$ &
		\eqref{eq:gen-res-abs} \\
		\hline
		${\rm gap}_k(G,\alpha,\omega)$ &
		lowest eigenvalue of $-L_{G,\alpha,\omega,k}$ for the $k$-particle absorbing system &
		Sec.~\ref{sec:non-conservative-stat-measures} \\
		\hline
		${\rm gap}_{\rm RW}(G,\alpha,\omega)$ &
		one-particle killed gap; killing rate $\omega_x$ at site $x$ &
		\eqref{eq:gap-rw-abs} \\
		\hline
		${\rm gap}_{\rm SIP}(G,\alpha,\omega)$ &
		absorbing interacting gap; ${\rm gap}_{\rm SIP}(G,\alpha,\omega):=\inf_{k\ge2}{\rm gap}_k(G,\alpha,\omega)$ &
		\eqref{eq:gap-sip-gap-k-non-conservative} \\
		\hline
			${\rm gap}_{\rm SIP}(G,\alpha,\omega,\varrho)$ &
		first $L^2(\nu_{\alpha,\varrho})$ gap of the reversible open SIP; case $\theta\equiv\varrho>0$ &
		Cor.~\ref{cor:non-conservative} \\
		\hline
		$D_{\alpha,\varrho}$ &
		orth.\ poly.\  duality functions; finite-particle absorbing systems $\longrightarrow$ open system &
		\eqref{eq:duality} \\
		\hline
		$F^\psi_{\alpha,\varrho}$ &
		polynomial lift of $\psi$ from $\Xi_k$ to $\Xi$ via $D_{\alpha,\varrho}$ &
		\eqref{eq:f-lift} \\
		\hline
		$\mathfrak b^\theta_{\alpha,\omega,\varrho,k-1}$ &
		lowering operator in the lifting relation; equals $0$ when $\theta\equiv\varrho$ &
		\eqref{eq:mathfrak-b} \\
		\hline
	\end{tabular}
\end{center}
\begin{acknowledgement}
		The authors wish to thank Pietro Caputo,   Matteo Quattropani, Justin Salez and Gunter Sch\"utz,  for fruitful conversations on various aspects of this work. \fs{The authors 	are grateful to the two anonymous reviewers for the careful reading of our manuscript and for their detailed comments, which helped us improve the presentation and clarify several points.} 
	The authors thank MIGE (Trieste), LMRS (Rouen) and IMPA (Rio de Janeiro) for the warm hospitality during their stays. 
\sk{SK was supported by the Basic Science Research Program through the National Research Foundation of Korea funded by the Ministry of Science and ICT (RS-2025-00518980), the Yonsei University Research Fund of 2025 (2025-22-0133), and the POSCO Science Fellowship of POSCO TJ Park Foundation.}
FS was supported by Microgrants 2022 (Regione FVG, legge LR 2/2011), and was a member of GNAMPA, INdAM, and of the  PRIN project  TESEO (2022HSSYPN). 
\end{acknowledgement}

\begin{data} Data sharing not applicable to this article as no datasets were generated or analyzed during the current study.
\end{data}

\begin{conflicts}
	All authors declare that they have no conflicts of interest.
\end{conflicts}


\end{document}